\documentclass[11pt]{scrartcl}
\usepackage[utf8]{inputenc}
\usepackage[T1]{fontenc}
\usepackage{lmodern}
\usepackage{alphabeta}

\usepackage{titling}
\usepackage{xspace}
\usepackage{soul} % for \st to strike out text

\usepackage[letterpaper, top=25.4mm, bottom=25.4mm, left=25.4mm, right=25.4mm, includefoot]{geometry}

\DeclareSectionCommand[%
level=4,
indent=0pt,
beforeskip=1ex plus 1ex minus .2ex,
afterskip=-1em,
font={},
tocindent=7em,
tocnumwidth=4.1em,
counterwithin=subsubsection
]{paragraph}

\usepackage[inline]{enumitem}

\usepackage{dsfont}
\usepackage{setspace}

\usepackage{bm}
\usepackage{microtype}
\usepackage{amsmath}
\usepackage{amssymb}
\usepackage{amsfonts}
\usepackage{mathtools}
\usepackage[mathscr]{euscript}

\usepackage[hyphens]{url}
\usepackage{amsthm}

\usepackage{tikz}
\usepackage{xcolor}
\usepackage{subcaption}
\usepackage{graphicx}
\usepackage{wrapfig}

\usepackage{hyperref}
\usepackage{zref-clever}

\usepackage{nicefrac}
\usepackage[textwidth=1.5in]{todonotes}

\colorlet{myGreen}{green!50!black}
\colorlet{myLightgreen}{green}
\colorlet{myRed}{red!90!black}
\definecolor{myBlue}{rgb}{0.25, 0.0, 1.0}
\definecolor{myLightBlue}{rgb}{0.39, 0.58, 0.93}
\colorlet{myViolet}{myBlue!55!myRed}
\definecolor{myOrange}{rgb}{1.0, 0.66, 0.07}

\definecolor{CornflowerBlue}{rgb}{0.39, 0.58, 0.93}
\definecolor{DarkGoldenrod}{rgb}{0.72, 0.53, 0.04}
\definecolor{BritishRacingGreen}{rgb}{0.0, 0.26, 0.15}
\definecolor{DarkMagenta}{rgb}{0.55, 0.0, 0.55}
\definecolor{AO}{rgb}{0.0, 0.5, 0.0}
\definecolor{BostonUniversityRed}{rgb}{0.8, 0.0, 0.0}
\definecolor{myRed}{rgb}{0.8, 0.0, 0.0}
\definecolor{DarkMidnightBlue}{rgb}{0.0, 0.2, 0.4}
\definecolor{DarkTangerine}{rgb}{1.0, 0.66, 0.07}
\definecolor{AppleGreen}{rgb}{0.55, 0.71, 0.0}
\definecolor{BrightUbe}{rgb}{0.82, 0.62, 0.91}
\definecolor{Amethyst}{rgb}{0.6, 0.4, 0.8}
\definecolor{DarkGray}{rgb}{0.52, 0.52, 0.51}
\definecolor{Gray}{rgb}{0.66, 0.66, 0.66}
\definecolor{BananaYellow}{rgb}{1.0, 0.88, 0.21}
\definecolor{Amber}{rgb}{1.0, 0.75, 0.0}
\definecolor{LightGray}{rgb}{0.83, 0.83, 0.83}
\definecolor{PrincetonOrange}{rgb}{1.0, 0.56, 0.0}
\definecolor{DeepCarrotOrange}{rgb}{0.91, 0.41, 0.17}
\definecolor{CarrotOrange}{rgb}{0.93, 0.57, 0.13}
\definecolor{MidnightBlue}{rgb}{0.1, 0.1, 0.44}
\definecolor{Magenta}{rgb}{0.50, 0.0, 0.50}
\definecolor{BrightPink}{rgb}{1.0, 0.0, 0.5}
\definecolor{BrilliantRose}{rgb}{1.0, 0.33, 0.64}
\definecolor{ChromeYellow}{rgb}{1.0, 0.65, 0.0}
\definecolor{HotMagenta}{rgb}{1.0, 0.11, 0.81}
\definecolor{Auburn}{rgb}{0.43, 0.21, 0.1}
\definecolor{BrightTurquoise}{rgb}{0.03, 0.91, 0.87}
\definecolor{DarkCyan}{rgb}{0.0, 0.55, 0.55}
\definecolor{Flame}{rgb}{0.89, 0.35, 0.13}

\vfuzz \hfuzz
\setlist[itemize]{topsep=0pt,partopsep=0pt,itemsep=0pt,parsep=0pt}
\setlist[itemize,1]{label={\small\textbullet}}
\setlist[itemize,2]{label={\tiny\textbullet}}
\setlist[itemize,3]{label=$\cdot$}
\setlist[enumerate]{topsep=0pt,partopsep=0pt,itemsep=0pt,parsep=0pt}
\setlist[enumerate,1]{label=\roman*)}
\setlist[enumerate,2]{label=\alph*)}
\setlist[enumerate,3]{label=\arabic*)}

\hypersetup{
colorlinks=true,
linkcolor=AO!65!black,
citecolor=AO!65!black,
urlcolor=AppleGreen!65!black,
bookmarksopen=true,
bookmarksnumbered,
bookmarksopenlevel=2,
bookmarksdepth=3
}

\theoremstyle{definition}

\newtheorem{environment}{Environment}[section]

\newtheorem{lemma}[environment]{Lemma}
\AddToHook{env/lemma/begin}{\zcsetup{countertype={environment=lemma}}}
\zcRefTypeSetup{lemma}{
Name-sg = Lemma ,
name-sg = Lemma ,
Name-pl = Lemmas ,
name-pl = Lemmas ,
}

\newtheorem*{lemma*}{Lemma}
\AddToHook{env/lemma*/begin}{\zcsetup{countertype={environment=lemma*}}}
\zcRefTypeSetup{lemma*}{
Name-sg = Lemma ,
name-sg = Lemma ,
Name-pl = Lemmas ,
name-pl = Lemmas ,
}

\newtheorem{proposition}[environment]{Proposition}
\AddToHook{env/proposition/begin}{\zcsetup{countertype={environment=proposition}}}
\zcRefTypeSetup{proposition}{
Name-sg = Proposition ,
name-sg = Proposition ,
Name-pl = Propositions ,
name-pl = Propositions ,
}

\newtheorem{corollary}[environment]{Corollary}
\AddToHook{env/corollary/begin}{\zcsetup{countertype={environment=corollary}}}
\zcRefTypeSetup{corollary}{
Name-sg = Corollary ,
name-sg = Corollary ,
Name-pl = Corollaries ,
name-pl = Corollaries ,
}

\newtheorem{theorem}[environment]{Theorem}
\AddToHook{env/theorem/begin}{\zcsetup{countertype={environment=theorem}}}
\zcRefTypeSetup{theorem}{
Name-sg = Theorem ,
name-sg = Theorem ,
Name-pl = Theorems ,
name-pl = Theorems ,
}

\newtheorem*{theorem*}{Theorem}
\AddToHook{env/theorem*/begin}{\zcsetup{countertype={environment=theorem*}}}
\zcRefTypeSetup{theorem*}{
Name-sg = Theorem ,
name-sg = Theorem ,
Name-pl = Theorems ,
name-pl = Theorems ,
}

\AddToHook{env/conjecture/begin}{\zcsetup{countertype={environment=conjecture}}}
\zcRefTypeSetup{conjecture}{
Name-sg = Conjecture ,
name-sg = Conjecture ,
Name-pl = Conjectures ,
name-pl = Conjectures ,
}

\newtheorem*{hypothesis*}{Hypothesis}
\AddToHook{env/hypothesis*/begin}{\zcsetup{countertype={environment=hypothesis*}}}
\zcRefTypeSetup{hypothesis*}{
Name-sg = Hypothesis ,
name-sg = Hypothesis ,
Name-pl = Hypotheses ,
name-pl = Hypotheses ,
}

\newtheorem{observation}[environment]{Observation}
\AddToHook{env/observation/begin}{\zcsetup{countertype={environment=observation}}}
\zcRefTypeSetup{observation}{
Name-sg = Observation ,
name-sg = Observation ,
Name-pl = Observations ,
name-pl = Observations ,
}

\AddToHook{env/example/begin}{\zcsetup{countertype={environment=example}}}
\zcRefTypeSetup{example}{
Name-sg = Example ,
name-sg = Example ,
Name-pl = Examples ,
name-pl = Examples ,
}

\AddToHook{env/remark/begin}{\zcsetup{countertype={environment=remark}}}
\zcRefTypeSetup{remark}{
Name-sg = Remark ,
name-sg = Remark ,
Name-pl = Remarks ,
name-pl = Remarks ,
}

\zcRefTypeSetup{equation}{
Name-sg = Equation ,
name-sg = Equation ,
Name-pl = Equations ,
name-pl = Equations ,
}

\zcRefTypeSetup{chapter}{
Name-sg = Chapter ,
name-sg = Chapter ,
Name-pl = Chapters ,
name-pl = Chapters ,
}

\zcRefTypeSetup{section}{
Name-sg = Section ,
name-sg = Section ,
Name-pl = Sections ,
name-pl = Sections ,
}

\zcRefTypeSetup{algorithm}{
Name-sg = Algorithm ,
name-sg = Algorithm ,
Name-pl = Algorithms ,
name-pl = Algorithms ,
}

\AddToHook{env/notation/begin}{\zcsetup{countertype={environment=notation}}}
\zcRefTypeSetup{notation}{
Name-sg = Notation ,
name-sg = Notation ,
Name-pl = Notations ,
name-pl = Notations ,
}

\AddToHook{env/question/begin}{\zcsetup{countertype={environment=question}}}
\zcRefTypeSetup{question}{
Name-sg = Question ,
name-sg = Question ,
Name-pl = Questions ,
name-pl = Questions ,
}

\AddToHook{env/problem/begin}{\zcsetup{countertype={environment=problem}}}
\zcRefTypeSetup{problem}{
Name-sg = Problem ,
name-sg = Problem ,
Name-pl = Problems ,
name-pl = Problems ,
}

\AddToHook{env/claim/begin}{\zcsetup{countertype={environment=claim}}}
\zcRefTypeSetup{claim}{
Name-sg = Claim ,
name-sg = Claim ,
Name-pl = Claims ,
name-pl = Claims ,
}

\AddToHook{env/definition/begin}{\zcsetup{countertype={environment=definition}}}
\zcRefTypeSetup{definition}{
Name-sg = Definition ,
name-sg = Definition ,
Name-pl = Definitions ,
name-pl = Definitions ,
}

\zcRefTypeSetup{figure}{
Name-sg = Figure ,
name-sg = Figure ,
Name-pl = Figures ,
name-pl = Figures ,
}

\usetikzlibrary{calc}
\usetikzlibrary{fit}
\usetikzlibrary{decorations}
\usetikzlibrary{decorations.pathmorphing}
\usetikzlibrary{decorations.text}
\usetikzlibrary{shapes,hobby}

\tikzset{
	position/.style args={#1:#2 from #3}{
		at=($(#3)+(#1:#2)$)
	}
}

\tikzset{
  v:main/.style = {draw, circle, scale=0.8, thick,fill=black,inner sep=0.7mm},
  v:ghost/.style = {inner sep=0pt,scale=1},
  >={latex},
  e:marker/.style = {line width=8.5pt,line cap=round,opacity=0.35,color=DarkGoldenrod},
  e:main/.style = {line width=1pt},
}

\newcommand*\samethanks[1][\value{footnote}]{\footnotemark[#1]}

\title{The price of homogeneity is polynomial}
\predate{}
\date{}
\postdate{}

\preauthor{}
\DeclareRobustCommand{\authorthing}{
	\begin{center}
		Maximilian Gorsky\thanks{Supported by the Institute for Basic Science (IBS-R029-C1).}~~\!\footnote{\href{mailto:m.gorsky@pm.me}{m.gorsky@pm.me}} \\
		{\small Discrete Mathematics Group, Institute for Basic Science (IBS), South Korea} \\
        \medskip
		Micha\l{} T.\ Seweryn\footnote{\href{mailto:michalsew@gmail.com}{michalsew@gmail.com}} \\
		  \medskip
		Sebastian Wiederrecht\samethanks[1]~~\!\footnote{\href{mailto:sebastian.wiederrecht@gmail.com}{wiederrecht@kaist.ac.kr}} \\
		{\small School of Computing, KAIST, South Korea}
\end{center}}
\author{\authorthing}
\postauthor{}

\begin{document}
\maketitle

\begin{abstract}
We provide explicit and polynomial bounds for the \textsl{Homogeneous Wall Lemma} which occurred for the first time implicitly in the $13$th entry of Robertson and Seymour's Graph Minors Series [JCTB 1990] and has since become a cornerstone in the algorithmic theory of graph minors.

A wall where each brick is assigned a set of colours is said to be \emph{homogeneous} if each brick is assigned the same set of colours.
The Homogeneous Wall Lemma says that there exists a function $h$ that, given non-negative integers $q$ and $k$ and an $h(q,k)$-wall $W$ where each brick is assigned a, possibly empty, subset of $\{ 1,\dots,q\}$ contains a $k$-wall $W'$ as a subgraph such that, if one assigns to each brick $B$ of $W'$ the union of the sets assigned to the bricks of $W$ in its interior, then $W'$ is homogeneous.
It is well-known that $h(q,k) \in k^{\mathcal{O}(q)}$.
The Homogeneous Wall Lemma plays a key role in most applications of the \textsl{Irrelevant Vertex Technique} where an exponential dependency of $h$ on $q$ usually causes non-uniform dependencies on meta-parameters at best and additional exponential blow-ups at worst.
By proving that $h(q,k) \in \mathcal{O}(q^4\cdot k^6)$, we provide a positive answer to a problem raised by Sau, Stamoulis, and Thilikos [ICALP 2020].
\end{abstract}
\let\sc\itshape
\thispagestyle{empty}

\newpage

\setcounter{page}{1}

\section{Introduction}\label{sec:Intro}
A central theme in structural graph theory is that graphs in which some specific parameter exceeds a certain threshold must contain large, well-organised substructures.
A prime example of such a behaviour is the celebrated \textsl{Grid Theorem} due to Robertson and Seymour \cite{RobertsonS1986Grapha}.
This theorem says that every graph of large \textsl{treewidth} contains a large \textsl{wall}\footnote{We postpone the definition of a wall to \zcref{subsec:Walls}.} as a subgraph.
The Grid Theorem plays a pivotal role in algorithmic and structural graph theory as a whole \cite{DemaineFHT2005Subexponential,DemaineHT2006Bidimensional,FominT2006Dominating,DemaineH2008Linearity,FominLST2020Bidimensionality}.

Recently a lot of effort has been dedicated to refining Robertson and Seymour's theory of graph minors, including a push to find new proofs for their structural results to give them explicit, polynomial bounds.
In this paper we continue this program by providing a \textsl{Homogeneous Wall Lemma} with polynomial bounds.
This lemma plays an integral role in most applications of the \textsl{Irrelevant Vertex Technique} \cite{RobertsonS1995Graph,RobertsonS2009Graph,RobertsonS2012Graph}.
Previously known exponential bounds have been identified as the sole reason for non-uniform behaviour in the running times of many parameterized algorithms based on this technique \cite{SauST2020An,MorelleSST2023Faster}.

As a side effect of the above effort, many structural key results may now be phrased to require a large wall as input \cite{KawarabayashiTW2018New,KawarabayashiTW2021Quickly,PaulPTW2024Obstructionsa,GorskySW2025Polynomial}.
Regarding algorithmic applications, one of those theorems stands out among the rest: the \textsl{Flat Wall Theorem} \cite{RobertsonS1995Graph,Chuzhoy2015Improved,KawarabayashiTW2018New}.
This theorem is the base of Robertson and Seymour's \textsl{Graph Minor Algorithm} and the related \textsl{Irrelevant Vertex Technique}.
Below is a summary of these two cornerstones.
\begin{description}
 \item[Flat Wall Theorem:] Any graph $G$ with a large wall $W$ contains either a $K_t$-minor or a small set of vertices $A\subseteq V(G)$ and a big wall $W'\subseteq W$ such that the subgraph of $G-A$ attaching to the ``interior'' of $W'$ behaves roughly like a planar graph\footnote{This definition is kept intentionally vague. See \zcref{sec:HomoMuralis} for a full definition.} -- this behaviour is called \emph{flatness}.
 \item[Irrelevant Vertex Technique:] For each instance\footnote{The \textsc{$k$-Disjoint Paths} problem takes as input a pair $(G,\Pi)$ where $G$ is a graph and $\Pi=\langle (s_i,t_i) \rangle_{i\in\{ 1,\dots,k\}}$ is a sequence of $k$ vertex pairs from $G$.
The question is if there exist internally vertex-disjoint paths $P_1,\dots,P_k$ such that $P_i$ has endpoints $s_i$ and $t_i$ for each $i\in\{ 1,\dots,k\}$.} $(G,\Pi)$ of the \textsc{$k$-Disjoint Paths} problem containing a large (relative to $k$) wall $W$, there exists a vertex $v$  -- called the \emph{irrelevant vertex} -- in the ``interior'' of $W$ such that $(G-v,\Pi)$ and $(G,\Pi)$ are equivalent instances.
\end{description}
The Irrelevant Vertex Technique combined with the Grid Theorem now allows for the following high level algorithm for the \textsc{$k$-Disjoint Paths} problem:
Given an instance $(G,\Pi)$ of the \textsc{$k$-Disjoint Paths} problem, either $G$ has small treewidth, in which case one can solve the problem using dynamic programming, or $G$ has large treewidth and therefore contains an irrelevant vertex whose deletion yields a smaller instance.
Hence, after at most $|G|$ iterations the first outcome must occur.
This technique has since found a wide range of applications in the design of parameterized algorithms \cite{FominLPSZ2020Hitting,SauST2020An,SauST2022kApices2,SauST2023kApices,GolovachST2023ModelChecking,MorelleSST2023Faster,CavallaroKK2024EdgeDisjointa,KorhonenPS2024Minora,SchirrmacherSSTV2024ModelChecking,SauST2025Parameterizing}, most of which deal with entirely different problems than finding paths between specified endpoints.

A key contributor for the success of the Irrelevant Vertex Technique is the comparative ease of access as a black box.
Proving the correctness of the original result required the full power of the Graph Minor Structure Theorem (see \cite{RobertsonS2003Grapha,GorskySW2025Polynomial} and \cite{RobertsonS2009Graph,RobertsonS2012Graph,CavallaroKK2024EdgeDisjointa,LiuY2025Disjoint}).
However, even with the recent success in making large parts of Robertson and Seymour's Graph Minors Series efficient, the Irrelevant Vertex Technique remains the main contributor to the ``galactic'' behaviour of the constants and functions involved in the Graph Minor Algorithm and its relatives.

\paragraph{The price of homogeneity.}
Once the correctness of the Irrelevant Vertex Technique is established and one argues away the presence of a large clique minor using established techniques (see \cite{RobertsonS1995Graph,ProtopapasTW2025Colorful}), one only requires a single refinement step after applying the Flat Wall Theorem to actually find an irrelevant vertex.
This refinement step is referred to as \emph{homogenisation}.
Recall that the final flat wall $W'$ behaves like a planar graph \textsl{after} deleting a small vertex set $A$.
The vertices in $A$ however may arbitrarily attach to the interior of the wall.
In order to argue for some vertex to be irrelevant, one needs to first tame these attachments.
In this simplified setting we say that a flat wall $W$ is \emph{homogeneous} for the set $A$ if for all non-boundary facial cycles $C$ of $W$, also called \textsl{bricks}, and every $a \in A$ it holds that $a$ has a neighbour in the interior of $C$ if and only if each member of $A$ has a neighbour in the interior of the face bounded by $C$.
If one aims for a homogeneous flat wall of order $k$, current techniques\footnote{For an example of a sophisticated proof carrying out such a technique, see the proof of Lemma 13 in \cite{SauST2023kApices}.} require to start with a flat wall of order $k^{\mathcal{O}(|A|)}$.
\smallskip

This exponential blow-up on the size of the required flat wall was dubbed the ``\emph{price of homogeneity}'' by Sau, Stamoulis, and Thilikos \cite{SauST2020An,SauST2022kApices2,SauST2023kApices}.
The exponential dependency on $|A|$, and thus the exponential nature of the price of homogeneity was believed to be unavoidable when relying on the Irrelevant Vertex Technique \cite{SauST2022kApices2,SauST2023kApices,MorelleSST2023Faster}.
Indeed, it appears that this is truly the limit of standard technique for homogenisation and in \cite{SauST2022kApices2}, it is asked whether one can prove that this price is unavoidable if one wants to apply the Irrelevant Vertex Technique.

\paragraph{Our result.}
For many algorithmic applications, the exponential bounds established as the state of the art force an additional exponentiation or a non-uniform dependency on a meta-parameter.
We discuss the latter phenomenon in more detail in \zcref{subsec:consequences} based on an example from a recent result of Morelle, Sau, Stamoulis, and Thilikos \cite{MorelleSST2023Faster}.
\smallskip

Our main result improves the bound of $k^{\mathcal{O}(|A|)}$ to $\mathcal{O}(|A|^4 \cdot k^6)$, thereby removing any additional exponential dependency or resulting non-uniformity in the running time of all algorithms relying on the Irrelevant Vertex Technique based on homogenisation.
\smallskip

To make our results adaptable to a variety of settings, we state them in a more abstract form.
Thus, we in particular provide a direct answer to the problem raised by Sau et al.\ \cite{SauST2022kApices2}, by making this part of the Irrelevant Vertex Technique more efficient.
This shows that the limitations of this method are still far from completely understood, suggesting potential for further improvements.

\subsection{A polynomial homogeneity lemma.}
We now provide a slightly more technical version of our main result.

Let $W$ be a $k$-wall with $k \geq 3$.
The \emph{perimeter} $C_P$ of $W$ is the facial cycle corresponding to the outer face of the wall and a \emph{brick} of $W$ is any facial cycle that is not the perimeter. 

Let $H$ be a subgraph of a connected graph $G$.
An \emph{$H$-bridge} in $G$ is a connected subgraph $J$ of $G$ such that $E(J) \cap E(H) = \emptyset$ and either $E(J)$ consists of a unique edge with both ends in $H$, or 
$J$ is constructed from a component $C$ of $G - V(H)$ and the non-empty set of edges $F \subseteq E(G)$ with one end in $V(C)$ and the other in $V(H)$, by taking the union of $C$, the endpoints of the edges in $F$, and $F$ itself.
The $H$-bridges induce a partition of $E(G)\setminus E(H)$.
The vertices in $V(J) \cap V(H)$ are called the \emph{attachments} of $J$ and the set $V(J)\setminus V(H)$ is called the \emph{interior} of $J$.

We define the \emph{compass} of $W$ to be the union of $C_P$ and the unique $C_P$-bridge in $G$ that contains all vertices of $W - V(C_P)$.
Let $C \subseteq W$ be any cycle of $W$.
The \emph{compass} of $C$ is the union of $C$ together with all $C$-bridges in $G$ that are entirely contained in the compass of $W$.
The \emph{interior} of $C$ is the the compass of $C$ with $V(C)$ deleted.

Let $q \geq 0$ be an integer.
A \emph{$q$-colorful graph} is a pair $(G,\chi)$ where $G$ is a graph and $\chi\colon V(G)\to 2^{[q]}$ is a map, assigning to each vertex $v\in V(G)$ a, possibly empty, subset of the \emph{colours} $\chi(v) \subseteq [q]$.

A flat wall $W$ in a $q$-colorful graph $(G,\chi)$ is called \emph{homogeneous} if there exists a bipartition $I \dot{\cup} O = [q]$ of the colours such that no vertex in the compass of $W$ carries a colour from $O$ and for every brick $B$ of $W$ and every $i \in I$ the interior of $B$ has a vertex with colour $i$.

The main difference between our main result and most variants of the Homogeneous Wall Lemma is that we do not guarantee the homogeneous wall $W'$ to be a \textsl{subwall} of the initial flat wall $W$.
Instead, in our result $W'$ is merely a subgraph of $W$ with some additional properties:
First, we require that the ``witness''\footnote{For us a witness for flatness is a form of ``almost embedding'' describing the planar-like behaviour of the flat wall. See \zcref{sec:HomoMuralis} for a definition.} for the flatness of $W$ also witnesses the flatness of $W'$.
Second, the \textsl{tangle}\footnote{See \zcref{sec:wallsAndTangles} for a definition.} of $W'$ should be a subtangle -- a so-called \textsl{truncation} -- of the tangle of $W$.
Tangles originate in Robertson and Seymour's Graph Minors Series \cite{RobertsonS1991Graph} and provide an abstract framework to identify ``areas'' in a graph which cannot be split apart by tree-decompositions of small adhesion.

For a graph $G$, we denote by $|G|$ the number of its vertices and by $||G||$ the number of its edges.

\begin{theorem}\label{thm:IntroHomoMuralis}
There exists a function $f\colon \mathbb{N}^2\to\mathbb{N}$ with $f(q,k) \in \mathcal{O}(q^4 \cdot k^6)$ such that for all non-negative integers $q,k$ and every $q$-colorful graph $G$ with a flat $f(q,k)$-wall $W_0$, there exists a flat, homogeneous $k$-wall $W_1\subseteq W_0$ whose tangle is a truncation of the tangle of $W_0$.

Moreover, there exists an algorithm that takes as input a colorful graph $(G,\chi)$ and a flat wall $W_0$ and computes the flat wall $W_1$ as above in time $\mathsf{poly}(q+k)||G||$.
\end{theorem}

While shifting from subwalls to subgraphs might seem to be a weakening of the original Homogeneous Wall Lemma, by ensuring that both the original and the new wall share the same witness for flatness and agree on their tangles, our main theorem maintains the core properties also guaranteed by subwalls.
This ensures the applicability of \zcref{thm:IntroHomoMuralis} whenever previous versions of the Homogeneous Wall Lemma have been used.
Moreover, it appears that exponential dependencies on $q$ might be unavoidable if one insists on $W_1$ being a subwall.
If this is indeed the case, \zcref{thm:IntroHomoMuralis} would in a sense be the strongest possible weakening that holds with polynomial bounds.

\paragraph{Two examples of how to adapt \zcref{thm:IntroHomoMuralis}.}
The idea of encoding ``types'' as colourings in order to state a general version of the Homogeneous Wall Lemma is not new.
Indeed, Sau, Stamoulis, and Thilikos state a similar lemma (Lemma 15 in \cite{SauST2024More}) with a slightly more technical definition for colourings.
However, their version uses the established single-exponential bound.
Let us briefly discuss how to use these colours to adapt our main result to different situations.

To see that \zcref{thm:IntroHomoMuralis} indeed allows for easy adaptation to a variety of settings and produces the desired outcome, if the initial wall $W_0$ is flat in $G-A$, where $A$ is a small vertex set, consider the following construction.
Set $q \coloneqq |A|$ with $A = \{ x_1, \ldots , x_q \} \subseteq V(G)$.
Now, for each vertex $v \in V(G - A)$, add the colour $i \in \{ 1, \ldots , q \}$ to the set of colours carried by $v$ if and only if $v \in N_G(x_i)$.
The result is a $q$-colorful graph $(G-A,\chi)$ containing the flat wall $W_0$.
An application of \zcref{thm:IntroHomoMuralis} now yields a flat wall $W_1 \subseteq W_0$ that is homogeneous in the sense of \zcref{thm:IntroHomoMuralis} which means that for any brick $B$ of $W_1$ and any $a\in A$, the vertex $a$ has a neighbour in the interior of $B$ if and only if $a$ has a neighbour in the interior of \textsl{every} brick of $W_1$ as desired.
\smallskip

A second scenario in which the Homogeneous Wall Lemma finds applications -- this is part of Robertson and Seymour's Graph Minor Algorithm -- is as follows:
The planar-like behaviour of a flat wall $W$ in a graph $G$ involves taking $G$ apart at separations of order at most three.
Let $H \subseteq G$ be a subgraph of $G$.
A graph $G_1$ with $H \subseteq G_1$ is an \emph{$H$-reduction} of $G$ if there exists a graph $G_2$ and cliques $J_i \subseteq G_i$, $i\in\{ 1,2 \}$, of size at most three such that $G$ can be obtained from $G_1$ and $G_2$ by identifying $J_1$ and $J_2$ into a single clique $J$ and then possibly deleting some of the edges of $J$.
If $W$ is flat in $G$ and $C$ is its perimeter, then there exists a graph $G'$ that can be obtained from $G$ by a sequence of $C$-reductions such that the union $G^*$ of $C$ and the $C$-bridge $B$ in $G'$ with $B \cap V(W) = (V(G') \cap V(W)) \setminus V(C)$ is planar and still contains a large wall $W'$.

In many applications one desires to homogenise the wall $W'$ with respect to some notion of ``profiles'' defined\footnote{The precise definition of such profiles may vary depending on the application. A common one is the set of all rooted graphs of bounded size that may be found as minors rooted on the vertex set of the clique $J$ as from above.} on the pieces of $G$ attaching to the cliques of size at most three inside $G^*$.
We may choose our $C$-reductions such that the triangles among such cliques are now facial.
If the number of possible such profiles is bounded by some integer $q$ one may simply perform the following augmentation of $G'$:
Let $\mathcal{P}$ be the set of all possible profiles and let $\psi\colon \mathcal{P}\to\{ 1,\dots,q\}$ be a bijection.
Now, for each clique $J$ of size at most three from the interior of $C$ in $G'$ introduce a new vertex $v_J$ adjacent to all vertices of $J$ and let $G''$ be the resulting graph.
Then set $\chi(v) \coloneqq \emptyset$ for all $v\in V(G')$ and $\chi(v_J)\subseteq \{ 1,\dots,q\}$ such that $i\in \chi(v_J)$ if and only if $J$ has been assigned the profile $\psi^{-1}(i)$.
If one now applies \zcref{thm:IntroHomoMuralis} to the flat wall $W'$ in the $q$-colorful graph $(G'',\chi)$ to obtain a homogeneous wall $W''$, this homogeneity transfers directly to $G'$.
This second scenario using an abstract definition of ``profiles'' also contains the first example as ``being a neighbour of some specified vertex from $A$'' may be considered a valid definition for a profile.

\subsection{Algorithmic consequences of our result.}\label{subsec:consequences}
To illustrate the impact of \zcref{thm:IntroHomoMuralis} on existing algorithms, we discuss an application to a recent result of Morelle et al.\@ \cite{MorelleSST2023Faster}.
Due to the technical nature of the Homogeneous Wall Lemma, discussing its application requires us to open up procedures which are often discussed in a black-box style.
Thus we present a single in-depth discussion of the applicability of \zcref{thm:IntroHomoMuralis} to help the reader see how the result may be incorporated to speed up the running time of related algorithms.

A core result of the work of Morelle et al.\@ is a parameterized algorithm for the \textsc{$k$-$\mathcal{H}$-Minor Deletion} problem.
Let $k$ be a non-negative integer and $\mathcal{H}$ be a finite set of graphs.
\medskip

\textsc{$k$-$\mathcal{H}$-Minor Deletion}\\
\textbf{Input:} A graph $G$.\\
\textbf{Objective:} Find, if it exists, a set $S\subseteq V(G)$ with $|S|\leq k$ such that $G-S$ does not contain a graph from $\mathcal{H}$ as a minor.
\medskip

Given a finite set $\mathcal{H}$ of graphs, we denote by $\mathsf{h}_{\mathcal{H}}$ the value $\sum_{H \in \mathcal{H}}(|H| + ||H||)$.
With this notation the result of Morelle et al.\@ we are interested reads as follows.

\begin{proposition}[Morelle, Sau, Stamoulis, Thilikos \cite{MorelleSST2023Faster}]\label{prop:kHDeletion}
There exists a function $f_{\ref{prop:kHDeletion}}\colon\mathbb{N}\to\mathbb{N}$ and an algorithm that, for every finite set $\mathcal{H}$ of graphs, every non-negative integer $k$, and all graphs $G$, solves \textsc{$k$-$\mathcal{H}$-Minor Deletion} in time $2^{k^{\mathcal{O}(f_{\ref{prop:kHDeletion}}(\mathsf{h}_\mathcal{H}))}} |G|^2$.
\end{proposition}

The running time of the algorithm in \zcref{prop:kHDeletion} is single-exponential in a polynomial in $k$.
However, this polynomial is of the form $k^{\mathcal{O}(f_{\ref{prop:kHDeletion}}(\mathsf{h}_\mathcal{H}))}$, that is, its degree depends on the choice of the \textsl{meta parameter} $\mathcal{H}$.
This behaviour is precisely what is meant by the term ``non-uniformity''.

Due to major milestones in the quest towards the creation of an efficient graph minors theory, the bounds for both the Grid Theorem \cite{ChuzhoyT2021Tighter} and the Flat Wall Theorem \cite{Chuzhoy2015Improved,KawarabayashiTW2018New,GorskySW2025Polynomial} are polynomial.
Hence there exists a constant $c$ such that in a $K_t$-minor-free graph $G$ one may either determine that the treewidth of $G$ is at most $(t + r)^c$ or there exists a set $A\subseteq V(G)$ with $|A|\leq t^c$ and a flat $r$-wall $W$ in $G-A$.
Note that every \textsc{Yes}-instance of \textsc{$k$-$\mathcal{H}$-Minor Deletion} must exclude $K_{k+\mathsf{h}_\mathcal{H}}$ as a minor.

The variant of the Homogeneous Wall Lemma used by Morelle et al.\@ stems from another work of Sau, Stamoulis, and Thilikos, namely Lemma 15 in \cite{SauST2024More}.

The algorithm from \zcref{prop:kHDeletion} now proceeds in three recursive steps, roughly outlined below.
\begin{description}
    \item[Step 1:] Check if the treewidth of $G$ is small.
    If not, find a small set $A\subseteq V(G)$ and large flat wall $W_1$ in $G-A$.
    Then consider several big subwalls $W'$ of $W$.
    If for some $W'$ the subset of $A$ with neighbours in the compass of $W'$ induces a complete graph as a minor whose size depends on $\mathsf{h}_\mathcal{H}$ but not on $k$, then proceed to \textbf{Step 2}.
    Otherwise proceed to \textbf{Step 3}.
    \item[Step 2:] Apply the Homogeneous Wall Lemma to $W'$ and find a vertex $v\in V(G-A)$ such that $(G-v,k)$ is equivalent to $(G,k)$ as an instance of \textsc{$k$-$\mathcal{H}$-Minor Deletion}.
    Recurse on $(G-v,k)$.
    \item[Step 3:]
    In this case one can show that there exists a vertex $y$ that must be contained in all optimal solutions or $(G,k)$ is a \textsc{No}-instance.
    This procedure detects a number of candidates for $y$ and branches over all of them by considering instances of the form $(G-y,k-1)$ recursively.
\end{description}
One may make two core observations on the origin of the dependency of the form $k^{\mathcal{O}(f_{\ref{prop:kHDeletion}}(\mathsf{h}_\mathcal{H}))}$ in the running time of their algorithm.
The first is:
The only reason for the dependency of the form $k^{\mathcal{O}(f_{\ref{prop:kHDeletion}}(\mathsf{h}_\mathcal{H}))}$ in \textbf{Step 2} is the direct reliance on the Homogeneous Wall Lemma.
Second: The reason for this type of dependency in \textbf{Step 3} may be found behind a more involved chain of arguments.
However, following along the chain of applications of different Lemmas one can see that also here the sole reason for the term $k^{\mathcal{O}(f_{\ref{prop:kHDeletion}}(\mathsf{h}_\mathcal{H}))}$ lies within an application of the Homogeneous Wall Lemma\footnote{For better comparison: In the proof of Theorem 6.1 in \cite{MorelleSST2023Faster} several quantities are defined. The necessary sizes for the initial flat wall in \textbf{Step 3} are $r_1$ and $r_2'$ where $r_1$ depends linearly on $r_2'$. The definition of $r_2'$ in turn depends on two functions where only one carries the dangerous dependency through its reliance on $r_2$.
In turn, $r_3$ again inherits the dependence from $r_4$, which inherits it from the definition of $t$.
The quantity $t$ finally depends on $r_5$ which is the function from the Homogeneous Wall Lemma.}.

Morelle et al.\@ ask in their paper if it is possible to change the running time from their original $2^{k^{\mathcal{O}(f_{\ref{prop:kHDeletion}}(\mathsf{h}_\mathcal{H}))}} n^2$ to an expression of the form $2^{f(\mathsf{h}_\mathcal{H}) \cdot k^d}$ for some function $f$ and some absolute constant $d$.
It now follows from the discussion above that simply replacing their exponential Homogeneous Wall Lemma with \zcref{thm:IntroHomoMuralis} achieves this goal.
Indeed, a more careful analysis of the proof of \zcref{prop:kHDeletion} in \cite{MorelleSST2023Faster} reveals that with this change we are even able to provide a concrete bound for the degree $d$ of the polynomial.

\begin{corollary}\label{cor:FasterkHDeletion}
There exists a computable function $f_{\ref{cor:FasterkHDeletion}}\colon\mathbb{N}\to\mathbb{N}$ and an algorithm that, for every finite set $\mathcal{H}$ of graphs, every non-negative integer $k$, and all graphs $G$, solves \textsc{$k$-$\mathcal{H}$-Minor Deletion} in time $2^{\mathcal{O}(f_{\ref{cor:FasterkHDeletion}}(\mathsf{h}_\mathcal{H}) \cdot k^{16})}|G|^2$.
\end{corollary}

The impact of \zcref{thm:IntroHomoMuralis} on the running time of other known parameterized algorithms is similar.
The two algorithms for \textsc{$k$-Elimination Distance to $\mathcal{G}$} for minor-closed graph classes $\mathcal{G}$ by Morelle et al.\@ \cite{MorelleSST2023Faster} profit from our main theorem in the same way as \zcref{prop:kHDeletion}.
Even graph modifications that do not fall under the umbrella of minor-friendly operations see direct improvements to the running time of their associated algorithms \cite{MorelleST2025Grapha}.

Beyond the scope of specialised algorithms one may also focus on the algorithm for \textsc{$k$-Disjoint Paths}.
As described above, the algorithm has two core routines: One for graphs of bounded treewidth and one if the treewidth is large.
The exact bound on the treewidth depends on four functions.
The first two are the function $g_1$ from the Grid Theorem itself and the function $g_2$ from the Flat Wall Theorem, both are known to be polynomial \cite{ChuzhoyT2021Tighter,Chuzhoy2015Improved,KawarabayashiTW2018New,GorskySW2025Polynomial}.
The third is the function $g_3$ from the Homogeneous Wall Lemma and the fourth is the so-called \emph{Unique Linkage Function} $g_4$ \cite{RobertsonS2009Graph,RobertsonS2012Graph}.
The total bound amounts to $g_1\!\left( g_2\!\left( g_3\!\left( g_4\!\left( k \right) \right) \right) \right) \in g_3 \!\left( g_4 \!\left( k \right) \right)^c$ for some universal constant $c>1$.

With previous bounds on $g_3$ this means that the bound on the treewidth distinguishing between the two cases of the \textsc{$k$-Disjoint Paths} algorithm is of the form $2^{\mathcal{O}(g_4(k))}$.
Applying \zcref{thm:IntroHomoMuralis} reduces this bound to $g_4(k)^{\mathcal{O}(1)}$.
As mentioned above, any application of the Irrelevant Vertex Technique relies on all four functions mentioned above.
Hence, the improvement explained above applies to any application of Irrelevant Vertices in the case of $H$-minor-free graphs\footnote{In most cases, the presence of a large clique minor provides an irrelevant vertex in more straightforward way with very good bounds. Only few examples are known where the clique case poses a challenge on its own (see \cite{FominLPSZ2020Hitting}).}.

Providing explicit and good bounds for $g_4$ is maybe one of the biggest challenges in the algorithmic theory of graph minors.
Robertson and Seymour did not give any explicit bound.
A first bound was found by Geelen, Huynh, and Richter \cite{GeelenHR2018Explicit} but their estimate is an iterated power tower function.
Before that, Kawarabayashi and Wollan announced a new proof for the Unique Linkage Theorem \cite{KawarabayashiW2010Shorter} but never computed their bound.
Wollan estimates this bound to be at most four-fold exponential \cite{Wollan2024Personal}.
The only known lower bound on $g_4$ is due to Adler, Kolliopoulos, Krause, Lokshtanov, Saurabh, and Thilikos \cite{AdlerKKLST2011Tight,AdlerK2019Lower} and shows that $g_4(k) \in 2^{\Omega(k)}$.
With our result and the recently established polynomial bounds for the Graph Minor Structure Theorem \cite{GorskySW2025Polynomial}, the only remaining pillar of algorithmic graph minor theory still resisting optimisation is now the Unique Linkage Function itself. 

\subsection{A short overview of our proof.}\label{subsec:proofOverview}
We conclude this introduction with a short description of our proof.
Our main tool is the ``planar-like'' behaviour of a flat wall, so for the purpose of this summary we advise the reader to think of $G$ as a planar graph.
Moreover, instead of walls, we think of a slight generalisation of walls called ``meshes''.
In short, a mesh $M$ is made up by a pair of families of disjoint paths $\mathcal{P}$ and $\mathcal{Q}$ intersecting as follows.
There are orderings $\mathcal{P}=\{ P_1,\dots,P_n\}$ and $\mathcal{Q}=\{ Q_1,\dots,Q_m\}$ such that every $Q_i\in\mathcal{Q}$ is a $V(P_1)$-$V(P_n)$-path intersecting all paths from $\mathcal{P}$ in order and where $Q_i \cap P_j$ is also a path.
Furthermore, every \(P_i \in \mathcal{P}\) is a \(V(Q_1)\)-\(V(Q_m)\)-path intersecting the paths from \(\mathcal{Q}\) in order.
The mesh $M$ is now the union of all those paths and we say that $M$ is an $(n \times m)$-mesh.
We also call the paths in $\mathcal{P}$ the \emph{horizontal paths} and the paths in $\mathcal{Q}$ the \emph{vertical paths}.
The cycle $P_1 \cup Q_1 \cup P_n \cup Q_m$ is the \emph{perimeter} of $M$.
See \zcref{fig:Intromesh} for an illustration of a $(6 \times 6)$-mesh.

\begin{figure}[ht]
    \centering
    \begin{tikzpicture}

        \pgfdeclarelayer{background}
		\pgfdeclarelayer{foreground}
			
		\pgfsetlayers{background,main,foreground}

        \begin{pgfonlayer}{background}
            \pgftext{\includegraphics[width=6cm]{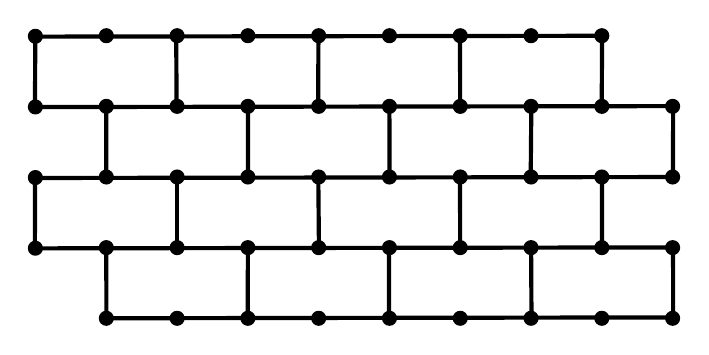}} at (C.center);
        \end{pgfonlayer}{background}
			
        \begin{pgfonlayer}{main}
        \node (C) [v:ghost] {};

        \node (bump) [v:ghost,position=270:2.5cm from C] {};
            
        \end{pgfonlayer}{main}
        
        \begin{pgfonlayer}{foreground}
        \end{pgfonlayer}{foreground}

    \end{tikzpicture}
    \begin{tikzpicture}

        \pgfdeclarelayer{background}
		\pgfdeclarelayer{foreground}
			
		\pgfsetlayers{background,main,foreground}

        \begin{pgfonlayer}{background}
            \pgftext{\includegraphics[width=5cm]{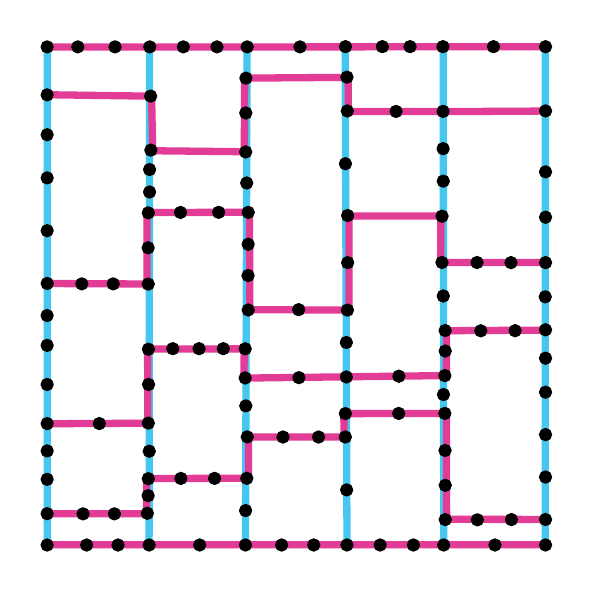}} at (C.center);
        \end{pgfonlayer}{background}
			
        \begin{pgfonlayer}{main}
        \node (C) [v:ghost] {};
            
        \end{pgfonlayer}{main}
        
        \begin{pgfonlayer}{foreground}
        \end{pgfonlayer}{foreground}

    \end{tikzpicture}
    \caption{A $4$-wall (on the left) and a $(6 \times 6)$-mesh (on the right).}
    \label{fig:Intromesh}
\end{figure}

Every wall is also a mesh.
The advantage of meshes is that they allow us to more easily to use their horizontal and vertical paths to divide the plane into regions.
This reduces a large amount of the discussion on the position of colours to a combinatorial problem with a geometric flavour.

The key insight allowing us to drop the bounds from exponential to polynomial is that one does not have to take a \textsl{subwall} to retain algorithmic applicability.
A subwall refers to a wall $W'$ obtained from another wall $W$ by taking
subpaths of horizontal paths of $W$ as horizontal paths of $W'$ and taking subpaths of vertical paths of $W$ as vertical paths of $W'$.
This restriction severely limits the applicability of more involved methods of counting and creative applications of the pigeon hole principle.
By allowing $W'$ to be a \textsl{subgraph} agreeing on the same tangle, we are able to retain the original (almost) embedding of $W$ as a witness of the flatness of $W'$ -- or in this example of planarity. 
This is all that is necessary for all known applications of the Irrelevant Vertex Technique.

\paragraph{Strips.}
We may now imagine our mesh $M$ to be embedded on a disk $\Delta$ with its perimeter drawn on the boundary of $\Delta$.
To further aid our intuition we imagine $\Delta$ to be a rectangle where each of the four paths of $M$ making up the perimeter is drawn on its private side of the rectangle and the vertices where horizontal and vertical perimeter paths intersect are drawn in the corners.

Now, any two horizontal (or vertical) paths from $M$ carve out another rectangle -- at least up to homeomorphism -- from $\Delta$ containing all the horizontal (or vertical) paths drawn in between them.
More explicitly, $P_i$ and $P_{i+k}$ together with the boundary of $\Delta$ bound a unique disk which contains $P_{i+j}$ for all $j\in\{ 1,\dots,k-1\}$.
We will use this to create what we call \emph{strips}, i.e. the restriction of $G$ to the part drawn in this disk defined by $P_i$ and $P_{i+k}$.
\zcref{fig:StripsOverview} contains a schematic illustration of a vertical strip.

\begin{figure}[ht]
    \centering
    \begin{tikzpicture}

        \pgfdeclarelayer{background}
		\pgfdeclarelayer{foreground}
			
		\pgfsetlayers{background,main,foreground}
			
        \begin{pgfonlayer}{main}
        \node (C) [v:ghost] {};

        \node(L) [v:ghost,position=180:5cm from C] {
            \begin{tikzpicture}

                \pgfdeclarelayer{background}
		          \pgfdeclarelayer{foreground}
			
		          \pgfsetlayers{background,main,foreground}

                \begin{pgfonlayer}{background}
                    \pgftext{\includegraphics[width=4cm]{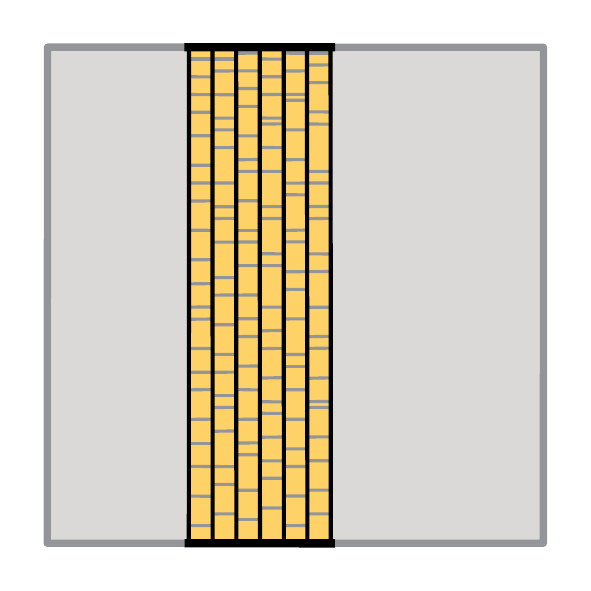}} at (C.center);
                \end{pgfonlayer}{background}
			
                \begin{pgfonlayer}{main}
                    \node (C) [v:ghost] {};
            
                \end{pgfonlayer}{main}
        
                \begin{pgfonlayer}{foreground}
                \end{pgfonlayer}{foreground}

            \end{tikzpicture}
        };

        \node(M) [v:ghost,position=0:0cm from C] {
            \begin{tikzpicture}

                \pgfdeclarelayer{background}
		          \pgfdeclarelayer{foreground}
			
		          \pgfsetlayers{background,main,foreground}

                \begin{pgfonlayer}{background}
                    \pgftext{\includegraphics[width=4cm]{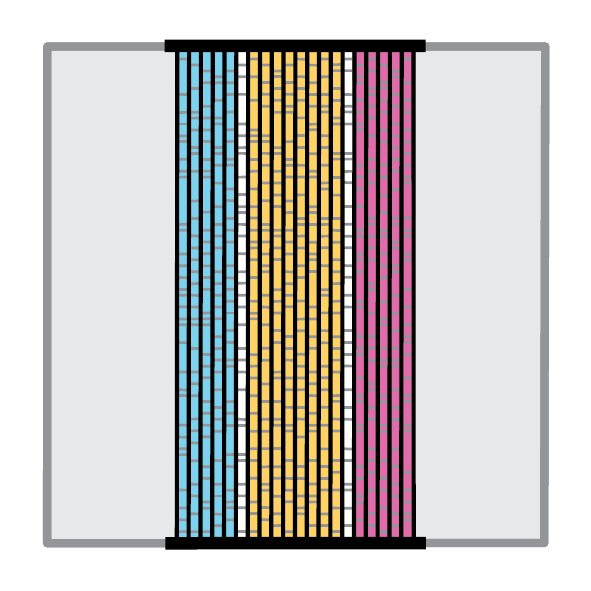}} at (C.center);
                \end{pgfonlayer}{background}
			
                \begin{pgfonlayer}{main}
                    \node (C) [v:ghost] {};
            
                \end{pgfonlayer}{main}
        
                \begin{pgfonlayer}{foreground}
                \end{pgfonlayer}{foreground}

            \end{tikzpicture}
        };

        \node (Llabel) [v:ghost,position=270:2.3cm from L] {(i)};

        \node (Mlabel) [v:ghost,position=270:2.3cm from M] {(ii)};
            
        \end{pgfonlayer}{main}
        
        \begin{pgfonlayer}{foreground}
        \end{pgfonlayer}{foreground}

    \end{tikzpicture}
    \caption{Diagrams of (i) a strip in a mesh -- the grey area is the rectangle $\Delta$ and (ii) a ``padded'' strip partitioned into a middle part and two buffer zones (one of the left and the other on the right).}
    \label{fig:StripsOverview}
\end{figure}
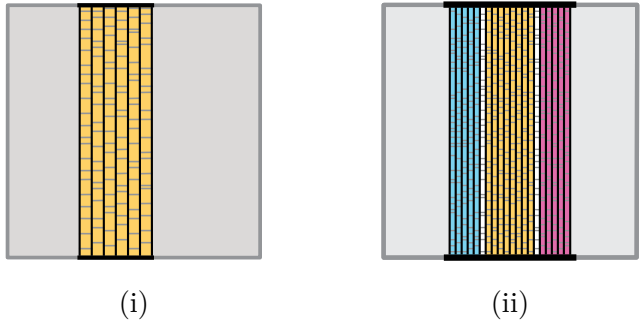

To each strip $S$ -- vertical or horizontal -- we may associate a \emph{colour profile}, which is simply the set of colours $i$ such that at least one vertex drawn in $S$ carries $i$ in its palette.

\paragraph{Step 1: Homogenise strips.}
Our goal is to find the following:
Suppose we have $q$ colours and are given positive integers $k$ and $r$.
Our first target is to start with a large collection of wide -- polynomial in $q$, $k$, and $r$ -- vertical (or horizontal) and pairwise disjoint strips, and find a collection of $k$ pairwise disjoint strips $\mathcal{S}$ together with a set $I\subseteq \{ 1,\dots,q\}$ such that for each $i\in I$, at least $r$ strips from $\mathcal{S}$ contain $i$ in their profile and for each $i\in \{ 1,\dots,q\}\setminus I$ no strip in $\mathcal{S}$ carries $i$.
\smallskip

If we were not interested in polynomial bounds we could easily achieve something stronger.
Let $\mathcal{S}_0$ be the initial selection of strips.
For each $i\in\{ 1,\dots,q\}$, let $\mathcal{S}_{i}$ be obtained from $\mathcal{S}_{i-1}$ by the following selection process.
If at least $\nicefrac{|\mathcal{S}_{i-1}|}{2}$ strips have $i$ in their palette let $\mathcal{S}_i$ be all those strips with $i$ in their palette.
Otherwise let $\mathcal{S}_i$ be all strips from $\mathcal{S}_{i-1}$ which do not carry $i$ in their palette.
As a result we have that, if $|\mathcal{S}_0|\geq 2^q \cdot r$, then $|\mathcal{S}_q|\geq r$ and the palettes of all strips in $\mathcal{S}_q$ are equal.
\smallskip

To ensure that this process is also polynomial in $q$, we alter it slightly.
Instead of considering plain strips, we split each strip into three parts: Two small strips, one on either side, acting as a buffer and a large middle strip.
See \zcref{fig:StripsOverview} for an illustration of such a splitting of a strip.
These buffers will help us to sort out colours more efficiently and also provide sufficient infrastructure to eventually find our desired mesh.
Such infrastructure is necessary as colours that occur only close to the boundary of a strip are difficult to incorporate into the middle of some mesh without leaving the confinement of the strips.

We may now alter the approach from above as follows.
Starting from a collection $\mathcal{A}_0$ of strips, we first find a subset $I_1 \subseteq \{ 1,\dots,q\}$ of colours which appear in less than $r$ strips.
We then remove all strips that carry at least one colour from $I_1$, thereby losing at most $|I_1|\cdot r \leq q \cdot r$ strips, set $X_1 \coloneqq \{ 1,\dots,q\}\setminus I_1$, and let $\mathcal{B}_0$ be the collection of remaining strips.
Then we select $J_1 \subseteq X_1$ to be the set of colours $i$ which appear in less than $r$ middle parts of the strips from $\mathcal{B}_0$.
In case $J_1 = \emptyset$ we terminate.
Otherwise we first forget all strips with some colour from $J_1$ in their middle part -- there are at most $q\cdot r$ many -- and set $Y_1 \coloneqq X_1 \setminus J_1$.
From each of the remaining strips we only keep their middle part, thereby forming the set $\mathcal{A}_1$ of remaining strips which we again partition into a middle part and two buffer zones.

Iterating this procedure eventually yields a set $\mathcal{B}_j$ of strips and a set $X_j$ of colours such that
\begin{enumerate}
    \item the profile of each strip in $\mathcal{B}_j$ is a subset of $X_j$,
    \item for each $i\in X_j$ there are at least $r$ strips in $\mathcal{B}_j$ whose middle parts carry the colour $i$, and
    \item $|\mathcal{B}_j| \geq |\mathcal{A}_0| - q^2r$ and each strip has shed its buffer zones at most $q$ times, meaning all strips are still as large as we desire.
\end{enumerate}

\paragraph{Step 2: Tiles and further homogenisation.}
The procedure from \textbf{Step 1} now yields a big collection of horizontal strips which are close to being homogeneous and a big collection of vertical strips which are also close to being homogeneous.
We need one further step of homogenisation before we can start constructing the homogeneous mesh we desire.

\begin{figure}[ht]
    \centering
    \begin{tikzpicture}

        \pgfdeclarelayer{background}
		\pgfdeclarelayer{foreground}
			
		\pgfsetlayers{background,main,foreground}
			
        \begin{pgfonlayer}{main}
        \node (C) [v:ghost] {};

        \node(L) [v:ghost,position=180:5cm from C] {
            \begin{tikzpicture}

                \pgfdeclarelayer{background}
		          \pgfdeclarelayer{foreground}
			
		          \pgfsetlayers{background,main,foreground}

                \begin{pgfonlayer}{background}
                    \pgftext{\includegraphics[width=4cm]{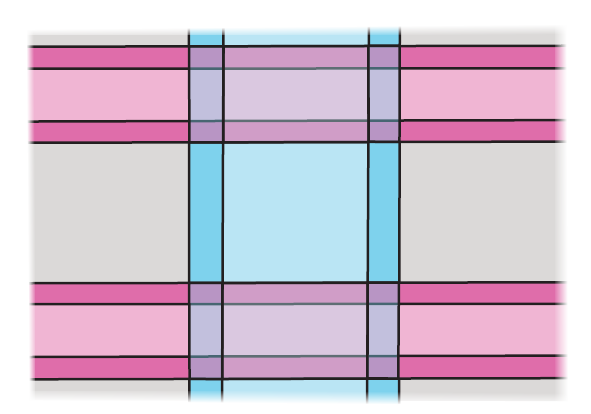}} at (C.center);
                \end{pgfonlayer}{background}
			
                \begin{pgfonlayer}{main}
                    \node (C) [v:ghost] {};
            
                \end{pgfonlayer}{main}
        
                \begin{pgfonlayer}{foreground}
                \end{pgfonlayer}{foreground}

            \end{tikzpicture}
        };

        \node(M) [v:ghost,position=0:0cm from C] {
            \begin{tikzpicture}

                \pgfdeclarelayer{background}
		          \pgfdeclarelayer{foreground}
			
		          \pgfsetlayers{background,main,foreground}

                \begin{pgfonlayer}{background}
                    \pgftext{\includegraphics[width=4cm]{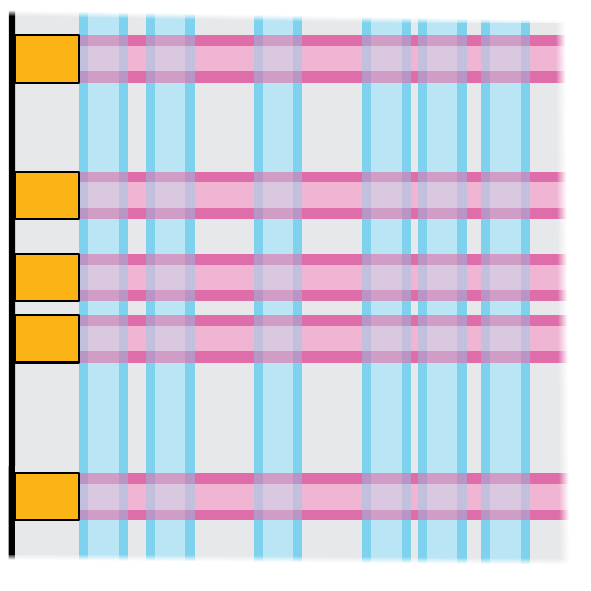}} at (C.center);
                \end{pgfonlayer}{background}
			
                \begin{pgfonlayer}{main}
                    \node (C) [v:ghost] {};
            
                \end{pgfonlayer}{main}
        
                \begin{pgfonlayer}{foreground}
                \end{pgfonlayer}{foreground}

            \end{tikzpicture}
        };

        \node(R) [v:ghost,position=0:5cm from C] {
            \begin{tikzpicture}

                \pgfdeclarelayer{background}
		          \pgfdeclarelayer{foreground}
			
		          \pgfsetlayers{background,main,foreground}

                \begin{pgfonlayer}{background}
                    \pgftext{\includegraphics[width=4cm]{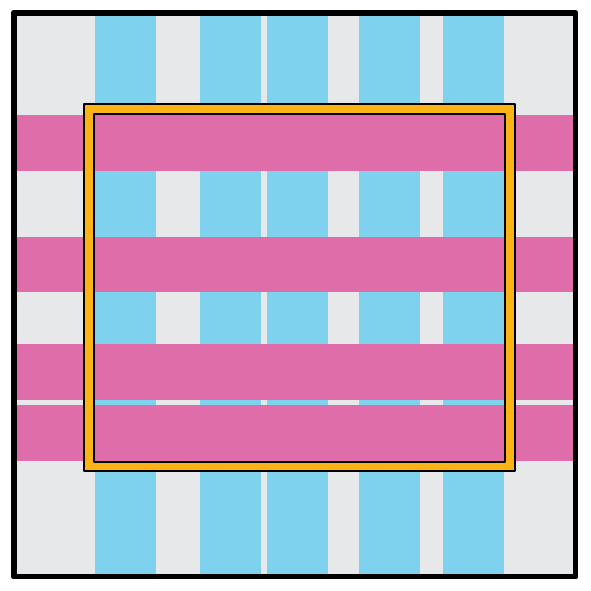}} at (C.center);
                \end{pgfonlayer}{background}
			
                \begin{pgfonlayer}{main}
                    \node (C) [v:ghost] {};
            
                \end{pgfonlayer}{main}
        
                \begin{pgfonlayer}{foreground}
                \end{pgfonlayer}{foreground}

            \end{tikzpicture}
        };

        \node (Llabel) [v:ghost,position=270:2.3cm from L] {(i)};

        \node (Mlabel) [v:ghost,position=270:2.3cm from M] {(ii)};

        \node (Rlabel) [v:ghost,position=270:2.3cm from R] {(iii)};
            
        \end{pgfonlayer}{main}
        
        \begin{pgfonlayer}{foreground}
        \end{pgfonlayer}{foreground}

    \end{tikzpicture}
    \caption{Diagrams of (i) tiles created by overlaying horizontal and vertical strips, (ii) dangerous ``end tiles'' that are not fully surrounded by buffer zones, and (iii) a submesh cropped to the union of horizontal and vertical strips.}
    \label{fig:TilesOverview}
\end{figure}
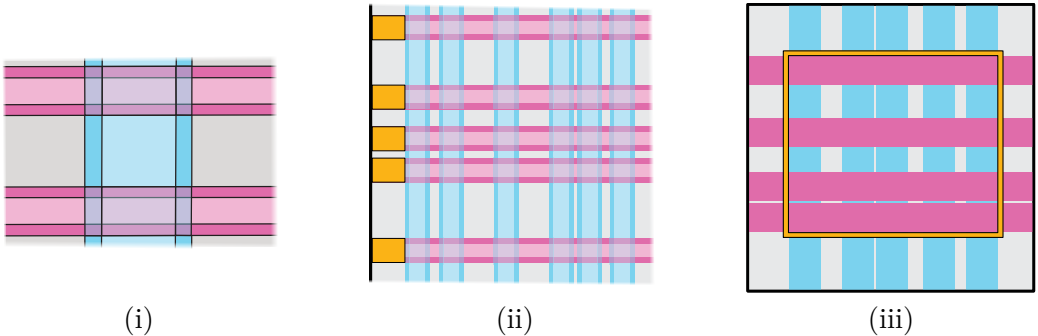

Overlaying vertical and horizontal strips gives rise to two new types of areas: Those areas where a horizontal and a vertical strip intersect (\emph{type 1}), and the area exclusive to some vertical (horizontal) strip, two sides of which are bordered by the boundary of the mesh or some horizontal (vertical) strip (\emph{type 2}).
See \zcref{fig:TilesOverview} for an illustration.
We call these area \emph{tiles}.

Note that some tiles are fully surrounded by buffer areas while others are surrounded by buffer areas only on three sides.
The latter may be considered to be ``end tiles'' as this only occurs at the boundary of our mesh as indicated in \zcref{fig:TilesOverview}.
For each tile we may now define a ``middle part'' as the intersection of the middle parts of the two strips involved if the tile is of type 1 and the intersection of the tile with the middle part of its strip if the tile is of type 2.

From \textbf{Step 1} we know that any colour that occurs in the profile of some vertical or horizontal strip in our collection must also occur in the middle part of at least $r$ strips.
Our next goal is to further refine this statement to say that any colour occurring in some strip occurs in the middle parts of at least $r$ distinct tiles which are \textsl{not} end tiles.
However, this requires additional arguments and sacrifices as some parts of one strip now belong to the buffer areas of another and because it might be possible that some colours are confined mostly into end tiles.

To realise our desired outcome, we essentially repeat the sorting procedure from the first step, but now we incorporate a step where we also crop the entire mesh in order to dismiss end tiles (see \zcref{fig:TilesOverview} for an illustration).
As in \textbf{Step 1}, in each iteration we only throw away a bounded number of strips and we shrink each strip only by a polynomial amount in order to readjust its buffer zones.
Since the total number of iterations is bounded by the number of remaining colours, the total sacrifice in number of strips and in their width remains polynomial as before.

\paragraph{Step 3: Constructing the homogeneous mesh.}
The outcome of \textbf{Step 2} is a collection of horizontal and vertical strips $\mathcal{S}$ together with a subset $X$ of the original $q$ colours such that the strips from $\mathcal{S}$ carry only the colours in $X$ and each colour in $X$ occurs in the middle parts of at least $r$ distinct non-end tiles.
All that is left is to show how this situation may be utilised to construct a homogeneous mesh.

\begin{figure}[ht]
    \centering
    \begin{tikzpicture}

        \pgfdeclarelayer{background}
		\pgfdeclarelayer{foreground}
			
		\pgfsetlayers{background,main,foreground}
			
        \begin{pgfonlayer}{main}
        \node (C) [v:ghost] {};

        \node(L) [v:ghost,position=180:5cm from C] {
            \begin{tikzpicture}

                \pgfdeclarelayer{background}
		          \pgfdeclarelayer{foreground}
			
		          \pgfsetlayers{background,main,foreground}

                \begin{pgfonlayer}{background}
                    \pgftext{\includegraphics[width=4cm]{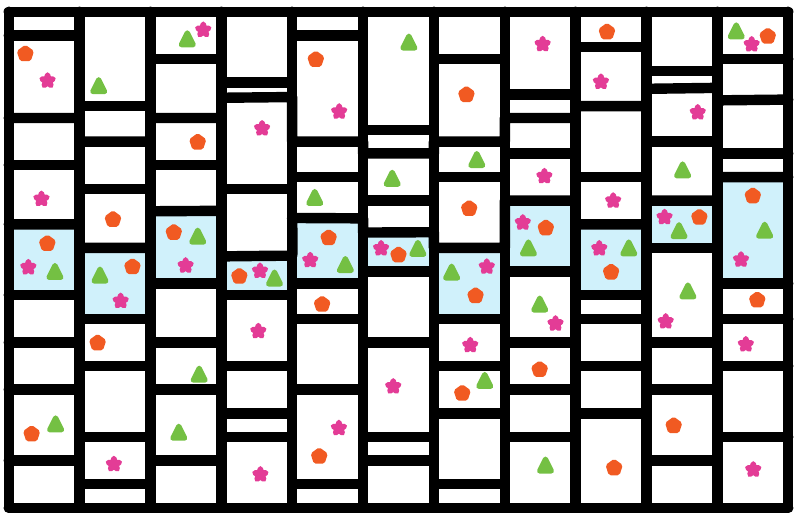}} at (C.center);
                \end{pgfonlayer}{background}
			
                \begin{pgfonlayer}{main}
                    \node (C) [v:ghost] {};
            
                \end{pgfonlayer}{main}
        
                \begin{pgfonlayer}{foreground}
                \end{pgfonlayer}{foreground}

            \end{tikzpicture}
        };

        \node(M) [v:ghost,position=0:0cm from C] {
            \begin{tikzpicture}

                \pgfdeclarelayer{background}
		          \pgfdeclarelayer{foreground}
			
		          \pgfsetlayers{background,main,foreground}

                \begin{pgfonlayer}{background}
                    \pgftext{\includegraphics[width=4cm]{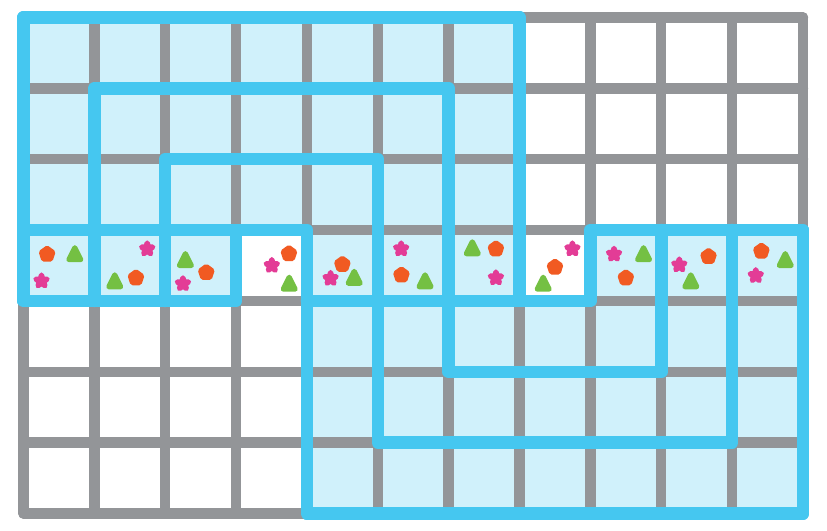}} at (C.center);
                \end{pgfonlayer}{background}
			
                \begin{pgfonlayer}{main}
                    \node (C) [v:ghost] {};
            
                \end{pgfonlayer}{main}
        
                \begin{pgfonlayer}{foreground}
                \end{pgfonlayer}{foreground}

            \end{tikzpicture}
        };

        \node (Llabel) [v:ghost,position=270:2.3cm from L] {(i)};

        \node (Mlabel) [v:ghost,position=270:2.3cm from M] {(ii)};
            
        \end{pgfonlayer}{main}
        
        \begin{pgfonlayer}{foreground}
        \end{pgfonlayer}{foreground}

    \end{tikzpicture}
    \caption{Diagrams of (i) a mesh with a rainbow middle row and (ii) the construction of a homogeneous mesh from a rainbow-middle-row mesh.}
    \label{fig:HomoOverview}
\end{figure}
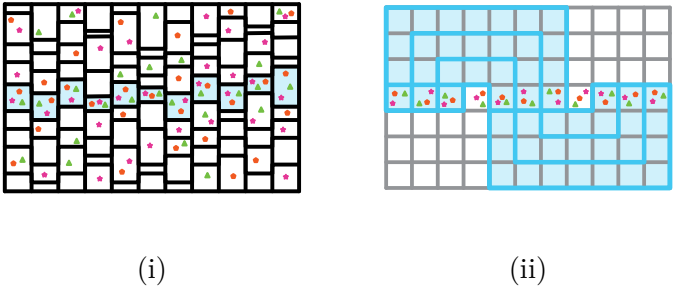

This construction is done in two steps.
First, we make use of the buffer zones and infrastructure of the union of all strips in $\mathcal{S}$ to construct a new mesh $M'$ where each face of the ``middle row'' of $M'$ carries every colour in $X$.
We call $M'$ a mesh with a \emph{rainbow middle row}.
See \zcref{fig:HomoOverview}.

By choosing $r$ and $M'$ large enough to ensure that $M'$ has $k^2$ many faces in the middle row, we may then weave vertical paths of a new mesh $M''$ around the middle row to construct the final homogeneous mesh as illustrated in \zcref{fig:HomoOverview}.

\newpage

\section{Preliminaries}\label{sec:wallsAndTangles}
We generally adhere to the notation for graphs defined in \cite{Diestel2010Graph}.
All graphs we consider are simple, meaning that they contain neither loops nor parallel edges, and contain at least one vertex.

By $\mathbb{Z}$ we denote the set of integers.
Given any two integers $a,b\in\mathbb{Z}$, we write $[a,b]$ for the set $\{z\in\mathbb{Z} ~\!\colon\!~ a\leq z\leq b\}$.
The set $[a,b]$ is empty whenever $a>b$.
For any positive integer $c$ we set $[c]\coloneqq [1,c]$.

Several notions from the Graph Minor Structure Theory due to Robertson and Seymour are fundamental to our methods.
Thus we use this section mainly to introduce the basic concepts from that context that we will need in our work.

\subsection{Tangles}\label{subsec:Tangles}
A \emph{separation} in a graph $G$ is a pair $(A,B)$ of vertex sets such that $A \cup B = V(G)$ and there is no edge in $G$ with one endpoint in $A \setminus B$ and the other in $B \setminus A$.
The \emph{order} of $(A,B)$ is $|A\cap B|$.
If $k$ is a positive integer, we denote the collection of all separations $(A,B)$ of order less than $k$ in $G$ as $\mathcal{S}_k(G)$.

An \emph{orientation} of $\mathcal{S}_k(G)$ is a set $\mathcal{O}$ such that for all $(A,B) \in \mathcal{S}_k(G)$ exactly one of $(A,B)$ and $(B,A)$ belongs to $\mathcal{O}$. 
A \emph{tangle} of order $k$ in $G$ is an orientation $\mathcal{T}$ of $\mathcal{S}_k(G)$ such that for all $(A_1,B_1), (A_2,B_2), (A_3,B_3) \in \mathcal{T}$, it holds that $G[A_1] \cup G[A_2] \cup G[A_3] \neq G$.
If $\mathcal{T}$ is a tangle and $(A,B)\in\mathcal{T}$, we call $A$ the \emph{small side} and $B$ the \emph{big side} of $(A,B)$.

Let $G$ be a graph and $\mathcal{T}$ and $\mathcal{D}$ be tangles of $G$.
We say that $\mathcal{D}$ is a \emph{truncation} of $\mathcal{T}$ if $\mathcal{D}\subseteq\mathcal{T}$.

\subsection{Walls and meshes}\label{subsec:Walls}
The main objects of our interest are walls.
However, for convenience we actually consider a slightly more general structure called ``meshes''.

\paragraph{Grids and Walls.}
Let $n,m\geq 1$ be integers.
The \emph{$(n \times m)$-grid} is the graph with vertex set $[n] \times [m]$ where $\{(i_1,j_1),(i_2,j_2)\}$ is an edge if and only if
\begin{itemize}
    \item $i_1 = i_2$ and $|j_1 - j_2|=1$, or
    \item $j_1 = j_2$ and $|i_1 - i_2|=1$.
\end{itemize}

\begin{figure}[ht]
    \centering
    \begin{tikzpicture}

        \pgfdeclarelayer{background}
		\pgfdeclarelayer{foreground}
			
		\pgfsetlayers{background,main,foreground}

        \begin{pgfonlayer}{background}
            \pgftext{\includegraphics[width=9cm]{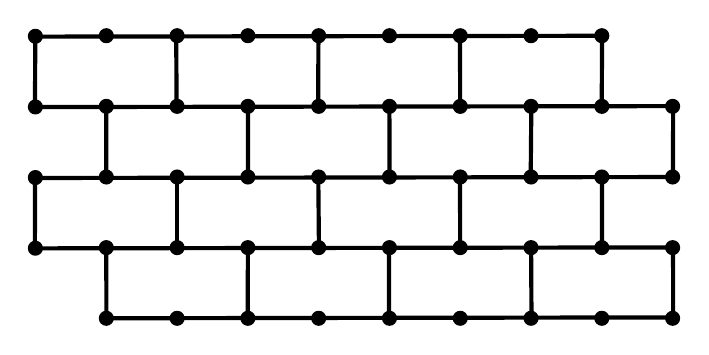}} at (C.center);
        \end{pgfonlayer}{background}
			
        \begin{pgfonlayer}{main}
        \node (C) [v:ghost] {};
            
        \end{pgfonlayer}{main}
        
        \begin{pgfonlayer}{foreground}
        \end{pgfonlayer}{foreground}

    \end{tikzpicture}
    \caption{The elementary $4$-wall.}
    \label{fig:wall}
\end{figure}

Given an integer $n$, the \emph{elementary $n$-wall} is the graph obtained from the $(n \times 2n)$-grid by deleting all edges of the form
\begin{itemize}
    \item $\{(i,j),(i+1,j) \}$ for $i\in[n-1]$ odd and $j\in[2n]$ even, and
    \item $\{(i,j),(i+1,j) \}$ for $i\in[n-1]$ even and $j\in[2n]$ odd.
\end{itemize}
See \zcref{fig:wall} for an illustration.
An \emph{$n$-wall} is any graph that can be obtained from the elementary $n$-wall by subdividing its edges.

One can easily see that any graph that contains the $(n\times 2n)$-grid as a minor must contain an $n$-wall as a subgraph.
Moreover, every $n$-wall contains the $(n\times n)$-grid as a minor.

\paragraph{Meshes.}
Let $A, B \subseteq V(G)$, an \emph{$A$-$B$-path} is a path $P$ that has one endpoint in $A$, the other in $B$, and $V(P) \cap (A \cup B)$ only contains the endpoints of $P$.
Let $n,m$ be integers with $n,m\geq 2$.
A \emph{$(n\times m)$-mesh} is a graph $M$ which is the union of paths $M=P_1\cup\cdots\cup P_n\cup Q_1\cup \cdots \cup Q_m$ where
    \begin{itemize}
        \item $P_1,\cdots,P_n$ are pairwise vertex-disjoint, and $Q_1,\cdots,Q_m$ are pairwise vertex-disjoint.
        \item for every $i\in [n]$ and $j\in [m]$, the intersection $P_i\cap Q_j$ induces a path,
        \item each $P_i$ is a  $V(Q_1)$-$V(Q_m)$-path intersecting the paths $Q_1,\cdots Q_m$ in the given order, and each $Q_j$ is a $V(P_1)$-$V(P_m)$-path intersecting the paths $P_1,\cdots, P_h$ in the given order. 
    \end{itemize}
We say that the paths $P_1,\cdots,P_n$ are the \emph{horizontal paths}, and the paths $Q_1,\cdots,Q_m$ are the \emph{vertical paths}.
A mesh $M'$ is a \emph{submesh} of a mesh $M$ if every horizontal (vertical) paths of $M'$ is a subpath of a horizontal (vertical) paths $M$, respectively.
The \emph{perimeter} of the mesh $M$ is the cycle $P_1 \cup Q_1 \cup P_n \cup Q_m$.
We write \emph{$n$-mesh} as a shorthand for an $(n \times n)$-mesh.
See \zcref{fig:mesh} for an illustration. 

\begin{figure}[ht]
    \centering
    \begin{tikzpicture}

        \pgfdeclarelayer{background}
		\pgfdeclarelayer{foreground}
			
		\pgfsetlayers{background,main,foreground}

        \begin{pgfonlayer}{background}
            \pgftext{\includegraphics[width=8cm]{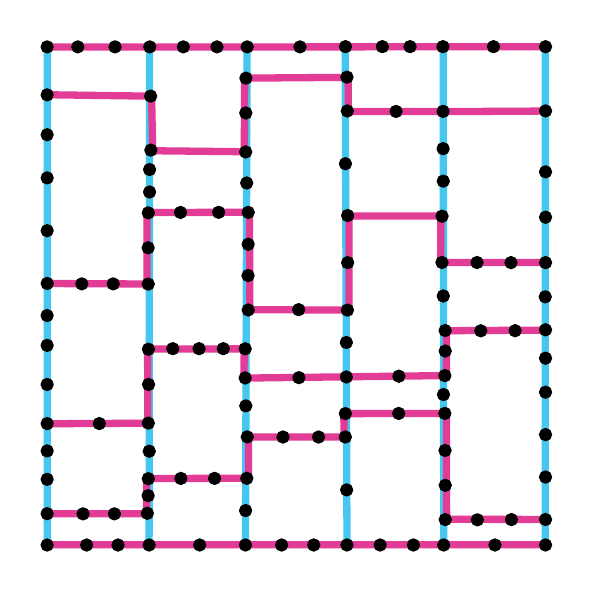}} at (C.center);
        \end{pgfonlayer}{background}
			
        \begin{pgfonlayer}{main}
        \node (C) [v:ghost] {};
            
        \end{pgfonlayer}{main}
        
        \begin{pgfonlayer}{foreground}
        \end{pgfonlayer}{foreground}

    \end{tikzpicture}
    \caption{A $(6 \times 6)$-mesh.}
    \label{fig:mesh}
\end{figure}

Every $n$-wall is in fact an $n$-mesh.
Moreover, every $2n$-mesh contains an $n$-wall.
So while meshes may appear much more complex than walls, they capture precisely the same kind of structure and therefore, it suffices to only consider meshes.

\paragraph{Tangles from meshes.}
Let $r, \ell \in \mathbb{N}$ with $r \geq \ell \geq 3$, let $G$ be a graph, and $M$ be an $(r \times \ell)$-mesh in $G$.
Let $\mathcal{T}_M$ be the orientation of $\mathcal{S}_\ell$ such that for every $(A,B) \in \mathcal{T}_M$, the set $B \setminus A$ contains the vertex set of both a horizontal and a vertical path of $M$, we call $B$ the \emph{$M$-majority side} of $(A,B)$.

It is a well-known fact that $\mathcal{T}_M$ defines a tangle of order $\ell$ \cite{RobertsonS1991Graph}.

\subsection{Almost embeddings and flatness}\label{sec:AlmostEmbedding}
The definitions we present are inspired by the work of \cite{KawarabayashiTW2018New,KawarabayashiTW2021Quickly,GorskySW2025Polynomial}.
Of course, most of them are in some way descendants of the work of Robertson and Seymour.
We will omit a definition of surfaces here (see \cite{MoharT2001Graphs} for a formal definition), as we will only be working in the disk and the sphere (equivalently the plane) in this paper.

\paragraph{Paintings in surfaces.}
A \emph{painting} in a surface $\Sigma$ is a pair $\Gamma = (U,N)$, where $N \subseteq U \subseteq \Sigma$, $N$ is finite, $U \setminus N$ has a finite number of arcwise-connected components, called \emph{cells} of $\Gamma$, and for every cell $c$, the closure $\overline{c}$ is a closed disk where $N_\Gamma(c) \coloneqq \overline{c} \cap N \subseteq \mathsf{bd}(\overline{c})$.
If $|N_\Gamma(c)| \geq 4$, the cell $c$ is called a \emph{vortex}.
We further let $N(\Gamma) \coloneqq N$, let $U(\Gamma) \coloneqq U$, and let $C(\Gamma)$ be the set of all cells of $\Gamma$.
\medskip

Any given painting $\Gamma = (U,N)$ defines a hypergraph with $N$ as its vertices and the set of closures of the cells of $\Gamma$ as its edges.
Accordingly, we call $N$ the \emph{nodes} of $\Gamma$.

\paragraph{Sphere-, Disk-, and $\Sigma$-renditions.}
Let $G$ be a graph and $\Sigma$ be a surface.
A \emph{$\Sigma$-rendition} of $G$ is a triple $\rho = (\Gamma, \sigma, \pi)$, where
\begin{itemize}
    \item $\Gamma$ is a painting in $\Sigma$,
    \item for each cell $c \in C(\Gamma)$, $\sigma(c)$ is a subgraph of $G$, and
    \item $\pi \colon N(\Gamma) \to V(G)$ is an injection,
\end{itemize}
such that
\begin{description}
    \item[R1] $G = \bigcup_{c \in C(\Gamma)}\sigma(c)$,
    \item[R2] for all distinct $c,c' \in C(\Gamma)$, the graphs $\sigma(c)$ and $\sigma(c')$ are edge-disjoint,
    \item[R3] $\pi(N_\Gamma(c)) \subseteq V(\sigma(c))$ for every cell $c \in C(\Gamma)$, and
    \item[R4] for every cell $c \in C(\Gamma)$, we have $V(\sigma(c) \cap \bigcup_{c' \in C(\Gamma) \setminus \{ c \}} (\sigma(c'))) \subseteq \pi(M_\Gamma(c))$.
\end{description}
We write $N(\rho)$ for the set $N(\Gamma)$, let $N_\rho(c) = N_\Gamma(c)$ for all $c \in C(\Gamma)$, and similarly, we lift the set of cells from $C(\Gamma)$ to $C(\rho)$.
If it is clear from the context which $\rho$ is meant, we will sometimes simply write $N(c)$ instead of $N_\rho(c)$, and if the $\Sigma$-rendition $\rho$ for $G$ is understood from the context, we usually identify the sets $\pi(N(\rho))$ and $N(\rho)$ along $\pi$ for ease of notation.

In the case where $\Sigma$ is homeomorphic to a sphere we talk about \emph{sphere-renditions} instead of $\Sigma$-renditions.
Similarly, if $\Sigma$ is homeomorphic to a disk, we call $\rho$ a \emph{disk-rendition}.
In this paper all renditions we consider are either sphere- or disk-renditions. 

\paragraph{Traces of paths and cycles.}
Let $\rho$ be a $\Sigma$-rendition of a graph $G$.
For every cell $c \in C(\rho)$ with $|N_\rho(c)| = 2$, we select one of the components of $\mathsf{bd}(c) - N_\rho(c)$.
This selection will be called a \emph{tie-breaker in $\rho$}, and we assume that every rendition comes equipped with a tie-breaker.

Let $G$ be a graph and $\rho$ be a $\Sigma$-rendition of $G$.
Let $Q$ be a cycle or path in $G$ that uses no edge of $\sigma(c)$ for every vortex $c \in C(\rho)$.
We say that $Q$ is \emph{grounded} if it uses edges of $\sigma(c_1)$ and $\sigma(c_2)$ for two distinct cells $c_1, c_2 \in C(\rho)$, or $Q$ is a path with both endpoints in $N(\rho)$.
If $Q$ is grounded we define the \emph{trace} of $Q$ as follows.
Let $P_1,\dots,P_k$ be distinct maximal subpaths of $Q$ such that $P_i$ is a subgraph of $\sigma(c)$ for some cell $c$.
Fix $i \in [k]$.
The maximality of $P_i$ implies that its endpoints are $\pi(n_1)$ and $\pi(n_2)$ for distinct nodes $n_1,n_2 \in N(\rho)$.
If $|N_\rho(c)| = 2$, let $L_i$ be the component of $\mathsf{bd}(c) - \{ n_1,n_2 \}$ selected by the tie-breaker, and if $|N_\rho(c)| = 3$, let $L_i$ be the component of $\mathsf{bd}(c) - \{ n_1,n_2 \}$ that is disjoint from $N_\rho(c)$.
We define $L_i'$ by pushing $L_i$ slightly so that it is disjoint from all cells in $C(\rho)$, while maintaining that the resulting curves intersect only at a common endpoint.
The \emph{trace} of $Q$ is defined to be $\bigcup_{i\in[k]} L_i'$.
If $Q$ is a cycle, its trace thus the homeomorphic image of the unit circle, and otherwise, it is an arc in $\Sigma$ with both endpoints in $N(\rho)$.

\paragraph{Aligned disks and flat subgraphs.}
Let $G$ be a graph and let $\rho = (\Gamma, \sigma, \pi)$ be a $\Sigma$-rendition of $G$. 
We say that a 2-connected subgraph $H$ of $G$ is \emph{grounded (in $\rho$)} if every cycle in $H$ is grounded and no vertex of $H$ is drawn by $\Gamma$ in a vortex of $\rho$.
A disk in $\Sigma$ is called \emph{$\rho$-aligned} if its boundary only intersects $\Gamma$ in nodes.
If $H$ is planar, we say that it is \emph{flat in $\rho$} if there exists a $\rho$-aligned disk $\Delta \subseteq \Sigma$ containing all cells $c \in C(\rho)$ with $E(\sigma(c)) \cap E(H) \neq \emptyset$ and $\Delta$ does not contain any vortices of $\Gamma$.

\paragraph{Societies.}
Let $\Omega$ be a cyclic ordering of the elements of some set which we denote by $V(\Omega)$.
A \emph{society} is a pair $(G,\Omega)$, where $G$ is a graph and $\Omega$ is a cyclic ordering with $V(\Omega)\subseteq V(G)$.
For a given set $S \subseteq V(\Omega)$ a vertex $s \in S$ is an \emph{endpoint} of $S$ if there exists a vertex $t \in V(\Omega) \setminus S$ that immediately precedes or succeeds $s$ in $\Omega$.
We call $S$ a \emph{segment} of $\Omega$ if $S$ has two or less endpoints.

Let $(G,\Omega)$ be a society and let $\Sigma$ be a surface with one boundary component $B$ homeomorphic to the unit circle.
A \emph{rendition} of $(G,\Omega)$ in $\Sigma$ is a $\Sigma$-rendition $\rho$ of $G$ such that the image under $\pi_{\rho}$ of $N(\rho) \cap B$ is $V(\Omega)$ and $\Omega$ is one of the two cyclic orderings of $V(\Omega)$ defined by the way the points of $\pi_{\rho}(V(\Omega))$ are arranged in the boundary $B$.

\section{A homogeneous wall}\label{sec:HomoMuralis}
In this section we perform steps 1 and 2 outlined in the intro.
This involves defining when we consider a wall to be flat or homogeneous, what strips within a flat wall are, and then proving that we can homogenise these strips appropriately.

\paragraph{Flat Walls.}
Let $n \ge 2$ be an integer.
Let $G$ be a graph, and let $M \subseteq G$ be an \((n \times n)\)-mesh.
We say that $M$ is a \emph{flat mesh} in $G$ if there exists a sphere-rendition $\rho$ of $G$ with a single vortex $c_0$ such that $M$ is flat in $\rho$ and the trace of the perimeter of $M$ in $\rho$ separates all vertices in $N(\delta) \cap V(M)$ from $c_0$.
We say that $\rho$ \emph{witnesses} the flatness of $M$.

\paragraph{Homogeneous meshes.}
Let $r\geq 3$ be an integer and $M$ be an $r$-mesh in $G$ and let $C_P$ be the perimeter of $M$.
We define the \emph{compass} of $M$, denoted as $\mathsf{compass}(M)$, as the union of $C_P$ and the unique $C_P$-bridge in $G$ that contains all vertices of $M - V(C_P)$.
Let $C\subseteq M$ be any cycle of $M$.
The \emph{compass} of $C$ is the union of $C$ together with all $C$-bridges in $G$ that are entirely contained in the compass of $M$.
We denote the compass of $C$ with respect to $M$ by $\mathsf{compass}_{M}(C)$.
The \emph{interior} of a cycle $C \subseteq M$, denoted by $\mathsf{int}_{M}(C)$, is the graph $\mathsf{compass}_{M,\rho}(C)-C$.
For both notions, we drop $M$ in the subscript if it is understood from the context.

\smallskip

Recall the definition of $q$-colorful graphs.
We allow for $q$ to be $0$.
In this case, for every $0$-colorful graph $(G,\chi)$ it holds that $\chi(v) = \emptyset$ for all $v\in V(G)$.
Moreover, in such a case we do not distinguish  between the graph $G$ and the $0$-colorful graph $(G,\chi)$.

Given an arbitrary $q$-colorful graph $(G,\chi)$, a set $X\subseteq V(G)$ and a subgraph $H\subseteq G$, we define
\begin{align*}
    \chi(X) \coloneq \bigcup_{v\in X} \chi(v),
\end{align*}
as well as $\chi(H) \coloneqq \chi (V(H))$.
Note that this should not be confused with the chromatic number of $H$, which is of no concern to us.
Finally, we write $(H,\chi)$ for the \emph{$q$-colorful subgraph} of $G$, where we implicitly restrict $\chi$ to the vertex set of $H$.
\medskip

\begin{figure}[ht]
    \centering
    \begin{tikzpicture}

        \pgfdeclarelayer{background}
		\pgfdeclarelayer{foreground}
			
		\pgfsetlayers{background,main,foreground}

        \begin{pgfonlayer}{background}
            \pgftext{\includegraphics[width=10cm]{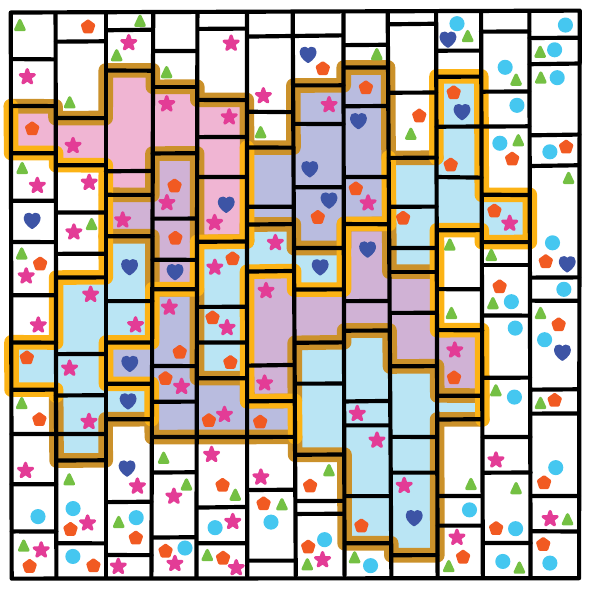}} at (C.center);
        \end{pgfonlayer}{background}
			
        \begin{pgfonlayer}{main}
        \node (C) [v:ghost] {};
            
        \end{pgfonlayer}{main}
        
        \begin{pgfonlayer}{foreground}
        \end{pgfonlayer}{foreground}

    \end{tikzpicture}
    \caption{An illustration of a flat mesh $M$ in a $5$-colorful graph $(G,\chi)$. The $5$ colours are represented as small shapes and each face of $M$ contains a symbol if and only if its interior carries the corresponding colour.
    Depicted as the \textcolor{ChromeYellow}{yellow} paths are vertical and horizontal paths of a $3$-mesh $M'\subseteq M$ such that $M'$ is homogeneous.}
    \label{fig:HomogeneousMesh}
\end{figure}

Let $q$ be a non-negative integer and $(G,\chi)$ be a $q$-colorful graph.
Let $r\geq 3$ be an integer and $M$ be an $r$-mesh in $G$.
We say that $M$ is \emph{homogeneous} in $(G,\chi)$ if for every cycle $C\subseteq M$ it holds that
\begin{align*}
 \chi( \mathsf{int} (C) ) = \chi( \mathsf{compass} (M) ).
\end{align*}
See \zcref{fig:HomogeneousMesh} for an illustration of a mesh that is not homogeneous but contains another mesh as a subgraph that is homogeneous for a subset of the present colours.

The fact that homogeneity as defined here is equivalent to homogeneity as defined in the introduction, where we solely discussed this for walls, is not hard to see.

\subsection{Sorting, packing, padding, and trimming strips}\label{subsec:PackingStrips}
We begin by proving a combinatorial lemma about how any mesh admits a collection of disjoint ``parts'' -- one may think of these parts as horizontal or vertical slices -- ensuring that each colour that appears within one of those parts will appear in many of them.

In addition, these parts will also be required to provide enough infrastructure to find our desired homogeneous mesh within.
This causes an additional problem as we now have to make sure that the colours do not appear \textsl{exclusively} within the subparts reserved for providing this infrastructure.
For this reason we will have to update our parts and their reserved infrastructure on the fly.

The definitions below provide a formalism to manage the data structure sketched above consisting of parts and their subparts with the reserved infrastructure within.

\paragraph{Strips in a mesh.}
Let $M=P_1\cup\cdots\cup P_n\cup Q_1\cup \cdots \cup Q_m$  be an $(n \times m)$-mesh.
Let $i\leq j \in [n]$.
The \emph{$(i,j)$-row} is the set $\{ P_h \colon h\in[i,j] \}$.
Similarly, for $i\leq j \in [m]$, the \emph{$(i,j)$-column} is the set $\{ Q_h \colon h\in[i,j] \}$.

Given an $(i,j)$-row of $M$ where $i<j$, its \emph{boundary} is defined as $Q_1^{i,j} \cup P_i \cup Q_m^{i,j} \cup P_j$, where $Q_x^{i,j}$ is the subpath of $Q_x$ between the endpoints of $P_i$ and $P_j$.
Analogously, given an $(i,j)$-column of $M$, its \emph{boundary} is defined as $P_1^{i,j} \cup Q_i \cup P_n^{i,j} \cup Q_j$ where $P_x^{i,j}$ is the subpath of $P_x$ between the endpoints of $Q_i$ and $Q_j$.

We say that a set $\mathcal{P} \subseteq \{ P_1,\dots,P_n\}$ is a \emph{row} if there are $i\leq j\in[n]$ with $j-i+1\geq 2$ such that $\mathcal{P}$ is the $(i,j)$-row of $M$.
Similarly, $\mathcal{Q}\subseteq \{ Q_1,\dots,Q_m\}$ is a \emph{column} if there are $i\leq j\in[m]$ with $j-i+1\geq 2$ such that $\mathcal{Q}$ is the $(i,j)$-column of $M$.
The notion of boundary for rows and columns is now straightforward from the notion of boundary for $(i,j)$-rows and $(i,j)$-columns.

\paragraph{Strips and frames.}
Now let $M$ be a flat $(n \times m)$-mesh in a graph $G$ witnessed by the rendition $\rho$.
Let $\mathcal{F}$ be a row or column of $M$.
The \emph{strip} of $\mathcal{F}$ is the compass of its boundary w.\@r.\@t.\@ $\rho$.
When given only a strip $S$ of $M$, we refer to its \emph{frame} as the column or row $\mathcal{F}$ of $M$ such that $S$ is the strip of $\mathcal{F}$.
See \zcref{fig:StripAndFrame} for an illustration.
We say that a strip $S$ is \emph{of type \texttt{X}} if $\texttt{X} \in \{\text{column}, \text{row} \}$ and its frame is a \texttt{X}.

\begin{figure}[ht]
    \centering
    \begin{tikzpicture}

        \pgfdeclarelayer{background}
		\pgfdeclarelayer{foreground}
			
		\pgfsetlayers{background,main,foreground}

        \begin{pgfonlayer}{background}
            \pgftext{\includegraphics[width=6cm]{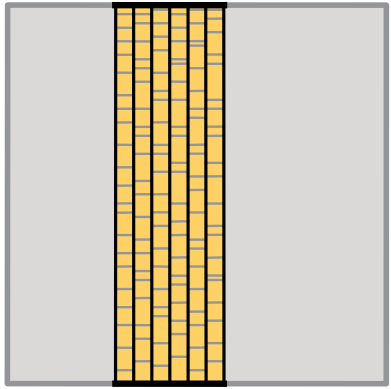}} at (C.center);
        \end{pgfonlayer}{background}
			
        \begin{pgfonlayer}{main}
        \node (C) [v:ghost] {};
            
        \end{pgfonlayer}{main}
        
        \begin{pgfonlayer}{foreground}
        \end{pgfonlayer}{foreground}

    \end{tikzpicture}
    \caption{An illustration of a strip $S$ of breadth $7$ in a mesh. The frame $
    \mathcal{F}$ of $S$ is formed by the $7$ vertical paths indicated in black. The \textcolor{ChromeYellow}{yellow} area indicates the compass of $\mathcal{F}$, that is the strip $S$.}
    \label{fig:StripAndFrame}
\end{figure}

The \emph{breadth} of a strip $S$ is the size of its frame.
That is, if $S$ is a strip of breadth $k$ and $\mathcal{F}$ is the frame of $S$, then $F$ consists of $k$ paths.

\paragraph{Strip packings.}
We say that two strips are \emph{disjoint} if their frames are vertex disjoint.

Any two disjoint strips must be of the same type.
That is, if $S_1$ and $S_2$ are disjoint strips, then their frames are either both rows or both columns.

A \emph{packing} of strips is a set $\mathcal{S}$ of pairwise disjoint strips.
Moreover, a \emph{packing of breadth $k$} is a packing $\mathcal{S}$ of strips such that every strip in $\mathcal{S}$ has breadth $k$.
\medskip

So far, we have introduced what we vaguely called the ``parts'' we wish to divide our mesh into: strip packings.
However, we have not yet given a formalisation for what we mean by the ``subparts for the reserved infrastructure''.
We discuss this below.

\paragraph{Padded strips.}
Let $k$ and $b$ be non-negative integers.
Let $S$ be a strip of breadth $2k+b$ with frame $\mathcal{F} = \{ F_i,\dots,F_{i+2k+b-1}\}$.
The \emph{$k$-padding} of $S$ are the paths $\mathsf{Pad}_k(\mathcal{F}) = \{F_i,\dots,F_{i+k-1},F_{i+k+b},\dots,F_{i+2k+b-1} \}$.
We refer to the paths $\{ F_i,\dots,F_{i+k-1}\}$ as the \emph{left buffer} and the paths $\{ F_{i+k+b},\dots,F_{i+2k+b-1}\}$ as the \emph{right buffer} of $\mathsf{Pad}_k(\mathcal{F})$.
The strip of the right buffer of $\mathsf{Pad}_k(\mathcal{F})$ is called the \emph{right $k$-buffer} of $S$ and the strip of the left buffer of $\mathsf{Pad}_k(\mathcal{F})$ is called the \emph{left $k$-buffer} of $S$.
The strip of $\{ F_{i+k},\dots,F_{i+k+b-1}\}$ is called the \emph{$k$-core} -- or \emph{core} if $k$ is understood from the context -- of $S$.
See \zcref{fig:PaddedStrips} for an illustration.

\begin{figure}[ht]
    \centering
    \begin{tikzpicture}

        \pgfdeclarelayer{background}
		\pgfdeclarelayer{foreground}
			
		\pgfsetlayers{background,main,foreground}

        \begin{pgfonlayer}{background}
            \pgftext{\includegraphics[width=8cm]{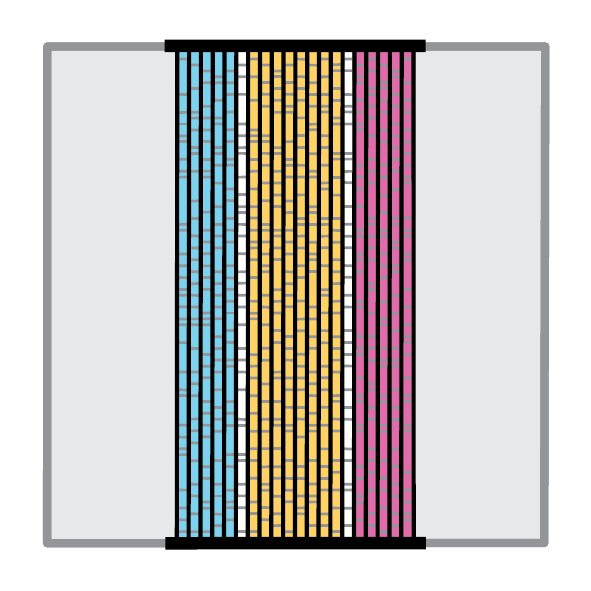}} at (C.center);
        \end{pgfonlayer}{background}
			
        \begin{pgfonlayer}{main}
        \node (C) [v:ghost] {};
            
        \end{pgfonlayer}{main}
        
        \begin{pgfonlayer}{foreground}
        \end{pgfonlayer}{foreground}

    \end{tikzpicture}
    \caption{A strip $S$ (the union of the \textcolor{CornflowerBlue}{blue}, the \textcolor{ChromeYellow}{yellow}, and the \textcolor{BrilliantRose}{magenta} are together with the two narrow gaps between them) with its frame $\mathcal{F}$ (the black horizontal paths). The $6$-core of $S$ is depicted in \textcolor{ChromeYellow}{yellow}. Moreover, the left buffer of $\mathsf{Pad}_k(\mathcal{F})$ is depicted in \textcolor{CornflowerBlue}{blue} while the right buffer is depicted in \textcolor{BrilliantRose}{magenta}.}
    \label{fig:PaddedStrips}
\end{figure}

We say that a packing $\mathcal{S}$ of strips is a \emph{$(k,b)$-padded packing} if it is a packing of breadth $\geq 2k+b$ where we implicitly assume that each strip $S\in \mathcal{S}$ with frame $\mathcal{F}$ comes divided into three strips, namely $L_S,C_S,R_S$ where $L_S$ is the strip of the left buffer of $\mathsf{Pad}_k(\mathcal{F})$, $C_S$ is the $k$-core, and $R_S$ is the strip of the right buffer of $\mathsf{Pad}_k(\mathcal{F})$.
\smallskip

Please note that the notions ``left'' and ``right'' here are chosen almost arbitrarily in order to give some form of orientation that mostly aligns with our figures.
Technically, ``left'' and ``right'' only make sense for columns and even here one might prefer to refer to them as the ``small'' and the ``big'' buffer determined by the values of their indices.
However, small and big might be misunderstood in other ways and a simple rotation of $M$ can always ensure that the notions of left and right make sense.
\smallskip

Lastly, we require a smooth update process that allows us to shed a fixed amount of the padding from all strips in a $(k,b)$-padded packing.
This might become necessary for the following reason:
Initially, we will divide the entire mesh into a large packing of broad strips.
We then make a selection of these strips such that every colour occurring in one of them occurs in many.
Finally, we fix a first padding.
By fixing this padding, however, it might happen that there exists a colour $i$ that was previously abundant and now only occurs in the core of few strips.
If this is the case we wish to remove those few strips entirely from the packing and shed off the previously selected padding from all other strips, thereby removing all occurrences of the colour $i$ within the remaining and trimmed strips.
Eventually, this process will need to come to a halt since we assume that the set of colours is bounded in size.

The following formalises this trimming process.

\paragraph{Trimming a padded packing.}
Let $k$ and $b\geq 1$ be non-negative integers.

Let $S$ be a strip of breadth $2k+b$ with frame $\mathcal{D}$ together with a $k$-padding $\mathsf{Pad}_k(\mathcal{F})$.
We say that a strip $S'$ is a \emph{$k$-trimming} of $S$ if the frame $\mathcal{F}'$ of $S'$ is equal to $\mathcal{F} \setminus \mathsf{Pad}_k(\mathcal{F})$.

Similarly, for a $(k,b)$-padded packing $\mathcal{S}$ of strips, the \emph{$k$-trimming} of $\mathcal{S}$ is the packing $\mathcal{S}'$ of breadth $b$ consisting precisely of the $k$-trimmings of the strips from $\mathcal{S}$.
\smallskip

If $S'$ is the $k$-trimming of a strip $S$ of breadth $2k+b$, then $S'$ is the $k$-core of $S$.
\medskip

We now have everything in place to state and prove the main lemma of this subsection.

\begin{lemma}\label{lemma:StripRepresentation}
Let $d$, $r$, $p$, $b$, $t\geq 2$, $x\geq 1$, and $q$ be non-negative integers and let $\texttt{X} \in \{\text{column},\text{row} \}$.
Let $(G,\chi)$ be a $q$-colorful graph, $d\geq ( q(r-1) + x ) \cdot ( 2(q+1)p + b )$ and $M$ be a $(d \times t)$-mesh in $G$ if $\texttt{X}=\text{column}$ and a $(t \times d)$ otherwise such that $M$ is flat witnessed by the rendition $\rho$.

Then there exists a, possibly empty, set $I\subseteq [q]$ and a $(p,b)$-padded packing $\mathcal{S}$ of strips such that
\begin{enumerate}
    \item $|\mathcal{S}| \geq |I|(r-1)+ x $,
    \item the breadth of each strip of $\mathcal{S}$ is at least $2(|I|+1)p + b$,
    \item for every $S\in \mathcal{S}$ it holds that $\chi(S) \subseteq I$,
    \item for every $i\in I$ there exist at least $r$ strips $S \in \mathcal{S}$ such that $i\in\chi (C_S)$ where $C_S$ denotes the $p$-core of $S$, and
    \item every strip in $\mathcal{S}$ is of type \texttt{X}.
\end{enumerate}
Moreover, there exists an algorithm that takes $p$, $b$, $r$, $(G,\chi)$, $M$, and $\rho$ as input and finds $\mathcal{S}$ in time $\mathbf{poly}(d)|E(G)|$.
\end{lemma}

\begin{proof}
Without loss of generality, we may assume $\texttt{X}=\text{column}$ as otherwise switching the roles of columns and rows in $M$ transforms the problem into the desired form.

Let us begin with the very first packing of strips.
With $d \geq ( q(r-1) + x ) \cdot ( 2(q+1)p + b )$ we may partition the first $( q(r-1) + x ) \cdot ( 2(q+1)p + b )$ vertical paths of $M$ into a total of $q(r-1) + x$ columns, each of breadth $2(q+1)p + b$.
This results in a packing $\mathcal{S}_0$ of breadth $2(q+1)p + b$ of $q(r-1) + x$ strips of type column.
Let $I_0 \coloneqq [q]$.

Our algorithm proceeds by executing the following subroutines.
Each step takes as input three \emph{parameters}, that is the number $q$ of colours, the set $I_0'$ of all $q'$ available colours, and a packing $\mathcal{S}_0'$ containing at least $q'(r-1) + x$ strips, each of breadth at least $2(q'+1)p + b$ such that $\chi(S) \subseteq I_0'$ for each $S\in\mathcal{S}_0'$.
\textbf{Trim} has the additional requirement that for every $j\in I_0'$ there are at least $r$ strips $S$ in $\mathcal{S}_0'$ with $j\in\chi(S)$.
This requirement is not necessary for \textbf{Sort}.
\medskip

We start our procedure by calling \textbf{Sort} and in this first call, we enter the input $I_0' \coloneqq I_0$, $q' \coloneqq q$, and $\mathcal{S}_0' \coloneqq \mathcal{S}_0$.
\medskip

\textbf{Sort:}
The input are a set $I_0'$ of $q'$ available colours and a packing $\mathcal{S}_0'$ of at least $q'(r-1) + x$ strips, each of breadth at least $2(q'+1)p + b$ such that $\chi(S) \subseteq I_0'$ for each $S\in\mathcal{S}_0'$.
\smallskip

If for every $i\in I_0'$ there are at least $r$ strips $S\in\mathcal{S}_0'$ such that $i\in\chi(S)$ we immediately proceed to \textbf{Trim} with the same input.
\smallskip

Otherwise, suppose there exists $j\in I_0'$ such that there are at most $r-1$ strips $S$ in $\mathcal{S}_0'$ with $j\in \chi(S)$ and let $\mathcal{Z}_1$ be the collection of all such strips.
Pick the smallest such $j$ and set
\begin{align*}
    I_1 & \coloneqq I_0' \setminus \{ j\}, \text{ and}\\
    \mathcal{S}_1 & \coloneqq \mathcal{S}_0' \setminus \mathcal{Z}_1.
\end{align*}
We now have that
\begin{align*}
    |I_1| & = q'-1\text{ and}\\
    |\mathcal{S}_1| & \geq q'(r-1) + x - (r-1)\\
    & \geq (q'-1)(r-1) +x\\
    & = |I_1|(r-1) +x.
\end{align*}
Moreover, there now does not exist a strip $S\in\mathcal{S}_1$ such that $j\in\chi(S)$.

By a straightforward induction, we may repeat this process up to $q'-1$ additional times until we reach some integer $i\in [q']$ together with a set $I_i \subseteq I_1$ where $|I_i|=q'-i$ and a packing $\mathcal{S}_i \subseteq \mathcal{S}_1$ of breadth $2(q'+1)p + b$ of at least $(q'-i)(r-1) + x = |I_i|(r-1) + x$ strips such that
\begin{enumerate}
    \item for every $j\in I_i$ there are at least $r$ strips $S\in\mathcal{S}_i$ such that $j\in\chi(S)$, and
    \item $\chi(S) \subseteq I_i$ for every $S\in\mathcal{S}_i$.
\end{enumerate}

If $i < q'$ we may now proceed to \textbf{Trim} with parameters $q' \coloneqq |I_i| = q'-i$, $I'_0 \coloneqq I_i$, and $\mathcal{S}'_0 \coloneqq \mathcal{S}_i$.
\smallskip

In case $i=q'$ we have that $I_i = \emptyset$.
This means that $\chi(S) = \emptyset$ for all $S \in \mathcal{S}_i$.
In this case we may select an arbitrary subsets of $x$ strips $\mathcal{S}^\star \subseteq \mathcal{S}_i$, since we have $|\mathcal{S}_i| \geq |I_i|(r-1) + x$ and set $\mathcal{S} \coloneqq \mathcal{S}^\star$.
Since each strip in $\mathcal{S}$ is of breadth $2(q'+1)p + b$, we may treat $\mathcal{S}$ as a $(p,b)$-padded packing of size one and since the strips in $\mathcal{S}$ carry no colour, for each strip $S \in \mathcal{S}$ both $p$-buffers carry the same set of colours as the core of $S$ which is equal to $\chi(S) = \emptyset$.
Hence, $\mathcal{S}$ satisfies all requirements of our lemma and we are done in this case by \textbf{returning} $\mathcal{S}$.
\medskip

We now proceed to \textbf{Trim}.
This step may send us back to \textbf{Sort} with reduced parameters.
\medskip

\textbf{Trim:}
The input are a set $I_0'$ of $q'$ available colours and a packing $\mathcal{S}_0'$ of at least $q'(r-1) + 1$ strips, each of breadth at least $2(q'+1)p + b$ such that $\chi(S) \subseteq I_0'$ for each $S\in\mathcal{S}_0'$.
Additionally, we require for the input that for every $j\in I_0'$ there are at least $r$ strips $S$ in $\mathcal{S}_0'$ with $j\in\chi(S)$.
\smallskip

Let us now consider $\mathcal{S}_0'$ as a $(p,b)$-padded packing.
That is, each $S\in\mathcal{S}_0'$ is now divided into the strips $L_S$ and $R_S$ which are the left and the right buffer of $\mathsf{Pad}_p(\mathcal{F}_S)$ where $\mathcal{F}_S$ is the frame of $S$, and its $p$-core $C_S$.
\smallskip

If for every $j\in I_0'$ there exist at least $r$ strips $S\in\mathcal{S}_0$ such that $j\in\chi(C_S)$, we have now reached the desired outcome of the lemma and may stop by returning $\mathcal{S} \coloneqq \mathcal{S}_0'$.
\smallskip

Otherwise let us select $j\in I_0'$ to be the smallest colour such that there are at most $r-1$ strips $S\in\mathcal{S}_0'$ with $j\in \chi(S)$.
Let $\mathcal{Z}\subseteq \mathcal{S}_0$ be the collection of all $S\in\mathcal{S}_0'$ such that $j\in \chi(C_S)$.
We set
\begin{align*}
    \mathcal{S}_1 & \coloneqq \mathcal{S}_0'\setminus \mathcal{Z}\text{ and}\\
    I_1 & \coloneqq I_0' \setminus \{ j\}.
\end{align*}
It follows that $|\mathcal{S}_1| \geq (q'-1)(r-1) + x$ and so far, every strip in $\mathcal{S}_1$ has breadth at least $2(q'+1)p + b$.

Now let $\mathcal{S}_2$ be the $p$-trimming of $\mathcal{S}_1$.
Still we have that $|\mathcal{S}_2| \geq (q'-1)(r-1) + x$ but the breadth of the strips in $\mathcal{S}_2$ might have dropped.
However, we have deleted precisely $2p$ paths from the frame of each strip in $\mathcal{S}_1$ and so the breadth of the strips in $\mathcal{S}_2$ is still at least $2q'p + b$.

In this state, $\mathcal{S}_2$ does not necessarily meet the requirements of \textbf{Trim}.
To check for this and, in case it fails, to adjust the input accordingly, we now call \textbf{Sort} again, this time with parameters $q' \coloneqq q'-1$, $I_0' \coloneqq I_1$, and $\mathsf{S}_0' \coloneqq \mathsf{S}_2$.
\bigskip

Every time \textbf{Trim} calls \textbf{Sort} again, the parameter $q'$ decreases by $1$.
Hence, \textbf{Trim} may called at most $q+1$ times in total where we may be sure that in the $(q+1)$st call, the set $I_0'$ is empty and therefore \textbf{Trim} will terminate the algorithm.

Since all possible end states of \textbf{Sort} are either termination or calling \textbf{Trim} this implies that our algorithm will correctly terminate and return the desired packing $\mathcal{S}$ after at most $q+1$ iterations of \textbf{Sort} and $q+1$ iterations of \textbf{Trim}.

For all executions of steps \textbf{Sort} and \textbf{Trim}, we only need to search the graph a constant number of times to check for the colours carried by each strip -- and possibly their cores -- as well as to determine the newly assigned buffer and core strips.
Since these searches can be done in linear time, this completes the proof due to $d$ being a common upper bound on all other involved quantities.
\end{proof}

\subsection{Tilings out of strips}\label{subsec:Tilings}

In \zcref{subsec:PackingStrips} we have restricted ourselves to packing \textsl{one} type of strips.
However, we need to deal with an additional difficulty.
Suppose we have already applied \zcref{lemma:StripRepresentation} and obtained a packing $\mathcal{C}$ of strips as specified in the lemma for strips of type column.
If we now run the same algorithm to obtain a packing $\mathcal{T}$ of strips of type row, we might impact the properties of our column-type strips.
That is, it might happen that for some $C\in \mathcal{C}$ and some colour $i\in[q]$, all occurrences of $i$ in $\chi(C)$ lie within the paddings -- more precisely the left and right buffers -- of the strips from $\mathcal{R}$.
As a result, this colour will no longer be available -- from the viewing point of $C$ -- outside of the reserved infrastructure.
In such a situation, however, it is clear that $i$ must also belong to the set of colours represented within the strips of $\mathcal{R}$ which will allow us to infer that this situation does not lead to us losing access to some colour by tweaking the arguments from the proof of \zcref{lemma:StripRepresentation} slightly.
However, to locate and ensure the abundance of all colours caught within the strips of $\mathcal{C}$ and $\mathcal{R}$, we need to introduce another level of data structure on top of our padded strips: \textsl{Tiles}.

Intuitively, the union of the strips from $\mathcal{C}$ and $\mathcal{R}$ may be split into four types of ``regions'',
\begin{enumerate}
    \item the \textsl{reserved infrastructure}, that is the union of all left and right buffers of all strips involved,
    \item parts of the cores of strips from $\mathcal{C}$ that are disjoint from the strips of $\mathcal{R}$,
    \item parts of the cores of strips from $\mathcal{R}$ that are disjoint from the strips of $\mathcal{C}$, and
    \item those areas where the cores of $\mathcal{C}$ and $\mathcal{R}$ coincide.
\end{enumerate}
In \zcref{fig:IntersectingStrips} we provide a sketch of such a situation.
\begin{figure}[ht]
    \centering
    \begin{tikzpicture}

        \pgfdeclarelayer{background}
		\pgfdeclarelayer{foreground}
			
		\pgfsetlayers{background,main,foreground}

        \begin{pgfonlayer}{background}
            \pgftext{\includegraphics[width=8cm]{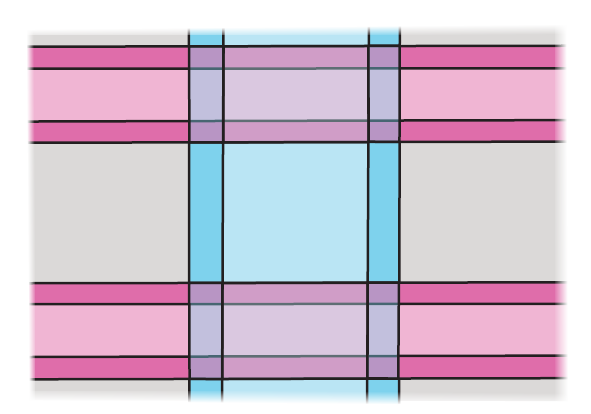}} at (C.center);
        \end{pgfonlayer}{background}
			
        \begin{pgfonlayer}{main}
        \node (C) [v:ghost] {};

        \node (X) [v:ghost,position=0:0mm from C] {$X$};

        \node (Y1) [v:ghost,position=90:15mm from C] {$Y_1$};
        \node (Y2) [v:ghost,position=270:17mm from C] {$Y_2$};

        \end{pgfonlayer}{main}
        
        \begin{pgfonlayer}{foreground}
        \end{pgfonlayer}{foreground}

    \end{tikzpicture}
    \caption{Two disjoint strips of type row (in \textcolor{BrilliantRose}{magenta}) and a single strip $C$ of type column (in \textcolor{CornflowerBlue}{blue}). The darker parts represent the padding of all strips while the lighter areas are the cores. At the centre of the sketch is a region in light \textcolor{CornflowerBlue}{blue} ($X$) enclosed by the padding of the column and the padding of both rows.
    Above and below this region we find the intersections of the core of $C$ with the cores of the two rows ($Y_1$ and $Y_2$).}
    \label{fig:IntersectingStrips}
\end{figure}
The goal of this subsection is to provide a formalism for the concept of tiles and to prove that every colour that is abundant in $\mathcal{C}$ or $\mathcal{R}$ may be forced to also be abundant amongst the tiles.

\paragraph{Tiles of a strip.}
Let $n,m \geq 2$ be integers and $\texttt{X},\texttt{Y}\in\{ \text{column},\text{row}\}$ be distinct.
Let $G$ be a graph and $M$ be an $(n \times m)$-mesh that is flat in $G$ witnessed by $\rho$.
Let $k,b$ be two non-negative integers and let $\mathcal{Z}$ be a $(k,b)$-padded packing of strips of type \texttt{X} where $\mathcal{F}_Z$ denotes the frame of $Z\in\mathcal{Z}$.
Finally, let $d,p$ be another two non-negative integers, let $S$ be a strip of type \texttt{Y} with breadth $2p+d$ and the frame $\mathcal{F}$, let $L$ be the rightmost path of the left buffer of $\mathsf{Pad}_p(\mathcal{F})$, and let $R$ be the leftmost path of the right buffer of $\mathsf{Pad}_p(\mathcal{F})$.

Every $Z\in\mathcal{Z}$ intersects $S$, indeed, it intersects every single path of $\mathcal{F}$.
Moreover, the strips in $\mathcal{Z}$ are naturally ordered as follows.
For two strips $Z_1,Z_2\in\mathcal{Z}$ we write $Z_1 \leq Z_2$ if and only if the right buffer of $Z_1$ separates $Z_2$ and the left buffer of $Z_1$ within the compass of $M$.

In other words, we assume the strips in $\mathcal{Z}$ to be ordered with respect to their occurrence in $M$ from ``left to right''.
As before the notion of ``left to right'' in this context appears to imply that the strips of $\mathcal{Z}$ are of type column, but as for buffers, by rotating $M$ we may also assume the rows to have a well-defined ``left'' and ``right''.

We say that two strips $Z_1 \neq Z_2\in\mathcal{Z}$ are \emph{consecutive} if $Z_1 \leq Z_2$ and there does not exist $Z\in\mathcal{Z}\setminus \{ Z_1,Z_2\}$ such that $Z_1 \leq Z \leq Z_2$.

Now fix $Z\in\mathcal{Z}$.
Then for $Z$ there exist four paths of $\mathsf{Pad}_{k}(\mathcal{F}_Z)$ we are interested in.
Those are
\begin{itemize}
    \item the leftmost path $L^1_Z$ of the left buffer of $\mathsf{Pad}_{k}(\mathcal{F}_Z)$,
    \item the rightmost path $L^2_Z$ of the left buffer of $\mathsf{Pad}_{k}(\mathcal{F}_Z)$,
    \item the leftmost path $R^1_Z$ of the right buffer of $\mathsf{Pad}_{k}(\mathcal{F}_Z)$, and
    \item the rightmost path $R^2_Z$ of the right buffer of $\mathsf{Pad}_{k}(\mathcal{F}_Z)$.
\end{itemize}

For a fixed $Z\in\mathcal{Z}$ let $D_{Z,S}$ be the unique cycle in $L \cup R \cup L^2_Z \cup R^1_Z$.
The \emph{$(Z,k)$-tile} of $S$ is the graph $\mathsf{int}(D_{Z,S})$.
\smallskip

For two consecutive $Z_1,Z_2 \in\mathcal{Z}$ let $D_{Z_1,Z_2,S}$ be the unique cycle in $L \cup R \cup R^2_{Z_1} \cup L^1_{Z_2}$.
The \emph{$(Z_1,Z_2,k)$-tile} of $S$ is the graph $\mathsf{int}(D_{Z_1,Z_2,S})$.
See \zcref{fig:Tiles} for an illustration.
\smallskip

\begin{figure}[ht]
    \centering
    \begin{tikzpicture}

        \pgfdeclarelayer{background}
		\pgfdeclarelayer{foreground}
			
		\pgfsetlayers{background,main,foreground}

        \begin{pgfonlayer}{background}
            \pgftext{\includegraphics[width=12cm]{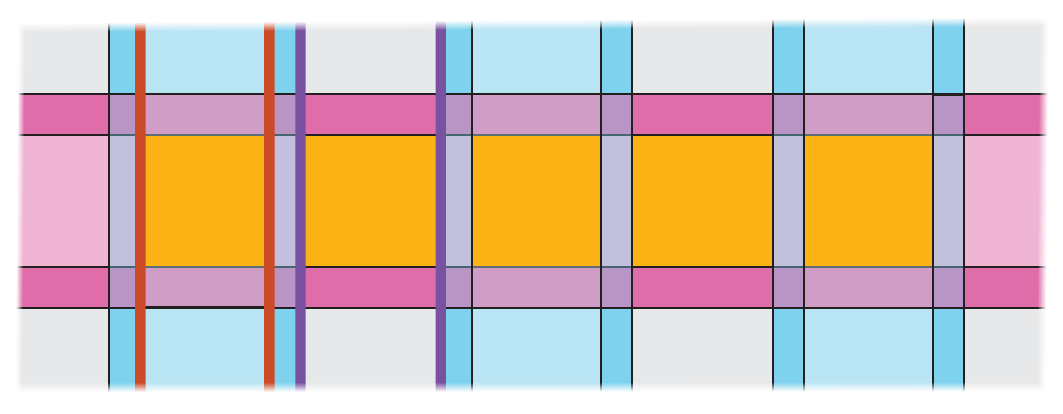}} at (C.center);
        \end{pgfonlayer}{background}
			
        \begin{pgfonlayer}{main}
        \node (C) [v:ghost] {};

        \node (T1) [v:ghost,position=180:36mm from C] {$T_1$};
        \node (T2) [v:ghost,position=180:18mm from C] {$T_2$};
        \node (T3) [v:ghost,position=0:0mm from C] {$T_3$};
        \node (T4) [v:ghost,position=0:19mm from C] {$T_4$};
        \node (T5) [v:ghost,position=0:38mm from C] {$T_5$};
            
        \end{pgfonlayer}{main}
        
        \begin{pgfonlayer}{foreground}
        \end{pgfonlayer}{foreground}

    \end{tikzpicture}
    \caption{The set $\mathsf{Tiles}_{\mathcal{C}}(R)$ (depicted in \textcolor{ChromeYellow}{yellow}) for a fixed row R and with respect to a packing of type column (depicted in \textcolor{CornflowerBlue}{blue}).
    The tiles $T_1$, $T_3$, and $T_5$ individually sit between the two innermost paths of the padding of some $C \in\mathcal{C}$ (for $T_1$ these paths are depicted in dark \textcolor{Flame}{orange}).
    The tiles $T_1$ and $T_2$ are of the form $D_{Z_i,Z_{i+1},R}$ for some consecutive $Z_i,Z_{i+1}\in \mathcal{C}$.
    Both each of them is bordered on the left by the rightmost path of the padding of $Z_i$ and on the right by the leftmost path of the padding of $Z_{i+1}$ (for $T_2$ these paths are depicted in \textcolor{Amethyst}{purple}).}
    \label{fig:Tiles}
\end{figure}
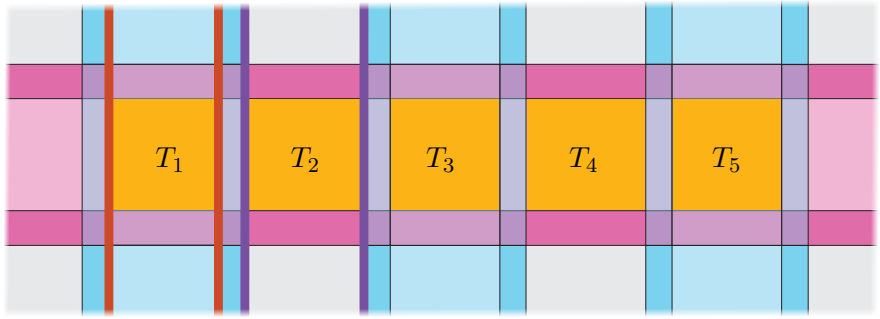

Finally, the set of all \emph{$(\mathcal{Z},k)$-tiles} of $S$, denoted by $\mathsf{Tiles}_{\mathcal{Z}}^k(S)$, is the collection of all $(Z,k)$-tiles of $S$ for $Z\in\mathcal{Z}$ and all $(Z_1,Z_2,k)$ tiles of $S$ for consecutive $Z_1,Z_2\in\mathcal{Z}$.
\smallskip

Tiles have two important features:
\begin{enumerate}
    \item They are either entirely private to one strip, those are the $(Z_1,Z_2,k)$-tiles of $S$, or they are shared by two strips, that is, the $(Z,k)$-tile of $S$ is also the $(S,p)$-tile of $Z$.
    \item Each tile is \textsl{entirely surrounded} by the union of the padding of $S$ and the paddings of either $Z_1$ and $Z_2$ or the padding of $Z$.
\end{enumerate}
The second feature in particular is important for our final goal to construct a homogeneous wall out of the infrastructure provided by the padding.
In turn, the $p$-core $C_S$ of the strip $S$ is naturally partitioned into three types of areas:
\begin{itemize}
    \item the intersection of $C_S$ with the paddings of the strips from $\mathcal{Z}$,
    \item the tiles, and
    \item (up to) two additional ``tiles'' where $C_S$ meets the perimeter of $M$.
\end{itemize}
See \zcref{fig:EndTiles} for an illustration.

\begin{figure}[ht]
    \centering
    \begin{tikzpicture}

        \pgfdeclarelayer{background}
		\pgfdeclarelayer{foreground}
			
		\pgfsetlayers{background,main,foreground}

        \begin{pgfonlayer}{background}
            \pgftext{\includegraphics[width=7cm]{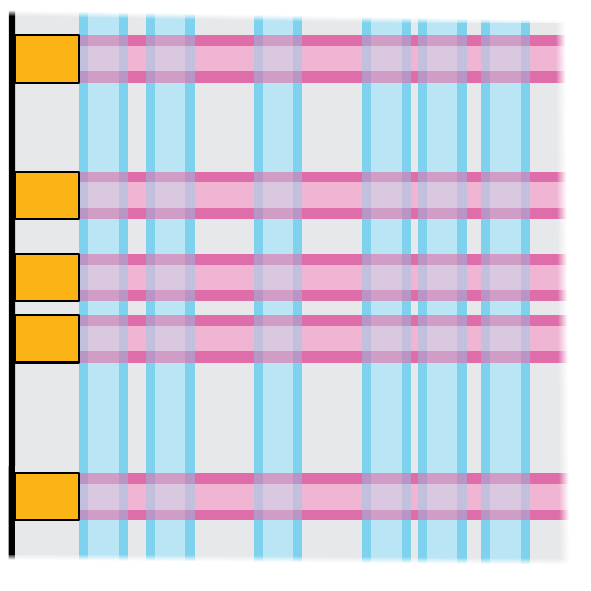}} at (C.center);
        \end{pgfonlayer}{background}
			
        \begin{pgfonlayer}{main}
        \node (C) [v:ghost] {};
        \node (Top) [v:ghost,position=90:28mm from C] {};
        \node (L) [v:ghost,position=180:30mm from Top] {};

        \node (E1) [v:ghost,position=0:0mm from L] {$E_1$};
        \node (E2) [v:ghost,position=270:16.3mm from E1] {$E_2$};
        \node (E3) [v:ghost,position=270:9.6mm from E2] {$E_3$};
        \node (E4) [v:ghost,position=270:7.5mm from E3] {$E_4$};
        \node (E5) [v:ghost,position=270:18.5mm from E4] {$E_5$};
            
        \end{pgfonlayer}{main}
        
        \begin{pgfonlayer}{foreground}
        \end{pgfonlayer}{foreground}

    \end{tikzpicture}
    \caption{An illustration of part of a flat mesh together with a packing of columns and a packing of rows. The \textcolor{ChromeYellow}{yellow} regions $E_1,\dots,E_5$ are ``end-tiles'', i.\@e.\@ regions contained in the cores of some of the rows -- in this example -- that are not completely surrounded by paddings and therefore not considered tiles.}
    \label{fig:EndTiles}
\end{figure}

These last to ``end-tiles'' are \textsl{not} surrounded by the infrastructure provided by the padding.
This might be dangerous because it could happen that all colours carried by the core of $S$ are accumulated in -- or close to -- the perimeter of $M$ which will make it hard to pick up these colours later on.
So we will need to crop these pieces and make sure that the entire core of $S$ is partitioned into tiles and padding.
As with the trimming operation, this might however lead to the loss of another colour.
Moreover, to deal with this, we are also required to include another round of trimmings as the colours carried by $S$ might all be \textsl{everywhere} in the padding and at the perimeter of $M$ within the core.
To ensure that this does not result in an expensive loop, our goal is to envelope the algorithm from \zcref{lemma:StripRepresentation} in another algorithm, bringing two additions to the table:
\begin{enumerate}
    \item We now create two strip packings simultaneously, one for the columns and the other for the rows, and
    \item we include one additional subroutine that ``crops'' $M$ towards the boundaries defined by the two packings and then scans for newly occurred losses of colours.
\end{enumerate}
Before we jump into the new lemma, let us define the ``cropping'' operation mentioned above.

\paragraph{Cropping meshes and strips.}
Let $M$ be a flat mesh in a graph $G$ witnessed by $\rho$ and let $\mathcal{C}$ be a $(k,b)$-padded packing of strips of type column as well as $\mathcal{R}$ be a $(d,p)$-padded packing of strips of type row.

For any $\mathcal{X}\in \{ \mathcal{C},\mathcal{R}\}$ we define $\mathsf{Crop}(M,\mathcal{X})$ to denote the smallest submesh of $M$ fully containing all paths of the frames of the strips in $\mathcal{X}$.
Moreover, for $\mathcal{Y}\in \{ \mathcal{C},\mathcal{R}\}$ we denote by $\mathsf{Crop}(\mathcal{Y},\mathcal{X})$ the packing of strips of $\mathsf{Crop}(M,\mathcal{X})$ that is the collection of all strips defined by restricting the frame $\mathcal{F}_Y$ of each $Y\in\mathcal{Y}$ to the mesh $\mathsf{Crop}(M,\mathcal{X})$.

We have $M' \coloneqq \mathsf{Crop}(\mathsf{Crop}(M,\mathcal{C}),\mathcal{R}) = \mathsf{Crop}(\mathsf{Crop}(M,\mathcal{R}),\mathcal{C})$.
Moreover, the perimeter of $M'$ consists entirely of subpaths of the outer-most paths of the paddings of $\mathcal{C}$ and $\mathcal{R}$.
That is, in $M'$ we have indeed gotten rid of those additional ``end-tiles'' mentioned above.
See \zcref{fig:Cropping} for an illustration.

\begin{figure}[ht]
    \centering
    \begin{tikzpicture}

        \pgfdeclarelayer{background}
		\pgfdeclarelayer{foreground}
			
		\pgfsetlayers{background,main,foreground}

        \begin{pgfonlayer}{background}
            \pgftext{\includegraphics[width=6cm]{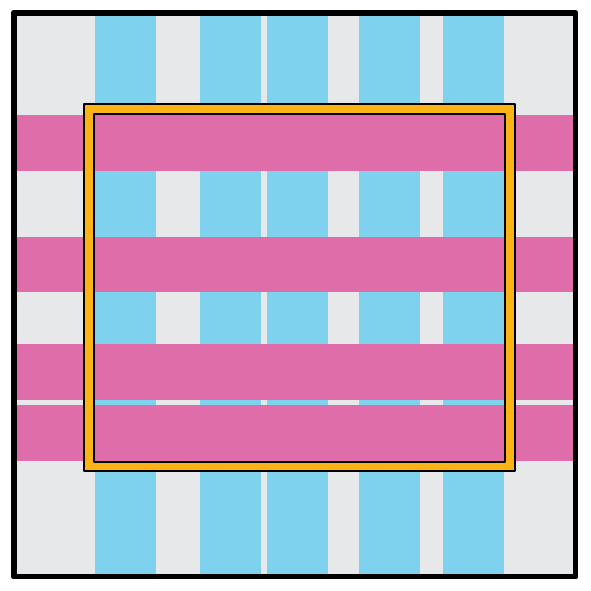}} at (C.center);
        \end{pgfonlayer}{background}
			
        \begin{pgfonlayer}{main}
        \node (C) [v:ghost] {};
            
        \end{pgfonlayer}{main}
        
        \begin{pgfonlayer}{foreground}
        \end{pgfonlayer}{foreground}

    \end{tikzpicture}
    \caption{A flat mesh $M$ together with a packing $\mathcal{C}$ of strips of type column and a packing $\mathcal{R}$ of strips of type row.
    The submesh $M'$ whose perimeter is indicated in \textcolor{ChromeYellow}{yellow} is precisely the mesh $\mathsf{Crop}(\mathsf{Crop}(M,\mathcal{C}),\mathcal{R})$ obtained by cropping $M$ along both packings.}
    \label{fig:Cropping}
\end{figure}

In some situations, we would also like to lift a strip from a submesh of $M$ to a strip of the entire mesh.
Let $M'$ be a submesh of $M$ and let $S'$ be a strip of $M'$ with frame $\mathcal{F}'$.
By $\mathsf{Lift}(M,S)$ we denote the strip $S$ of $M$ with frame $\mathcal{F}$ where $F\in \mathcal{F}$ if and only if there exists $F'\in \mathcal{F}'$ such that $F' \subseteq F$, and $S$ and $S'$ are of the same type.
Similarly, if $\mathcal{S}$ is a packing of strips of some submesh of $M$, we denote by $\mathsf{Lift}(M,\mathcal{S})$ the set $\{ \mathsf{Lift}(M,S) \colon S\in\mathcal{S} \}$.

The first issue we have to deal with is that by cropping we might lose access to additional colours.
This is because, while an application of \zcref{lemma:StripRepresentation} allows us to ensure that every colour carried by some strips also appears in the cores of many strips, it does not provide more information on the precise location of this colour.
Indeed, it could be that the majority of vertices carrying a fixed colour appears in the ``end-tiles'' which would cropped and therefore forgotten.
There is a second danger we need to take care off.
That is, even if the problem above can be handled, it could still happen that most occurrences of colour $i$ in the cores of, say, the rows are within the left buffer of the leftmost column and the right buffer of the rightmost column.
This case is \textsl{almost} the same situation as the one with the ``end-tiles'' and can be taken care off in essentially the same way.

Let $M$ be a flat mesh with a $(p,b)$-padded packing $\mathcal{C}$ of columns and a $(k,h)$-padded packing $\mathcal{R}$ of rows such that $M = \mathsf{Crop}(\mathsf{Crop}(M,\mathcal{C}),\mathcal{R})$.
Then let $B_{\mathcal{C}}$ be the union of the leftmost strip of $\mathcal{C}$ together with the rightmost strip of $\mathcal{C}$.
Similarly let $B_{\mathcal{R}}$ be the union of the leftmost strip in $\mathcal{R}$ and the rightmost strip in $\mathcal{R}$.
We call $B= B_{\mathcal{R}} \cup B_{\mathcal{C}}$ the \emph{boundary} of $\mathcal{R}$ and $\mathcal{C}$.

To account for the difficulties mentioned above, we begin with a strengthening of \zcref{lemma:StripRepresentation}.

\begin{lemma}\label{lemma:Cropping}
Let $d$, $r$, $p$, $b$, and $q$ be non-negative integers.
Let $(G,\chi)$ be a $q$-colorful graph, $d\geq ( q(r+1) + 2q + 1 ) \cdot ( 2(q+1)p + b )$ and $M$ be a $(d \times d)$-mesh in $G$ such that $M$ is flat witnessed by the rendition $\rho$.

Then there exist possibly empty sets $I_{\mathcal{C}},I_{\mathcal{R}}\subseteq [q]$ and two $(p,b)$-padded packings $\mathcal{C}$ of type column and $\mathcal{R}$ of type row such that if $\widehat{M}\coloneqq \mathsf{Crop}(\mathsf{Crop}(M,\mathcal{C}),\mathcal{R})$, $\widehat{\mathcal{C}} \coloneqq \mathsf{Crop}(\mathcal{R},\mathcal{C})$, $\widehat{\mathcal{R}} \coloneqq \mathsf{Crop}(\mathcal{C},\mathcal{R})$, and $B$ is the boundary of $\widehat{\mathcal{R}}$ and $\widehat{\mathcal{C}}$, then the following hold
\begin{enumerate}
    \item $|\mathcal{X}| \geq |I_{\mathcal{X}}|(r-1)+ 2|I_{\mathcal{Y}}| + 1$ for all $\mathcal{X}\neq\mathcal{Y}\in \{ \widehat{\mathcal{C}},\widehat{\mathcal{R}}\}$,
    \item for both $\mathcal{X} \in \{ \widehat{\mathcal{C}},\widehat{\mathcal{R}}\}$ and every $S\in \mathcal{X}$ it holds that $\chi(S) \subseteq I_{\mathcal{X}}$, and
    \item for each $\mathcal{X} \in \{ \widehat{\mathcal{C}},\widehat{\mathcal{R}}\}$ and every $i\in I_{\mathcal{X}}$ there exist at least $r$ strips $S \in \mathcal{X}$ such that $i\in\chi (C_S-B)$ where $C_S$ denotes the $p$-core of $S$.
\end{enumerate}
Moreover, there exists an algorithm that takes $p$, $b$, $r$, $(G,\chi)$, $M$, and $\rho$ as input and finds $\mathcal{C}$ and $\mathcal{R}$ in time $\mathbf{poly}(d)|E(G)|$.
\end{lemma}

\begin{proof}
We begin with an initialisation step, setting $M_0 \coloneqq M$, letting $\texttt{X} = \text{column}$, and letting $\mathcal{C}_0$ and $I_{\mathcal{C}_0}$ be the output of \zcref{lemma:StripRepresentation} applied to $M$ and $d$, $r$, $p$, $b$, and $x=2q+1$.
Further, let $\texttt{X} = \text{row}$ and let $\mathcal{R}_0$ and $I_{\mathcal{R}_0}$ be the output of \zcref{lemma:StripRepresentation} applied to $M$ and $d$, $r$, $p$, and $b$.

It follows that for both $\mathcal{X} \in \{ \mathcal{C}_0,\mathcal{R}_0 \}$ we have $|\mathcal{X}| \geq |I_{\mathcal{X}}|(r-1)+ 2q + 1 \geq |I_{\mathcal{X}}|(r-1)+ 2|I_{\mathcal{Y}}| + 1$ where $\mathcal{Y}\in\{ \mathcal{C}_0,\mathcal{R}_0 \}\setminus \{ \mathcal{X}\}$.
Moreover, for every $S\in\mathcal{X}$ we have that the breadth of $S$ is at least $2(|I_{\mathcal{X}}|+q+1)p + b$, $\chi(S) \subseteq I_{\mathcal{X}}$ and for every $i\in I_{\mathcal{X}}$ there are at least $r$ strips in $\mathcal{X}$ whose $p$-cores carry the colour $i$.

This initialisation step takes $\mathbf{poly}(d)|E(G)|$ time.
\smallskip
We now proceed inductively with $j\in[0,q]$.

Now, suppose we have constructed a mesh $M_j$ together with sets $I_{\mathcal{C}_j}, I_{\mathcal{R}_j} \subseteq [q]$ as well as $(p,b)$-padded packings $\mathcal{C}_j$ of type column and $\mathcal{R}_j$ of type row in $M_j$ such that for all $\mathcal{X} \neq \mathcal{Y}\in \{ \mathcal{C}_j,\mathcal{R}_j\}$
\begin{enumerate}
    \item $|\mathcal{X}| \geq |I_{\mathcal{X}}|(r-1) + 2|I_{\mathcal{Y}}| + 1$,
    \item for every $S\in \mathcal{X}$ it holds that $\chi(S) \subseteq I_{\mathcal{X}}$,
    \item  and every $S\in\mathcal{X}$ we have that the breadth of $S$ is at least $2(|I_{\mathcal{X}}|+q+1)p + b$, and
    \item for every $i\in I_{\mathcal{X}}$ there exist at least $r$ strips $S \in \mathcal{X}$ such that $i\in\chi (C_S)$ where $C_S$ denotes the $p$-core of $S$.
\end{enumerate}
We say that $M_j$, $I_{\mathcal{C}_j}$, $I_{\mathcal{R}_j}$, $\mathcal{C}_j$, and $\mathcal{R}_j$ \emph{satisfy the invariant}.
\smallskip

Let $M_{j+1} \coloneqq \mathsf{Crop}(\mathsf{Crop}(\mathcal{C}_j),\mathcal{R}_j)$ as well as $\mathcal{C}_{j+1}' \coloneqq \mathsf{Crop}(\mathcal{R}_j,\mathcal{C}_j)$ and $\mathcal{R}_{j+1}' \coloneqq \mathsf{Crop}(\mathcal{C}_j, \mathcal{R}_j)$.
Moreover, let $I'_{\mathcal{C}_{j+1}} \coloneqq I_{\mathcal{C}_j}$ and $I'_{\mathcal{R}_{j+1}} \coloneqq I_{\mathcal{R}_j}$.

We distinguish two major cases as follows.
It might be that $M_{j+1}$, $I'_{\mathcal{C}_{j+1}}$, $I'_{\mathcal{R}_{j+1}}$, $\mathcal{C}'_{j+1}$, and $\mathcal{R}'_{j+1}$ satisfy the invariant, however, in this case it is not a given that they also satisfy the \emph{core  property}:
\begin{itemize}
    \item For each $\mathcal{X} \in \{ \mathcal{C}_{j+1},\mathcal{R}_{j+1}\}$ and every $i\in I'_{\mathcal{X}}$ there exist at least $r$ strips $S \in \mathcal{X}'$ such that $i\in\chi (C_S-B)$ where $C_S$ denotes the $p$-core of $S$ in $M_{j+1}$ and $B$ denotes the boundary of $\mathcal{R}_{j+1}'$ and $\mathcal{C}_{j+1}'$.
\end{itemize}
However, they might also simply fail to satisfy the invariant in the first place.
In both cases we will discard at least one colour from at least one of the two sets $I'_{\mathcal{C}_{j+1}}$ and  $I'_{\mathcal{R}_{j+1}}$ and proceed with a sorting subroutine akin to the one from \zcref{lemma:StripRepresentation}.

Before we proceed, notice that the construction so far has made sure that $M_{j+1}$, $I'_{\mathcal{C}_{j+1}}$, $I'_{\mathcal{R}_{j+1}}$, $\mathcal{C}'_{j+1}$, and $\mathcal{R}'_{j+1}$ must always satisfy properties i), ii), and iii) of the invariant.
However, there will be one situation in which we enter \textbf{Case 2} where both, property i) and property iv) are violated.
This situation arises because we sometimes first enter \textbf{Case 1} and then remove two strips from the packing $\mathcal{Y}$ without updating the index set $I_{\mathcal{Y}}$, resulting in the invariant no longer being satisfied due to both, properties i) and iv) failing.
In this case we proceed to \textbf{Case 2} but have to make a slight tweak to the numbers in order to account for this issue.

\paragraph{Case 1: The invariant is satisfied.}
Suppose $M_{j+1}$, $I_{\mathcal{C}_j}$, $I_{\mathcal{R}_j}$, $\mathcal{C}'_{j+1}$, and $\mathcal{R}'_{j+1}$ satisfy the invariant.
In case they also satisfy the core property, we may set $\mathcal{C}\coloneqq \mathsf{Lift}(M,\mathcal{C}'_{j+1})$ and $\mathcal{R}\coloneqq \mathsf{Lift}(M,\mathcal{R}'_{j+1})$ and have thereby found the desired packings.
\smallskip

If, however, the core property is not satisfied, there exists a selection $\mathcal{X} \neq \mathcal{Y} \in \{ \mathcal{C}_{j+1},\mathcal{R}_{j+1}\}$ and a colour $i\in I'_{\mathcal{X}}$ such that there are at most $r-1$ strips $S \in \mathcal{X}'$ such that $i\in\chi (C_S-B)$ where $C_S$ denotes the $p$-core of $S$ in $M_{j+1}$ and $B$ denotes the boundary of $\mathcal{R}_{j+1}'$ and $\mathcal{C}_{j+1}'$.

We now update $\mathcal{Y}'$ to be the packing obtained from $\mathcal{Y}'$ by removing the left- and the rightmost strip of the packing.
Then we update $M_{j+1}$ to be the mesh $\mathsf{Crop}(M_{j+1},\mathcal{Y}')$ and we update $\mathcal{X}'$ to be $\mathsf{Crop}(\mathcal{Y}',\mathcal{X}')$, both after updating $\mathcal{Y}'$.

After performing the updates, by our assumption, we know that there are now at most $r-1$ strips $S$ in the updated $\mathcal{X}'$ such that $i\in\chi(C_S)$.
Hence, property iv) of the invariant is violated.
Moreover, since $\mathcal{Y}'$ has shrunken by $2$ without us updating the set $I'_{\mathcal{X}}$ it might now be true that also property i) of the invariant is violated.
Since we are sure that at least property iv) is violated, we may now pass over to \textbf{Case 2} which will now be forced to remove colour $i$ from the set $I'_{\mathcal{X}}$, thereby rectifying -- if necessary -- the issue with property i).

\paragraph{Case 2: The invariant is not satisfied.}
We may assume that there exist $\mathcal{X}\in \{ \mathcal{C}_{j+1},\mathcal{R}_{j+1}\}$ and $i\in I'_{\mathcal{X}}$ such that $\mathcal{X}'$ contains at most $r-1$ strips that carry the colour $i$ in their $p$-cores.

In case we reached \textbf{Case 2} via \textbf{Case 1} we may assume that $\mathcal{X}$ and $i$ are precisely the same as they were in \textbf{Case 1}.
Moreover, if this is the case, we have removed two strips from $\mathcal{Y}$ and may set $n_{\mathcal{X}} \coloneqq |I'_{\mathcal{X}}|-1$ as well as $n_{\mathcal{Y}} \coloneqq |I'_{\mathcal{Y}}|$.
Otherwise, we entered \textbf{Case 2} directly because the invariant was not satisfied in the first place.
In this case we set $n_{\mathcal{X}} \coloneqq |I'_{\mathcal{X}}|$ as well as $n_{\mathcal{Y}} \coloneqq |I'_{\mathcal{Y}}|$.

Since $\mathcal{C}_j$ and $\mathcal{R}_j$ satisfy the invariant, we have that $M_{j+1}$ is an $(x_{j+1},y_{j+1})$-mesh where
\begin{align*}
    x_{j+1} & \geq (|I_{\mathcal{R}_j}|(r-1) + 2n_{\mathcal{C}_{j+1}} + 1) \cdot (2(|I_{\mathcal{R}_j}|+q+1)p + b)\\
    y_{j+1} & \geq (|I_{\mathcal{C}_j}|(r-1) + 2n_{\mathcal{R}_{j+1}} + 1) \cdot (2(|I_{\mathcal{C}_j}|+q+1)p + b)
\end{align*}
Moreover, $(\mathsf{compass}(M_{j+1}),\chi^{\mathcal{Z}})$ is a $q^{\mathcal{Z}}$-colorful graph where $q^{\mathcal{Z}} \coloneqq |I'_{\mathcal{Z}}|$ and $\chi^{\mathcal{Z}}$ is the restriction of the image of $\chi$ to the set $I'_{\mathcal{Z}}$ for each $\mathcal{Z} \in \{ \mathcal{C}_{j+1},\mathcal{R}_{j+1}\}$.

Under these assumptions, we may apply the method established in the proof of \zcref{lemma:StripRepresentation} to $(\mathsf{compass}(M_{j+1}),\chi^{\star})$, $I_{\mathcal{Z}}$, and the packing $\mathcal{Z}'$ for each of the two choices for $\mathcal{Z}\neq \mathcal{W}\in\{ \mathcal{C}_{j+1},\mathcal{R}_{j+1}\}$ individually.
The result are a set $I_{\mathcal{Z}^{\star}}$ and a packing $\mathcal{Z}^{\star}$ such that
\begin{enumerate}
    \item $|\mathcal{Z}^{\star}| \geq |I_{\mathcal{Z}^{\star}}|(r-1) + 2n_{\mathcal{W}} + 1$,
    \item each strip of $\mathcal{Z}^{\star}$ has breadth at least $2(|I_{\mathcal{Z}^{\star}}| + 1)p + b$
    \item for every $S \in \mathcal{Z}^{\star}$ it holds that $\chi^{\mathcal{Z}}(S) \subseteq I_{\mathcal{Z}^{\star}}$,
    \item for every $h\in I_{\mathcal{Z}^{\star}}$ there exist $r$ strips $S\in\mathcal{Z}^{\star}$ such that $h\in \chi^{\mathcal{Z}}(C_S)$ where $C_S$ denotes the $p$-core of $S$, and
    \item every strip in $\mathcal{Z}^{\star}$ is contained in some strip from $\mathcal{Z}'$.
\end{enumerate}
Moreover, this procedure takes at most $\mathbf{poly}(d)|E(G)|$ time.

Please notice that we make use of the \textsl{proof} of \zcref{lemma:StripRepresentation} here rather than its statement.
The reason is, because in this situation we need to prescribe the initial packing $\mathcal{Z}'$ of strips instead of starting with one greedily chosen such as the one in the beginning of the proof of \zcref{lemma:StripRepresentation}.
However, it is easily seen that the subroutines \textbf{Sort} and \textbf{Trim} may be applied directly to $\mathcal{Z}'$ in order to reach the outcome described above.

From the fact that $\chi(S) \subseteq I'_{\mathcal{Z}}$ for all $S\in \mathcal{Z}'$ and the fact that each strip in $\mathcal{Z}^{\star}$ is contained in some strip of $\mathcal{Z}'$ it also follows that $\chi^{\mathcal{Z}}(S) \subseteq I'_{\mathcal{Z}^{\star}}$ implies $\chi(S) \subseteq I_{\mathcal{Z}^{\star}}$.
It also follows from our assumptions on $i\in I'_{\mathcal{X}}$ that $|I_{\mathcal{X}^{\star}}| \leq |I'_{\mathcal{X}}| - 1$.

Let now $\mathcal{C}_{j+1} \coloneqq \mathcal{C}_{j+1}^{\star}$ and $\mathcal{R}_{j+1} \coloneqq \mathcal{R}_{j+1}^{\star}$.
We also set $I_{\mathcal{C}_{j+1}} \coloneqq I_{\mathcal{C}_{j+1}^{\star}}$ and $I_{\mathcal{R}_{j+1}} \coloneqq I_{\mathcal{R}_{j+1}^{\star}}$.

If we reached \textbf{Case 2} via \textbf{Case 1}, we had that $n_{\mathcal{X}} = |I'_{\mathcal{X}}|-1$.
However, in this case we also knew that $i\in I'_{\mathcal{X}}$ and we know that $i \notin I_{\mathcal{X}}$.
Hence, we now have that $n_{\mathcal{X}} \geq |I_{\mathcal{X}}|$.
This inequality is trivially true -- in fact, even a strict inequality holds -- if we entered \textbf{Case 2} directly.

It now follows that $M_{j+1}$, $I_{\mathcal{C}_{j+1}}$, $I_{\mathcal{R}_{j+1}}$, $\mathcal{C}_{j+1}$, and $\mathcal{R}_{j+1}$ satisfy the invariant.
\medskip

In the above case distinction, one of two outcomes is always reached: 
Either we find the objects as required by the assertion, or we create a collection $M_{j+1}$, $I_{\mathcal{C}_{j+1}}$, $I_{\mathcal{R}_{j+1}}$, $\mathcal{C}_{j+1}$, and $\mathcal{R}_{j+1}$ satisfying the invariant our of the collection $M_{j}$, $I_{\mathcal{C}_{j}}$, $I_{\mathcal{R}_{j}}$, $\mathcal{C}_{j}$, and $\mathcal{R}_{j}$
In the second outcome we made sure that $|I_{\mathcal{C}_{j+1}}| + |I_{\mathcal{R}_{j+1}}| < |I_{\mathcal{C}_{j}}| + |I_{\mathcal{R}_{j}}|$.
Hence, in every ``round'' of the iterative process above we either terminate or lose one of at most $q$ colours from one of the two sets $I_{\mathcal{R}_j}$ and $I_{\mathcal{C}_j}$.
Therefore, this process must come to a halt after at most $2q$ iterations and the assertion follows.
\end{proof}

\paragraph{Confining colours into tiles.}
With \zcref{lemma:Cropping} we have now gotten rid of the elusive ``end-tiles'' and may finally focus on the main problem:
Instead of capturing the colours in the cores of many strips, we want to ensure that a single colour appears in many \textsl{tiles}.
Indeed, we prove a slightly stronger variant that allows for better analysis.
That is, we prove that we may ensure that for every colour $i$ that appears in \textsl{some} strip there are many strips, each with their own tile in which $i$ appears.

We set 
\begin{align*}
    \mathsf{r}(q,r) & \coloneqq r(r-1) + (q+1)r, \text{ and}\\
    \mathsf{d}_{\ref{lemma:Tiles}}(q,p,b,r) & \coloneqq \big(q (\mathsf{r}(q,r) + 1) + 2q + 1 \big) \cdot \big( 2(q+1)p + b \big).
\end{align*}

\begin{lemma}\label{lemma:Tiles}
Let $d$, $r\geq 2$, $p$, $b$, and $q$ be non-negative integers.
Let $(G,\chi)$ be a $q$-colorful graph, $d \geq \mathsf{d}_{\ref{lemma:Tiles}}(q,p,b,r)$ and $M$ be a $(d \times d)$-mesh in $G$ such that $M$ is flat witnessed by the rendition $\rho$.

Then there exist possibly empty sets $I_{\mathcal{C}},I_{\mathcal{R}}\subseteq [q]$ and two $(p,b)$-padded packings $\mathcal{C}$ of type column and $\mathcal{R}$ of type row such that if $\widehat{M}\coloneqq \mathsf{Crop}(\mathsf{Crop}(M,\mathcal{C}),\mathcal{R})$, $\widehat{\mathcal{C}} \coloneqq \mathsf{Crop}(\mathcal{R},\mathcal{C})$, and $\widehat{\mathcal{R}} \coloneqq \mathsf{Crop}(\mathcal{C},\mathcal{R})$, then the following hold for both $\mathcal{X} \in \{ \mathcal{C},\mathcal{R} \}$
\begin{enumerate}
    \item $\widehat{\mathcal{X}}\neq\emptyset$,
    \item for every $S\in \widehat{\mathcal{X}}$ we have that $\chi(S) \subseteq I_{\mathcal{X}}$, and
    \item for every $i\in I_{\mathcal{X}}$ there exist $\mathcal{Y}\neq \mathcal{Z} \in \{ \widehat{\mathcal{C}},\widehat{\mathcal{R}} \}$ such that there exist at least $r$ distinct strips $S \in \mathcal{Y}$ such that $i\in\chi(T)$ for some $T \in \mathsf{Tiles}^p_{\mathcal{Z}}(S)$.
\end{enumerate}
Moreover, there exists an algorithm that takes $p$, $b$, $r$, $(G,\chi)$, $M$, and $\rho$ as input and finds $\mathcal{C}$ and $\mathcal{R}$ in time $\mathbf{poly}(d)|E(G)|$.
\end{lemma}

\begin{proof}
We begin by applying \zcref{lemma:Cropping} to $(G,\chi)$ and $M$ with parameters $d$, $p$, $b$, $q$, and $r' = \mathsf{r}(q,r)$.
By doing so, we obtain in time $\mathbf{poly}(d)|E(G)|$ sets $I_{\mathcal{C}_0},I_{\mathcal{R}_0}\subseteq [q]$ and two $(p,b)$-padded packings $\mathcal{C}'_0$ of type column and $\mathcal{R}'_0$ of type row such that if $M'\coloneqq \mathsf{Crop}(\mathsf{Crop}(M,\mathcal{C}'_0),\mathcal{R}'_0)$, $\mathcal{C}_0 \coloneqq \mathsf{Crop}(\mathcal{R}'_0,\mathcal{C}'_0)$, $\mathcal{R}_0 \coloneqq \mathsf{Crop}(\mathcal{C}'_0,\mathcal{R}'_0)$, and $B$ is the boundary of $\mathcal{R}_0$ and $\mathcal{C}_0$, then the following hold for all $\mathcal{X}\neq \mathcal{Y}\in\{ \mathcal{R}_0,\mathcal{C}_0\}$
\begin{itemize}
    \item $|\mathcal{X}| \geq |I_{\mathcal{X}}|(\mathsf{r}(q,r)-1)+ 2|I_{\mathcal{Y}}| + 1$,
    \item for every $S\in \mathcal{X}$ it holds that $\chi(S) \subseteq I_{\mathcal{X}}$, and
    \item for every $i\in I_{\mathcal{X}}$ there exist at least $\mathsf{r}(q,r)$ strips $S \in \mathcal{X}$ such that $i\in\chi (C_S-B)$ where $C_S$ denotes the $p$-core of $S$.
\end{itemize}

We first process the colours in $I \coloneqq I_{\mathcal{C}_0} \cap I_{\mathcal{R}_0}$.
Let $I^{\star} \coloneqq \emptyset$ be the set of all colours from $I$ processed so far.
We prove by induction on $|I^{\star}|$ that we can find $(p,b)$-padded packings $\mathcal{R}_i$ of rows and $\mathcal{C}_i$ of columns such that for all choices of $\mathcal{X}\neq\mathcal{Y} \in \{\mathcal{R},\mathcal{C} \}$
\begin{enumerate}[label=(\alph*)]
    \item $\mathcal{X}_{|I^{\star}|}\neq\emptyset$,
    \item for every $S\in \mathcal{X}_{|I^{\star}|}$ we have that $\chi(S) \subseteq I_{\mathcal{X}_0}$,
    \item for every $i\in I_{\mathcal{X}_0}\setminus I^{\star}$ there exist at least $\mathsf{r}(q,r) - r|I^{\star}|$ strips $S \in \mathcal{X}$ such that $i\in\chi (C_S-B)$ where $C_S$ denotes the $p$-core of $S$ and $B$ is the boundary of $\mathcal{R}_{|I^{\star}|}$ and $\mathcal{C}_{|I^{\star}|}$, and
    \item for every $i\in I^{\star}$ there exist at least $r$ distinct strips $S \in \mathcal{C}_{|I^{\star}|}$ such that $i\in\chi(T)$ for some $T \in \mathsf{Tiles}^p_{\mathcal{R}_{|I^{\star}|}}(S)$.
\end{enumerate}
The base case, that is the case where $I^{\star}=\emptyset$, holds trivially by our choice of $M'$, $\mathcal{C}_0$, and $\mathcal{R}_0$.

Now suppose we have that $n \coloneqq |I^{\star}| \in [0,|I|-1]$ and let $i\in I\setminus I^{\star}$.
Further assume that we have already constructed $\mathcal{C}_n$ and $\mathcal{R}_n$ satisfying the requirements above and let us denote by $\mathcal{C}_n'$ the packing obtained from $\mathcal{C}_n$ by removing the left- and the rightmost strip.
Similarly, denote by $\mathcal{R}'_n$ the packing obtained from $\mathcal{R}_n$ by removing the left- and the rightmost strip.
Let $B_n$ denote the boundary of $\mathcal{C}_n$ and $\mathcal{R}_n$.
For any $S \in \mathcal{R}'_n \cap \mathcal{C}'_n$, denote by $\mathsf{Core}_n(S)$ the $p$-core of $S$ minus $B_n$.

Suppose there exist at least $r$ strips $S\in\mathcal{C}'_n$ for which $i\in \chi(T)$ for some tile $T\in\mathsf{Tiles}^p_{\mathcal{R}_n}$.
In this case we may add $i$ to $I^{\star}$ immediately and no further processing is required.
We set $\mathcal{C}_{n+1} \coloneqq \mathcal{C}_n$ and $\mathcal{R}_{n+1} \coloneqq \mathcal{R}_n$ and are done by induction.
\smallskip

Hence, we may assume that there are at most $r-1$ strips $S\in\mathcal{C}'_n$ for which $i\in \chi(T)$ for some tile $T\in\mathsf{Tiles}^p_{\mathcal{R}_n}$.
In this situation, we greedily choose pairs $(C_j,R_j)$, $j\in[\ell]$ such that
\begin{itemize}
    \item $C_j \in \mathcal{C}'_n\setminus \{ C_h \colon h\in[j-1] \}$,
    \item $R_j \in \mathcal{R}'_n\setminus \{ R_h \colon h\in[j-1]\}$, and
    \item $i\in \chi( \mathsf{Core}_n(C_j) \cap \mathsf{Core}_n(R_j) ) $
\end{itemize}
until no such choice is possible any more.
What follows is a case distinction on whether $\ell \geq r$ or $\ell \leq r-1$.

\paragraph{Case 1: $\ell \geq r$.}
In this case we add $i$ to $I^{\star}$.
Let $\mathcal{R}_{n+1} \coloneqq \mathcal{R}\setminus \{ R_1,\dots, R_r\}$ and $\mathcal{C}_{n+1} \coloneqq \mathcal{C}_n$.
We claim that this selection of strips satisfies all four requirements above.
First of all (b) is still true for both $\mathcal{R}_{n+1}$ and $\mathcal{C}_{n+1}$.
Also, (a) clearly holds for $\mathcal{C}_{n+1}$.
Since we removed precisely $r$ strips from $\mathcal{R}_n$, it follows from (c) for $\mathcal{R}_n$ that $\mathcal{R}_{n+1}$ still contains at least $\mathsf{r}(q,r)-r|I^{\star}| > 0$ many strips that carry the colour $i$ in their core.
Hence, (a) also holds for $\mathcal{R}_{n+1}$.
Similarly, it is easy to see that (c) holds for both $\mathcal{R}_{n+1}$ and $\mathcal{C}_{n+1}$ by the same observation.
That is, for each $\mathcal{X}\in \{\mathcal{R},\mathcal{C}\}$, and $j\in I_{\mathcal{X}_0}\setminus I^{\star}$ the property (c) for $\mathcal{X}_{n}$ ensured that there are at least $\mathsf{r}(q,r)-r|I^{\star}\setminus\{ i\}|$ many strips $S$ in $\mathcal{X}_n$ that carry the colour $j$ in their cores.
Since we removed precisely $r$ strips from $\mathcal{R}_n$ and none from $\mathcal{C}_n$ we get that there now are still at least
\begin{align*}
    \mathsf{r}(q,r)-r|I^{\star}\setminus\{ i\}| - r = \mathsf{r}(q,r)-r|I^{\star}|
\end{align*}
many strips remaining that carry the colour $j$ in their core.
\smallskip

So the only requirement left to discuss now is (d).

The only difference between $\mathcal{R}_{n+1}$ and $\mathcal{R}_n$ is the removal of the strips $R_1,\dots,R_n$.
Consider any $C\in\mathcal{C}_n'$ and $R\in\mathcal{R}_n'$ and let $\mathcal{L}\coloneqq \mathcal{R}_n' \setminus \{ R\}$.
Notice that $\mathsf{Tiles}^p_{R}(C)$ contains precisely $3$ tiles in which $R$ is involved:
The tile $T_1$ which borders $R$ only in its left $p$-buffer, the tile $T_2$ that lives in the intersection of $C$ and $R$, and the tile $T_3$ that borders $R$ only in its right $p$-buffer.
All three tiles are also present in $\mathsf{Tiles}^p_{\mathcal{R}_n}(C)$.
However, $\mathsf{Tiles}^p_{\mathcal{L}}(C)$ has precisely two tiles less.
Indeed, there exists a tile $T\in \mathsf{Tiles}^p_{\mathcal{L}}(C)$ which contains the union of $T_1$, $T_2$, and $T_3$ together with the intersection of the buffer of $R$ with the $p$-core of $C$.
The key assumption which makes this argument work is that we chose both $C$ and $R$ to be distinct from the extremal strips in their respective packings.
This implies that removing $R$ still leaves another strip from $\mathcal{L}$ to the left and the right of $R$ which contribute their respective boundaries to the creating of the tile $T$.

This means, that any tile from $\mathsf{Tiles}^p_{\mathcal{R}_n}(C)$ is contained in a tile from $\mathsf{Tiles}^p_{\mathcal{R}_{n+1}}(C)$ for any $C\in\mathcal{C}_n' = \mathcal{C}_{n+1}'$.
Hence, if $C\in\mathcal{C}_{n+1}'$ had a tile $T\in\mathsf{Tiles}^p_{\mathcal{R}_n}(C)$ such that $j\in \chi(T)$ for some colour $j\in[q]$, then there must also exist at least one such tile in $\mathsf{Tiles}^p_{\mathcal{R}_{n+1}}(C)$.
It follows that (d) is satisfied for all $j\in I^{\star}\setminus\{ i\}$.
Moreover, for each $j\in[r]$ we have that $i\in \chi( \mathsf{Core}_n(C_j) \cap \mathsf{Core}_n(R_j) )$.
From the discussion above it now follows that there exists a tile $T_j \in \mathsf{Tiles}^p_{\mathcal{R}_{n+1}}(C_j)$ such that $\mathsf{Core}_n(C_j) \cap \mathsf{Core}_n(R_j) \subseteq T_j$.
Hence, $i\in \chi(T_j)$.
Since we were able to pick at least $r$ distinct such columns $C_j$ from $\mathcal{C}'_{n+1}$ we are now sure that there exist at least $r$ distinct columns in $\mathcal{C}_{n+1}$ such that each of them has a tile that carries the colour $i$.
So (d) is satisfied by $\mathcal{R}_{n+1}$ and $\mathcal{C}_{n+1}$ and the case is complete.

\paragraph{Case 2: $\ell \leq r-1$.}
Now we have found at most $r-1$ pairs $(C_j,R_j)$ as above.
However, by (c), there are at least
\begin{align*}
    \mathsf{r}(q,r) - r|I^{\star}| & \geq \mathsf{r}(q,r) - r\cdot q\\
    & = r(r-1) + (q+1)r - r \cdot q\\
    & \geq r(r-1) + 2 \\
    & = r^2 - r + 2\\
    & = r-1 + (r-1)^2 +2
\end{align*}
columns $C\in\mathcal{C}_{n}$ such that $i \in \chi(C_S)$ where $C_S$ denotes the $p$-core of $C$.
At most $2$ of those belong to the boundary of $\mathcal{C}_{n}$ and $\mathcal{R}_{n}$ which leaves at least $r-1 + (r-1)^2$ many such columns in $\mathcal{C}_{n}'$.
Consider a set $\mathcal{F}$ of $(r-1)^2$ such columns which are distinct from $C_1, \ldots , C_j$, then there must exist some $j^{\star} \in[\ell]$ such that $i\in \chi( \mathsf{Core}_n(C) \cap \mathsf{Core}_n(R_{j^{\star}}) )$ for each $C\in\mathcal{F}$ by the pigeonhole principle.
Add $C_{j^{\star}}$ and we have found a family $\mathcal{U}$ of $r$ columns from $\mathcal{C}_{n}'$, all of which carry the colour $i$ in the intersection of their $p$-core with the $p$-core of $R_{j^{\star}}$.

We now set $\mathcal{C}_{n+1} \coloneqq \mathcal{C}_n$ and $\mathcal{R}_{n+1} \coloneqq \mathcal{R}_n\setminus \{ R_{j^{\star}}\}$, add $i$ to $I^{\star}$, and claim that this selection satisfies properties (a), (b), (c), and (d).

Similar to before, it is easy to see that (a) still holds for $\mathcal{C}_{n+1}$ and $\mathcal{R}_{n+1}$, the latter due to the lower bound on $\mathcal{R}_n$ guaranteed by (c).
Similarly, (b) must still hold.
Now for (c), we have increased the size of $I^{\star}$ by $1$ and decreased the size of $\mathcal{R}_{n+1}$ by $1$ compared to $\mathcal{R}_n$.
Hence (c) is also easily seen to hold.

As before, the only remaining case is (d).
Recall our discussion on the way new tiles emerge from old when removing strips from $\mathcal{R}_n'$.
This time, we only removed $R_{j^{\star}}$.
However, for each $C\in \mathcal{U}$ there now exists a tile $T$ in $\mathsf{Tiles}^p_{\mathcal{R}_{n+1}}(C)$ which contains the entire intersection of the $p$-core of $R_{j^{\star}}$ with the $p$-core of $C$.
Hence, $i\in \chi(T)$.
Moreover, as before, the same discussion shows that any colour present in some tile from $\mathsf{Tiles}^p_{\mathcal{R}_{n}}(C)$ for any $C\in\mathcal{C}_{n+1}$ must also be present in a tile from $\mathsf{Tiles}^p_{\mathcal{R}_{n+1}}(C)$ and thus, (d) is satisfied by $\mathcal{C}_{n+1}$ and $\mathcal{R}_{n+1}$.
\bigskip

It now follows by induction that, after $|I|$ rounds of the process described above, we have reached the case $n=|I|$ and $I^{\star} = I$.
In this case we now know that for $\mathcal{C}_n$ and $\mathcal{R}_n$, the assertion of our lemma holds for all $i\in I = I_{\mathcal{C}_0} \cap I_{\mathcal{R}_0}$.
So all that is left to discuss are the colours that are private to either the rows or the columns.

However, let $\mathcal{X}\neq\mathcal{Y} \in \{ \mathcal{R}, \mathcal{C} \}$, $i\in I_{\mathcal{X}_0}\setminus I^{\star}$, and $C\in \mathcal{X}_n$ such that $i\in \chi(C_S)$ where $S$ denotes the $p$-core of $S$.
There are still at least $\mathsf{r}(q,r) - qr = r(r-1) + r \geq r$ many such strips in $\mathcal{X}_n$.
Moreover, since $i\notin I^{\star}$, we have that $i\notin I_{\mathcal{Y}_0}$.
This means, that $i$ cannot appear in the intersection between a row and a column.
Hence, if $i\in \chi(C_S)$ for some $S\in\mathcal{X}_n$, then there must be some $T\in\mathsf{Tiles}^p_{\mathcal{Y}_n}(S)$ such that $i\in \chi(T)$.
Thus, by lifting our choice of $\mathcal{R}_n$ and $\mathcal{C}_n$ back to $M$, we have found the desired packings of strips together with the sets $I_{\mathcal{C}_0}$ and $I_{\mathcal{R}_0}$ of colours.
\end{proof}

\section{Capturing a rainbow in the middle row}\label{sec:Rainbow}
Now that we are able to sort out colours and make sure that colours that were not sorted out are abundant in the tiles of the rows and columns we selected, we want to start constructing the new mesh which is supposed to be homogeneous by the end of our efforts.
In this section, we make the first of two steps towards this.

We also want to ensure that the tangle of the new mesh we create is a truncation of the tangle of the original mesh.
This is actually quite easy to argue and mainly involves the use of the following easy observation about paths in meshes.

\begin{observation}\label{obs:meshconnectivity}
    Let $n,m,k$ be positive integers with $k \geq \min(n,m)$.
    Let $G$ be a graph containing an $(n \times m)$-mesh $M$ and let $S \subseteq V(M)$ be a set of vertices intersecting at least $k$ pairwise distinct horizontal or vertical paths of $M$.
    Then for any horizontal or vertical path $P$ of $M$, there exists $k$ disjoint paths between $V(P)$ and $S$ in $M$ (and thus in $G$).
\end{observation}

The first kind of mesh we will construct will not immediately be homogeneous, but instead host all colours found in the entire mesh in its middle row.

Let $q$, $m$, and $n$ be positive integers with $n$ being even, let $(G,\chi)$ be a $q$-colorful graph, and let $M$ be an $(n \times m)$-mesh such that $P_1, \ldots , P_n$ are the horizontal paths of $M$ and $Q_1, \ldots , Q_m$ are the vertical paths of $M$.
Further, for each $i \in [m-1]$, let $C_i$ be the unique cycle in $Q_i \cup Q_{i+1} \cup P_{\frac{m}{2}} \cup P_{\frac{m}{2}+1}$ and let $M$ be flat witnessed by the rendition $\rho$, with $H$ being the compass of $M$ with respect to $\rho$.
Then we say that $M$ has a \emph{rainbow middle row}, if we have $\chi(H) = \chi(H_i)$ for all $i \in [m-1]$, where $H_i$ is the compass of $C_i$ with respect to $\rho$.
See \zcref{fig:rainbowmesh} for an example of a $(10 \times 12)$-mesh with a rainbow middle row.

\begin{figure}[ht]
    \centering
    \begin{tikzpicture}

        \pgfdeclarelayer{background}
		\pgfdeclarelayer{foreground}
			
		\pgfsetlayers{background,main,foreground}

        \begin{pgfonlayer}{background}
            \pgftext{\includegraphics[width=8cm]{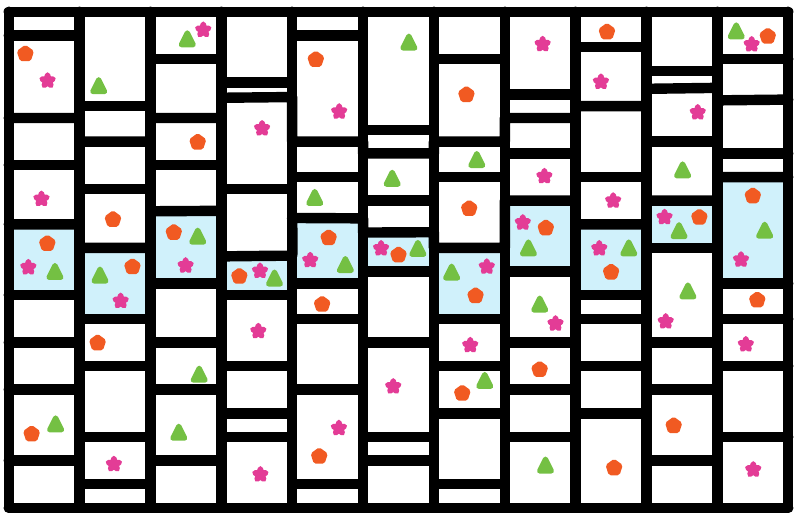}} at (C.center);
        \end{pgfonlayer}{background}
			
        \begin{pgfonlayer}{main}
        \node (C) [v:ghost] {};
        \end{pgfonlayer}{main}
        
        \begin{pgfonlayer}{foreground}
        \end{pgfonlayer}{foreground}

    \end{tikzpicture}
    \caption{An example of a mesh with a rainbow middle row. The middle row is marked in blue.}
    \label{fig:rainbowmesh}
\end{figure}

Note that a mesh with a rainbow middle row is not necessarily homogeneous.
However, as we will see in the next section, it is fairly easy to derive a slightly smaller homogeneous mesh using all the same colours and thus, finding a rainbow middle row is the last big step we need to take towards our main theorem.

Before we prove that we can find such a mesh, we introduce one more function:
\begin{align*}
    \mathsf{d}_{\ref{lem:rainbowrowfinder}}(n,m,q) \coloneqq \mathsf{d}_{\ref{lemma:Tiles}}(q, \max(\nicefrac{n}{2}+m,2m), 2, q(m-2)+1).
\end{align*}

Due to the length of the following proof, we want to point out that the explanation at the beginning of it, together with the illustrations provided should serve to explain all necessary details to convince the reader that our construction produces the desired outcome.
The technical arguments provided are mainly there for the sake of completeness.

\begin{lemma}\label{lem:rainbowrowfinder}
Let $d$, $n$, and $m$ be positive integers and let $q$ be a non-negative integer.
Let $(G,\chi)$ be a $q$-colorful graph, let $m$ be even, let $d \geq \mathsf{d}_{\ref{lem:rainbowrowfinder}}(n,m,q)$, and $M$ be a $(d \times d)$-mesh in $G$ such that $M$ is flat witnessed by the rendition $\rho$.

Then there exists an $(n \times m)$-mesh $M_0 \subseteq M$ that is
\begin{enumerate}
    \item flat witnessed by the rendition $\rho$,

    \item has a rainbow middle row, and

    \item the tangle of $M_0$ is a truncation of the tangle of $M$.
\end{enumerate}

Moreover, there exists an algorithm that takes $d$, $n$, $m$, $(G,\chi)$, $M$, and $\rho$ as input and finds $M_0$ in time $\mathbf{poly}(d)|E(G)|$.
\end{lemma}
\begin{proof}
This proof consists of a simple construction, which proceeds as follows:
First, we apply \zcref{lemma:Tiles} to find the submesh of $M$ together with the cleaned-up row and column packings promised in that statement.
We then anchor ourselves with an $(\nicefrac{n}{2} \times m)$-mesh in the bottom left corner of this submesh.
From there we first collect colours from all tiles shared by rows and the first column, and all tiles exclusive to the first column by walking through the boundary and greedily picking up a colour from a tile whenever we notice that a part of the middle row of the mesh we are constructing is still missing it.
Once we are at the top, we move slightly further right, using the rightmost half of the right buffer of the first column and the leftmost part of the left buffer of the second column to move down and do the same to the rightmost tiles that are exclusive to the rows, which we missed so far.
From there we again move slightly further to the right and now repeat this process with the second column.
Once we arrive at the rightmost column, we close our construction with another $(\nicefrac{n}{2} \times m)$-mesh (see \zcref{fig:rainbowmeshconstruction}).

\begin{figure}[btp]
    \centering
    \begin{tikzpicture}

        \pgfdeclarelayer{background}
		\pgfdeclarelayer{foreground}
			
		\pgfsetlayers{background,main,foreground}

        \begin{pgfonlayer}{background}
            \pgftext{\includegraphics[width=\textwidth]{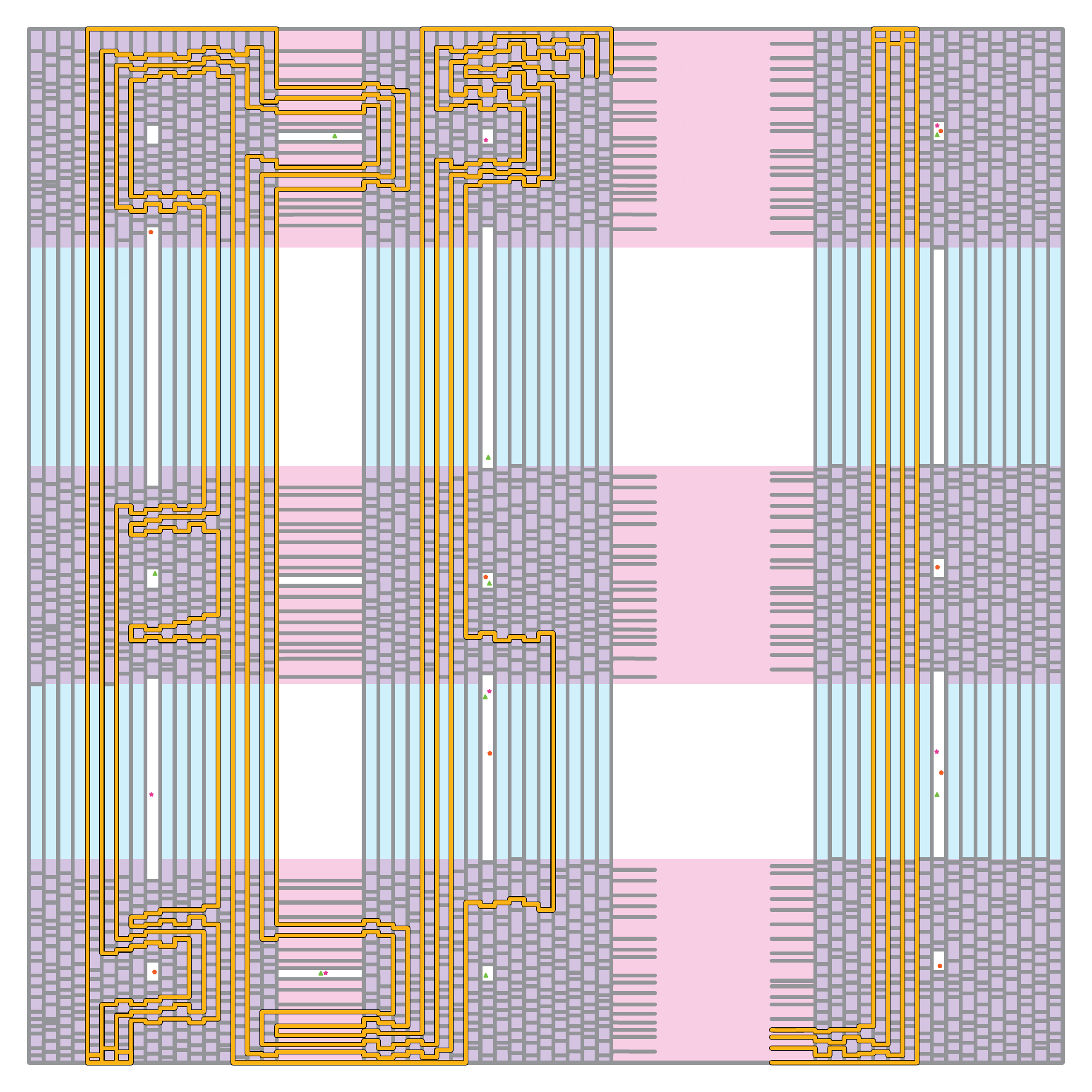}} at (C.center);
        \end{pgfonlayer}{background}
			
        \begin{pgfonlayer}{main}
        \node (C) [v:ghost] {};
        \end{pgfonlayer}{main}
        
        \begin{pgfonlayer}{foreground}
        \end{pgfonlayer}{foreground}

    \end{tikzpicture}
    \caption{An example of the construction we present in the proof of \zcref{lem:rainbowrowfinder}. Here, we construct a $(4 \times 4)$-mesh represented by the orange lines in a larger mesh that hosts 3 colours, represented by the coloured shapes. On the left, the figure shows the progress after walking up from the $(2 \times 4)$-mesh we select in the bottom right, then walking down, and then walking up once more. On the right of the illustration, we close the construction with another $(2 \times 4)$-mesh. \zcref{fig:rainbowmeshconstructiondetail} shows a close-up of this construction that may be more legible to some readers.}
    \label{fig:rainbowmeshconstruction}
\end{figure}

If enough tiles exist with respect to the total number of colours we need to collect in each face of the middle row, it will be easy to prove that this results in a rainbow middle row.
In particular, since our construction only moves through the cleaned-up rows and columns from \zcref{lemma:Tiles}, we can guarantee that no colour present in the area we sweep through will be missing from the middle row.
Showing that the tangle of the new mesh is a truncation of the old one is then a simple matter of applying \zcref{obs:meshconnectivity} appropriately.

Let $P_1^*, \ldots , P_d^*$ be the horizontal paths of $M$ and let $Q_1^*, \ldots , Q_d^*$ be the vertical paths of $M$.

\paragraph{Setting up our paddings and the bottom half of our mesh.}
We first apply \zcref{lemma:Tiles} to our objects, with $r = q(m-2)+1$, $p = \max(\nicefrac{n}{2}+m, 2m)$, and $b = 2$.
Thanks to $d \geq \mathsf{d}_{\ref{lem:rainbowrowfinder}}(n,m,q)$, this yields the following:
Two possibly empty sets $I_{\mathcal{C}},I_{\mathcal{R}}\subseteq [q]$ and two $(p,2)$-padded packings $\mathcal{C}$ of type column and $\mathcal{R}$ of type row such that if $\widehat{M}\coloneqq \mathsf{Crop}(\mathsf{Crop}(M,\mathcal{C}),\mathcal{R})$, $\widehat{\mathcal{C}} \coloneqq \mathsf{Crop}(\mathcal{R},\mathcal{C})$, and $\widehat{\mathcal{R}} \coloneqq \mathsf{Crop}(\mathcal{C},\mathcal{R})$, then the following hold for both $\mathcal{X} \in \{ \mathcal{C},\mathcal{R} \}$
\begin{enumerate}
    \item $\widehat{\mathcal{X}}\neq\emptyset$,
    \item for every $S\in \widehat{\mathcal{X}}$ we have that $\chi(S) \subseteq I_{\mathcal{X}}$, and
    \item for every $i\in I_{\mathcal{X}}$ there exist $\mathcal{Y}\neq \mathcal{Z} \in \{ \widehat{\mathcal{C}},\widehat{\mathcal{R}} \}$ such that there exist at least $r$ distinct strips $S \in \mathcal{Y}$ such that $i\in\chi(T)$ for some $T \in \mathsf{Tiles}^p_{\mathcal{Z}}(S)$.
\end{enumerate}
Since we will have to work with these objects rather actively, we spend some time further specifying them and their contents.
We let $r' \coloneqq |\mathcal{R}|$ and $c' \coloneqq |\mathcal{C}'|$.
Let $\mathcal{C} = \{ S_1, \ldots , S_{c'} \}$ such that for all $i,j \in [c']$ with $i < j$ and $h \in [d]$, we have that $V(P_h^*) \cap V(S_i) \neq \emptyset$ implies that $V(P_\ell^*) \cap V(S_j) = \emptyset$ for all $\ell \in [h,d]$.
In other words, in the natural drawing of our setting (see \zcref{fig:rainbowmeshconstruction} for reference), the column strips $S_1, \ldots , S_{c'}$ are sorted from left to right according to their indices, with $S_1$ being the leftmost column and $S_{c'}$ being the rightmost column.
We index $\mathcal{R} = \{ R_1, \ldots , R_{r'} \}$ analogously with respect to $Q_1^*, \ldots , Q_d^*$, such that the row strips are indexed from the bottom to the top, with $R_{r'}$ being the lowest row and $R_1$ being the highest.

For each $i \in [c']$, we let $\mathcal{F}_i^{\mathrm{c}} = \{ P_{(i-1)(p+2) + 1}, P_{(i-1)(p+2) + 2}, \ldots , P_{i(p+2)} \}$ be the frame of $S_i$, where for each $i',j \in [(i-1)(p+2) + 1, i(p+2)]$ and $h \in [d]$, we have that $V(P_h^*) \cap V(P_{i'}) \neq \emptyset$ implies that $V(P_\ell^*) \cap V(P_j) = \emptyset$ for all $\ell \in [h,d]$, which again just means that we want the paths in the frame of $R_i$ to be indexed from left to right.
We analogously want that for each $i \in [r']$ we have that $\mathcal{F}_i^{\mathrm{r}} = \{ Q_{(i-1)(p+2) + 1}, Q_{(i-1)(p+2) + 2}, \ldots , Q_{i(p+2)} \}$ is the frame of $R_i$ with its paths indexed from bottom to top.
Furthermore, for each $i \in [r']$, let $\mathcal{B}_i^{\mathrm{r},T}$ and $\mathcal{B}_i^{\mathrm{r},B}$ respectively be the left and right buffer of $R_i$, where we identify the left buffer of $R_i$ here with the \textsl{top} buffer of $R_i$, meaning higher indices, and identify the right buffer of $R_i$ with the \textsl{bottom} buffer of $R_i$.

Let $M_B$ be the unique $(\nicefrac{n}{2} \times m)$-mesh in
\[ \bigcup_{i = m+1}^{2m} P_i \cup \bigcup_{i=1}^{\nicefrac{n}{2}} Q_i . \]
Since $p = \max(\nicefrac{n}{2}+m, 2m)$, we in particular know that the $m$ paths in $\mathcal{B}_1^{\mathrm{r},B}$ with the highest indices are all disjoint from $M_B$.
By our choice, it is clear that each of the $m$ vertical paths of $M_B$ is a subpath of a different path of $P_{m+1}, \ldots , P_{2m}$ (see \zcref{fig:cornermesh} for an illustration of this). 

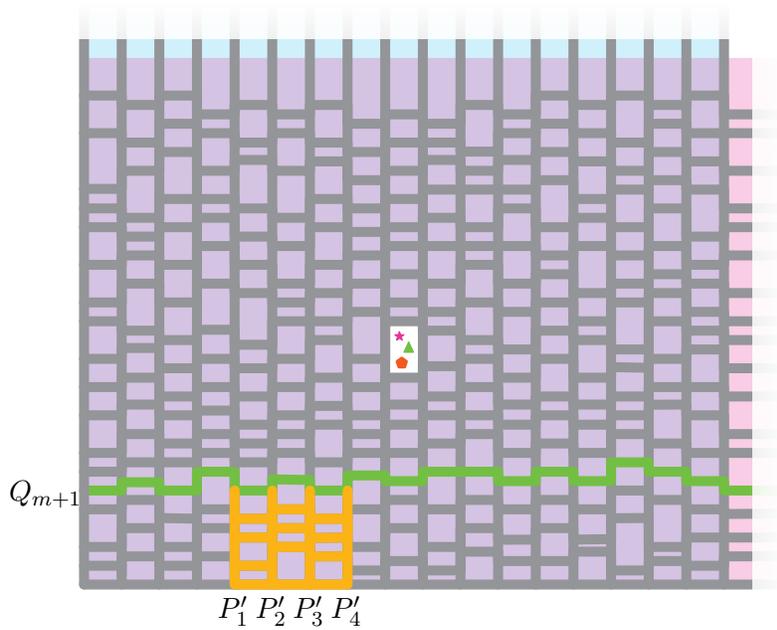
\begin{figure}[ht]
    \centering
    \begin{tikzpicture}

        \pgfdeclarelayer{background}
		\pgfdeclarelayer{foreground}
			
		\pgfsetlayers{background,main,foreground}

        \begin{pgfonlayer}{background}
            \pgftext{\includegraphics[width=10cm]{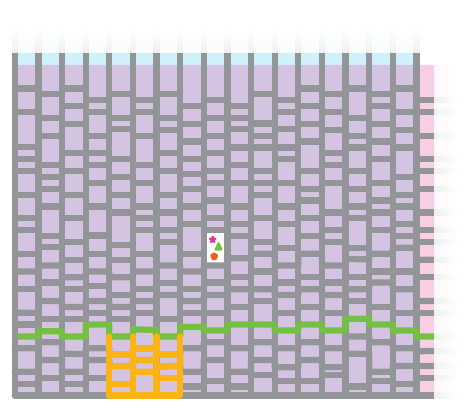}} at (C.center);
        \end{pgfonlayer}{background}
			
        \begin{pgfonlayer}{main}
        
        \node (C) [v:ghost] {};

        \node (Q) at (-5.2,-2.8) [draw=none] {$Q_{m+1}$};
        \node (P1) at (-2.7,-4.35) [draw=none] {$P_1'$};
        \node (P2) at (-2.2,-4.35) [draw=none] {$P_2'$};
        \node (P3) at (-1.7,-4.35) [draw=none] {$P_3'$};
        \node (P4) at (-1.2,-4.35) [draw=none] {$P_4'$};
            
        \end{pgfonlayer}{main}
        
        \begin{pgfonlayer}{foreground}
        \end{pgfonlayer}{foreground}

    \end{tikzpicture}
    \caption{An illustration of the mesh we use to initialise our construction in the bottom left. Here the orange parts constitute the mesh and the vertical paths of this mesh are slightly extended upwards to meet the path $Q_{m+1}$, which is marked in green.}
    \label{fig:cornermesh}
\end{figure}

We let $\mathcal{P}' = \{ P_1', \ldots , P_m' \}$ be $m$ paths that we will continuously update throughout the coming construction by elongating them.
To initialise these paths, we let $P_i'$ be the $V(Q_1)$-$V(Q_{m+1})$-path within $P_{i+m}$ for each $i \in [m]$ (see \zcref{fig:cornermesh}).
By construction $P_1', \ldots , P_m'$ cover all verticals paths of $M_B$.
We refer to the endpoints of the paths in $\mathcal{P}'$ that are found in $V(Q_1)$ as their \emph{fixed endpoints} and from this point onwards during this proof, whenever we refer to endpoints of the paths in $\mathcal{P}'$ we only refer to their non-fixed endpoints, unless otherwise stated.

Our goal is to extend these paths through the infrastructure provided by the buffers we just meticulously described.
As we do this, we will pass by tiles generated by $\mathcal{R}$ and $\mathcal{C}$, which we will collect between two paths with consecutive indices (see \zcref{fig:rainbowmeshconstruction,fig:rainbowmeshconstructiondetail}).
Because of this, we also want to set up $m-1$ sets $I_1, \ldots , I_{m-1}$ into which we will greedily collect colours from the tiles we pick up along the way until we have $I_i = I_\mathcal{C} \cup I_\mathcal{R}$ for all $i \in [m-1]$.
These later let us verify that we constructed a rainbow row.

We now proceed to describe the two slightly different steps our construction process undertakes in alternation until we have arrived in the top right of $\widehat{M}$, where we will attach the top half of our mesh.
The following constructions are stated for arbitrary $i \in [c']$ (or $i \in [c'-1]$ in the case of the second step), but the reader may as well imagine that $i = 1$ to ease comprehension.

\paragraph{Walking up.}
Let $i \in [c']$ and assume that that the endpoint of $P_j'$ is found in $V(Q_{m+1}) \cap V(P_{(i-1)(p+2) + m + j})$ for each $j \in [m]$.
Let $T_1, \ldots T_{2r' - 1} \in \mathsf{Tiles}_{\mathcal{R}}^k(S_i)$ be indexed in the natural way from bottom to top.\footnote{$T_1$ is the $(R_1,p)$-tile of $S_i$, $T_2$ is the $(R_1,R_2,p)$-tile of $S_i$, $T_3$ is the $(R_2,p)$-tile of $S_i$, etc.\ }

We now \textsl{walk up} along $S_i$ to extend the paths in $\mathcal{P}'$ consecutively \textsl{collecting} the tiles $T_1, \ldots , T_{2r' - 1}$.
Let $j \in [2r' - 1]$ and assume that all tiles with an index in $[j-1]$ have been collected already.
If $j$ is odd, we take this to mean that the paths in $\mathcal{P}'$ have been extended upwards to have endpoints in $V(Q_{\nicefrac{1}{2}(j-1)(p+2) + m + 1})$. 
Otherwise, if $j$ is even, the paths in $\mathcal{P}'$ must have been extended to have their endpoints in $V(Q_{\nicefrac{1}{2}(j-2)(p+2) + p + m + 3})$.

If we have $I_k \subseteq \chi(T_j)$ for all $k \in [m-1]$, then for all $h \in [m]$ we extend $P_h'$ along $P_{(i-1)(p+2) + m + h}$ to move their endpoint to lie in $V(Q_{\nicefrac{1}{2}(j-1)(p+2) + p + m + 3})$, if $j$ is odd, and to lie in $V(Q_{\nicefrac{j}{2}(p+2) + m + 1})$, if $j$ is even.
This ensures that in our assumptions on the endpoints of the paths in $\mathcal{P}'$ above we did not make a mistake, if this case occurs.

So we may instead assume that there exists a $k \in [m-1]$ such that $\chi(T_j) \setminus I_k \neq \emptyset$.
We pick such a $k$ arbitrarily and update $I_k$ to be $I_k \cup \chi(T_j)$.
For each $h \in [k]$, we then extend $P_h'$ along $P_{(i-1)(p+2) + m + h}$ as above such that their endpoint lies in $V(Q_{\nicefrac{1}{2}(j-1)(p+2) + p + m + 3})$, if $j$ is odd, and lies in $V(Q_{\nicefrac{j}{2}(p+2) + m + 1})$, if $j$ is even.
However for each $h \in [k+1,m]$, which notably is at least one index, we distinguish two very similar cases.
See \zcref{fig:rainbowmeshconstruction,fig:rainbowmeshconstructiondetail} for illustrations.

\begin{figure}[ht]
    \centering
    \begin{tikzpicture}

        \pgfdeclarelayer{background}
		\pgfdeclarelayer{foreground}
			
		\pgfsetlayers{background,main,foreground}

        \begin{pgfonlayer}{background}
            \pgftext{\includegraphics[width=\textwidth]{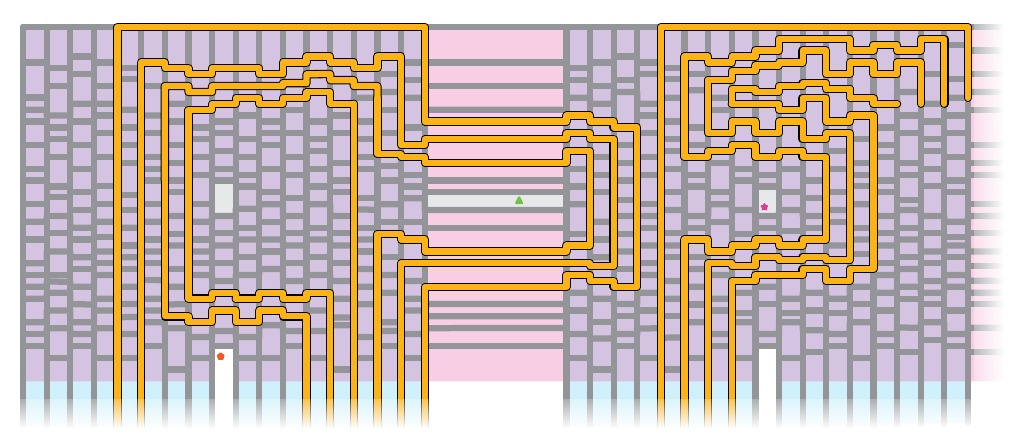}} at (C.center);
        \end{pgfonlayer}{background}
			
        \begin{pgfonlayer}{main}
        \node (C) [v:ghost] {};
        \end{pgfonlayer}{main}
        
        \begin{pgfonlayer}{foreground}
        \end{pgfonlayer}{foreground}

    \end{tikzpicture}
    \caption{A close-up view of the top left part of the illustration in \zcref{lem:rainbowrowfinder}. Here, we can see three situations that cover most of the interesting possibilities that happen in our construction. On the left, a tile is circumnavigated during the walking-up stage, since it does not contain any colours. In the middle, a tile is captured by the rightmost face of the mesh we are constructing during the walking-down stage. Finally, on the right, we capture another tile whilst walking up, this time in the leftmost face of the mesh we are constructing.}
    \label{fig:rainbowmeshconstructiondetail}
\end{figure}

If $j$ is odd, we extend $P_h'$ first along $P_{(i-1)(p+2) + m + h}$ to meet $Q_{\nicefrac{1}{2}(j-1)(p+2) + m + 1 + (m - h)}$, then follow this path until we meet $P_{(i-1)(p+2) + p + 2 + h}$ (which is contained in the right buffer of $S_i$).
From there we follow $P_{(i-1)(p+2) + p + 2 + h}$ until we meet $Q_{\nicefrac{1}{2}(j-1)(p+2) + p + 2 + h}$ (which is contained in the top buffer of $R_{\nicefrac{1}{2}(j+1)}$), and then rejoin $P_{(i-1)(p+2) + m + h}$ to finally fix our new endpoint for $P_h'$ to be the first vertex of $V(Q_{\nicefrac{1}{2}(j-1)(p+2) + p + m + 3})$ we meet.

Otherwise, if $j$ is even, we proceed similarly, extending $P_h'$ by starting first along $P_{(i-1)(p+2) + m + h}$, switching onto $Q_{\nicefrac{1}{2}(j-2)(p+2) + p + m + 3 + (m - h)}$ (which lies in the top buffer of $R_{\nicefrac{j}{2}}$), going from there to meet $P_{(i-1)(p+2) + p + 2 + h}$, until we meet $Q_{\nicefrac{j}{2}(p+2) + h}$ (which is found in the bottom buffer of $R_{\nicefrac{j}{2}+1}$), which allows us to finally rejoin $P_{(i-1)(p+2) + m + h}$ and choose our new endpoint for $P_h'$ as the first vertex of $V(Q_{\nicefrac{j}{2}(p+2) + m + 1})$ we meet.

In the end, we will have extended $\mathcal{P}'$ all the way up such that their endpoints all lie in $V(Q_{(r'-1)(p+2) + p + m + 3})$.
This prepares us neatly for the next step of our construction, which we will alternate with the second step to construct the mesh we ultimately desire.

\paragraph{Walking down.}
Let $i \in [c'-1]$ and assume that the endpoint of $P_j'$ is found in $V(Q_{(r'-1)(p+2) + p + m + 3}) \cap V(P_{(i-1)(p+2) + m + h})$ for each $j \in [m]$.
The second step of our construction is extremely similar to the first in spirit and lands us back in a position to repeat the first step.
Our goal here is to collect the private tiles belonging to our rows that are awkwardly stuck between two columns.

We let $T_1, \ldots , T_{r'}$ be the collection of $(S_i,S_{i+1},p)$ tiles for the elements of $\mathcal{R}$, such that they are indexed in the natural way from the bottom to the top, with $T_j$ being found in $R_j$ for each $j \in [r']$.

Our first task will be to turn around the paths in $\mathcal{P}'$ to face downwards.
Thus, for each $j \in [m]$ we start by walking along $P_{(i-1)(p+2) + m + j}$ and then extend $P_j'$ along $Q_{(r'-1)(p+2) + p + m + 3 + (m-h)}$ (which is found in the top buffer of $R_{r'}$) until we reach $P_{(i-1)(p+2) + p + 3 + m + (m-h)}$ (which is found in the right buffer of $S_i$) and from there we go back down and make the first vertex in $V(Q_{(r'-1)(p+2) + p + m + 2})$ we see our new endpoint for $P_j'$.

We now walk \textsl{down} along $S_i$ and in some sense also $S_{i+1}$ to collect $T_{r'}, \ldots , T_1$ in the given order.
Since there is only one type of tile in this step, our job is a little easier.
Let $j \in [2,r']$ and assume all tiles with an index \textsl{higher} than $j$ have already been collected.
We take this to mean that the paths in $\mathcal{P}'$ have their endpoints in $V(Q_{(j-1)(p+2) + p + m + 2})$.

If $I_k \subseteq \chi(T_j)$ for all $k \in [m-1]$, then for all $h \in [m]$ we extend $P_h'$ along $P_{(i-1)(p+2) + p + 3 + m + (m-h)}$ until we reach $V(Q_{(j-2)(p+2) + p + m + 2})$.
Otherwise, we let $k \in [m-1]$ be chosen arbitrarily such that $\chi(T_j) \setminus I_k$ and update $I_k$ to be $I_k \cup \chi(T_j)$.
For each $h \in [k+1,m]$, we also extend $P_h'$ along $P_{(i-1)(p+2) + p + 3 + m + (m-h)}$ until we reach $V(Q_{(j-2)(p+2) + p + m + 2})$.

For the remaining $h \in [k]$, we elongate $P_h'$ first along $P_{(i-1)(p+2) + p + 3 + m + (m-h)}$ until we reach $Q_{(j-1)(p+2) + p + 3 + (m-h)}$ (found in the top buffer of $R_j$).
Here we switch to the path $P_{i(p+2) + 1 + (m-h)}$ (which is found in the left buffer of $S_{i+1}$), which we take until we see $Q_{(j-1)(p+2) + m + 1 + (m-h)}$ (which is found in the bottom buffer of $R_j$).
From there we again follow $P_{(i-1)(p+2) + p + 3 + m + (m-h)}$ until we reach $V(Q_{(j-2)(p+2) + p + m + 2})$ to establish the new endpoint of $P_h'$.
See \zcref{fig:rainbowmeshconstruction,fig:rainbowmeshconstructiondetail} for illustrations.

The case in which $j=1$ and we want to collect $T_1$ proceeds analogously, but stops each path in $\mathcal{P}'$ at $V(Q_m)$ instead of $V(Q_{(j-2)(p+2) + p + m + 2})$, which would be ill-defined here.
Here, we want to turn around our paths once more to reach the initial conditions necessary to perform the walking-up part of our construction again.
Thus, once we have collected all tiles and all of the endpoints of the paths in $\mathcal{P}'$ are found in $V(Q_m)$, we extend the paths in $\mathcal{P}'$ one more time as follows.

For each $h \in [m]$, we move $P_h'$ further along $P_{(i-1)(p+2) + p + 3 + m + (m-h)}$ until we meet $Q_{1 + (m-h)}$, which allows us to move to $P_{i(p+2) + m + h}$.
From there we follow the path until we see $V(Q_{m+1})$, which is the home of the new endpoint of $P_h'$.
This concludes the instructions for the walking-down step of our construction.

\paragraph{Attaching the top of the mesh.}
At this point, we may describe the construction of most of our mesh simply as taking $M_B$, initialising $\mathcal{P}'$, then walking up and down in an alternating fashion until we have walked up the column $S_{c'}$.
Thus, for each $i \in [m]$, we now know that the endpoint of $P_i'$ is found in $V(Q_{(r'-1)(p+2) + p + m + 3}) \cap V(P_{(c'-1)(p+2) + m + i})$.
We may now take the unique $(\nicefrac{n}{2} \times m)$-mesh $M_T$ in
\[ \bigcup_{i = (c'-1)(p+2) + m + 1}^{(c'-1)(p+2) + 2m} P_i \cup \bigcup_{i=(r'-1)(p+2) + p + m + 3}^{(r'-1)(p+2) + p + m + \nicefrac{n}{2} + 2} Q_i \]
and note that the endpoints of $\mathcal{P}'$ are all found in the horizontal path of $M_T$ that is a subpath of $Q_{(r'-1)(p+2) + p + m + 3}$.
It is therefore easy to observe that $M' = M_B \cup M_T \cup \bigcup_{i=1}^m P_i'$ is an $(n \times m)$-mesh in which the vertical paths $P_1'', \ldots , P_m''$ of $M'$ are such that for each $i \in [m]$ we have $P_i' \subseteq P_i''$ and the horizontal paths $Q_1'', \ldots , Q_n''$ of $M'$ are similarly contained in $Q_1, \ldots , Q_{\nicefrac{n}{2}}, Q_{(r'-1)(p+2) + p + m + 3}, \ldots , Q_{(r'-1)(p+2) + p + m + \nicefrac{n}{2} + 2}$.
Furthermore, the compass of $M'$ with respect to $\rho$ is clearly contained in $\bigcup_{i=1}^{c'} S_i \cup \bigcup_{i=1}^{r'} R_i$ by construction and thus we have $\chi(\mathsf{compass}(M')) \subseteq I_\mathcal{C} \cup \mathcal{R}$ according to the properties which \zcref{lemma:Tiles} guarantees for us.

\paragraph{Tasting the rainbow.}
Next, let us establish the appropriate boundaries for the middle row of $M'$.
For each $i \in [m-1]$, we let $C_i$ be the unique cycle $C_i$ in $P_i'' \cup P_{i+1}'' \cup Q_{\nicefrac{n}{2}}'' \cup Q_{\nicefrac{n}{2}+1}''$.
According to our construction, we observe that $I_i \subseteq \chi(\mathsf{int}(C_i))$ for all $i \in [m-1]$.
Thus we must confirm that $I_i = I_\mathcal{C} \cup I_\mathcal{R}$ to establish that $M'$ does indeed have a rainbow middle row.

Suppose that this is not true for some $i \in [m-1]$ and let $j \in (I_\mathcal{C} \cup I_\mathcal{R}) \setminus I_i$.
By our construction, we only ever collect a tile into the interior of a cycle from $C_1, \ldots , C_{m-1}$ if it contains a colour we have not yet collected in the particular set associated with that cycle.
At worst $|I_\mathcal{C} \cup I_\mathcal{R}| = q$ and thus our process swallows up at most $q(m-1)$ tiles if it really had to collect one tile per cycle per colour.
If we ignore our deficient cycle $C_i$, the other cycles can thus have at worst gathered $q(m-2)$ tiles amongst themselves.
However, according to what \zcref{lemma:Tiles} guarantees for us, each colour is present in at least $q(m-2)+1$ pairwise distinct tiles and thus we must have collected $j$ into $I_i$, a contradiction.
We therefore have actually built an $(n \times m)$-mesh with a rainbow middle row.

\paragraph{Making sure the tangles agree.}
Recall that $\widehat{M}$ is a subwall of $M$ and let $n' = \min(n,m)$.
Suppose that $\mathcal{T}_{M'}$ is not a truncation of $\mathcal{T}_M$.
Thus there exists a separation $(A,B)$ of order less than $n'$ in $\mathcal{T}_{M'}$ such that $(B,A) \in \mathcal{T}_M$.
Note that $B$ contains a horizontal path $Q''$ and a vertical path $P''$ of $M'$, since $(A,B) \in \mathcal{T}_{M'}$.
Furthermore, $A$ entirely contains a horizontal path $Q$ of $M$ and a vertical path $P$ of $M$, since $(B,A) \in \mathcal{T}_M$.
Clearly, by construction of $M'$, $V(Q'' \cup P'')$ intersects at least $n'$ pairwise distinct horizontal and vertical paths of $M$.
Thus according to \zcref{obs:meshconnectivity}, there exist $n'$ disjoint paths between $V(Q'' \cup P'')$ and $V(Q \cup P)$, contradicting the fact that $(A,B)$ is a separation of order less than $n'$.
This concludes our proof.
\end{proof}

\section{Proof of \texorpdfstring{\zcref{thm:IntroHomoMuralis}}{Theorem 1.1}}\label{sec:ProofThm1}
We are now almost ready to prove our main theorem.
First, we must take the second step from a mesh with a rainbow middle row to a homogeneous mesh.
This turns out to not be too difficult.

\begin{lemma}\label{lemma:homogeneousfromrainbowrow}
Let $n \geq 1$ and $q$ be non-negative integers.
Let $(G,\chi)$ be a $q$-colorful graph and let $M$ be a $(2n \times (n^2 - n))$-mesh in $G$ such that $M$ is flat witnessed by the rendition $\rho$.

Then there exists a homogeneous $n$-mesh $M_0 \subseteq M$ that is flat witnessed by $\rho$ and whose the tangle is a truncation of the tangle of $M$.

Moreover, there exists an algorithm that takes $n$, $(G,\chi)$, $M$, and $\rho$ as input and finds $M_0$ in time $\mathbf{poly}(d)|E(G)|$.
\end{lemma}
\begin{proof}
Our proof for this proceeds in much the same way as the proof of \zcref{lem:rainbowrowfinder}.
In lieu of a detailed explanation in the beginning, we point to \zcref{fig:homomeshfromrainbow}, which shows off our relatively simple construction for a small example.

\begin{figure}[ht]
    \centering
    \begin{tikzpicture}

        \pgfdeclarelayer{background}
		\pgfdeclarelayer{foreground}
			
		\pgfsetlayers{background,main,foreground}

        \begin{pgfonlayer}{background}
            \pgftext{\includegraphics[width=8cm]{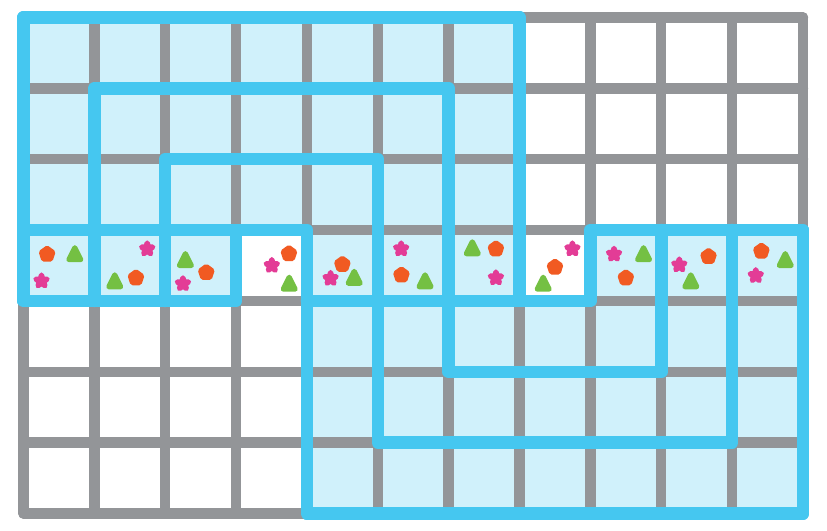}} at (C.center);
        \end{pgfonlayer}{background}
			
        \begin{pgfonlayer}{main}
        \node (C) [v:ghost] {};
        \end{pgfonlayer}{main}
        
        \begin{pgfonlayer}{foreground}
        \end{pgfonlayer}{foreground}

    \end{tikzpicture}
    \caption{An example for the construction presented in the proof of \zcref{lemma:homogeneousfromrainbowrow} for the case in which $n=4$. The solid horizontal lines in blue represent the horizontal paths in the finished homogeneous mesh. For the sake of simplicity we represent this construction with a $(8 \times 12)$-grid, but this does not affect the details of the construction.}
    \label{fig:homomeshfromrainbow}
\end{figure}

Let $n' = n^2 - n$, let $P_1, \ldots , P_{n'}$ be the vertical paths of $M$ (where we consider $P_1$ to be the leftmost path), and let $Q_1, \ldots , Q_{2n}$ be the horizontal paths of $M$ (where we consider $Q_1$ to be the topmost path).
We start by letting $M'$ be the unique $(2 \times n)$-mesh in $\bigcup_{i=1}^n P_1 \cup Q_n \cup Q_{n+1}$.
Note that $M'$ is clearly homogeneous, since it consists entirely of a subset of the rainbow row of $M$.
Let $P_1', \ldots , P_n'$ be the vertical paths in $M'$.
We will now iteratively update these paths and $M'$ in much the same way as we updated the paths in the proof of \zcref{lem:rainbowrowfinder}.
Similar to the situation there, we consider the endpoints of $\mathcal{P}' = \{ P_1, \ldots , P_n' \}$ in $Q_{n+1}$ to be \emph{fixed} and will not change them during our construction.
Thus, for the remainder of this proof, whenever we mention an endpoint of a path in $\mathcal{P}'$ we mean its unique non-fixed endpoint.

We divide our construction into two analogous steps.
In the first, we construct a row of odd index and in the second we construct a row of even index within $M'$.

\paragraph{Folding over.}
Let $i \in [3,n]$ be odd.
We assume that for each $j \in [n]$ the endpoint of $P_j'$ is currently found in the intersection of $V(Q_n)$ and $V(P_{(i-3)n + j})$, and that the mesh $M'$ we have constructed so far is homogeneous.

For each $j \in [n]$, we now extend $P_j'$ as follows.
First, we follow along $P_{(i-3)n + j}$ until we meet $Q_j$, which then allows us to move straight to $P_{(i-2)n + 1 + (n-j)}$ and down to $Q_{n+1}$, where we find our new endpoint.
This then allows us to find a new row for $M'$ in $Q_{n+1}$ and since this captures a cycle of the rainbow row of $M$ between any pair of paths in $\mathcal{P}'$ with consecutive indices, our slightly grown $M'$ remains homogeneous.

\paragraph{Folding under.}
Let $i \in [3,n]$ be even.
We assume that for each $j \in [n]$ the endpoint of $P_j'$ is currently found in the intersection of $V(Q_{n+1})$ and $V(P_{(i-3)n + 1 + (n - j)})$, and that the mesh $M'$ we have constructed so far is homogeneous.

For each $j \in [n]$, we now extend $P_j'$ as follows.
First, we follow along $P_{(i-3)n + 1 + (n-j)}$ until we meet $Q_{n+j}$.
This allows us to move straight to $P_{(i-2)n + j}$ and up to $Q_n$.
The new row for $M'$ is found in $Q_n$ and this preserves homogeneity for $M'$ for the same reasons as mentioned above.

\paragraph{Analysing our construction.}
We now construct our desired $M_0$ by first starting with the $(2 \times n)$-mesh mentioned at the start and then iteratively folding over, then under, and repeating this until we have extended our mesh to have square dimensions.
Note that each iteration of our process eats up $n$ faces of the rainbow grid, since we lose one face each time as a spacer.
Thus, we start with $n$ faces to initialise and perform $n-2$ iterations eating up $n$ faces each.
This means we need $n-1 + (n-2)n = n^2 - (n + 1)$ rows, which necessitates $M$ initially having width $n^2 - n$.
It is also easy to see that we need $2n$ rows in $M$ to perform our construction.
This justifies the dimensions we chose for $M$.

The argument that ensures that the tangle of $M_0$ is a truncation of the tangle of $M$ is now entirely analogous to what we presented at the end of the proof of \zcref{lem:rainbowrowfinder}.
This concludes our proof.
\end{proof}

From this we can derive the following corollary using \zcref{lem:rainbowrowfinder}, where we first define the function
\begin{align*}
    \mathsf{d}_{\ref{cor:almostmainthm}}(\ell,q) \coloneqq  \mathsf{d}_{\ref{lem:rainbowrowfinder}}(2\ell, \ell^2 - \ell, q) .
\end{align*}
We note that the following is essentially already as close to \zcref{thm:IntroHomoMuralis} as one really needs to be to apply our result in practice.
The remaining gap is closed right after this result.

\begin{corollary}\label{cor:almostmainthm}
Let $d$, and $\ell$ be positive integers and let $q$ be a non-negative integer.
Let $(G,\chi)$ be a $q$-colorful graph, let $d \geq \mathsf{d}_{\ref{cor:almostmainthm}}(\ell,q)$, and let $M$ be a $(d \times d)$-mesh in $G$ such that $M$ is flat witnessed by the rendition $\rho$.

Then there exists a homogeneous $\ell$-mesh $M_0 \subseteq M$ that is flat witnessed by $\rho$ and whose tangle is a truncation of the tangle of $M$.

Moreover, there exists an algorithm that takes $d$, $\ell$, $(G,\chi)$, $M$, and $\rho$ as input and finds $M_0$ in time $\mathbf{poly}(d)|E(G)|$.
\end{corollary}
\begin{proof}
    We first apply \zcref{lem:rainbowrowfinder} to our objects, with $n = 2\ell$ and $m = \ell^2 - \ell$.
    Since $d \geq \mathsf{d}_{\ref{cor:almostmainthm}}(\ell,q)$, this yields a $(2\ell \times \ell^2 - \ell)$-mesh $M' \subseteq M$, which is flat witnessed by $\rho$, has a rainbow middle row, and whose tangle is a truncation of the tangle of $M$.
    To this mesh, we then apply \zcref{lemma:homogeneousfromrainbowrow} to find a homogeneous $\ell$-mesh that is flat witnessed by $\rho$ and whose the tangle is a truncation of the tangle of $M'$.
    By transitivity, the tangle of $M_0$ is thus also a truncation of the tangle of $M$.
\end{proof}

To actually derive \zcref{thm:IntroHomoMuralis} from \zcref{cor:almostmainthm}, we still need to translate from a mesh to a wall.
Thus, we provide another small proof to bridge this gap.

\begin{proof}[Proof of \zcref{thm:IntroHomoMuralis}.]
    To start with, we choose the function $f \colon \mathbb{N} \to \mathbb{N}$ as
    \[ f(q,k) \coloneqq \mathsf{d}_{\ref{cor:almostmainthm}}(2k,q) . \]
    Let $n \coloneqq f(q_k)$.
    Recall that any $n$-wall is an $n$-mesh, and thus we can now directly apply \zcref{cor:almostmainthm}, with $\ell = 2k$, and this yields a homogeneous $\ell$-mesh $M \subseteq W_0$ that is flat witnessed by $\rho$ and whose tangle is a truncation of the tangle of $W_0$.

    Let $P_1, \ldots , P_\ell$ be the vertical paths of $M$, and $Q_1, \ldots , Q_\ell$ be the horizontal paths of $M$.
    For each odd $i \in [\ell]$, we let $P_i^* \subseteq P_i$ be the union of all $V(Q_j)$-$V(Q_{j+1})$-paths for odd $j \in [k-1]$, and for each even $i \in [\ell]$, let $P_i^* \subseteq P_i$ be the union of all $V(Q_j)$-$V(Q_{j+1})$-paths for even $j \in [k-1]$.
    Note that for all odd $i \in [\ell - 1]$, the graph $\bigcup_{j=1}^{k} Q_j \cup P_i^* \cup P_{i+1}^*$ has maximum degree 3 and contains a unique $V(Q_1)$-$V(Q_k)$-path $P_i^\dagger$.
    Using this, we observe that the graph
    \[ \bigcup_{j=1}^{k'} Q_j \cup \bigcup_{\substack{i \in [\ell-1] \\ i \text{ is odd}}} P_i^\dagger \]
    contains a $k$-wall $W$.
    In particular, $W \subseteq M$ and thus every cycle in $W$ is a cycle in $M$.
    Thus, all properties we desire are transferred directly to $W$ from $M$ and thus $W$ is homogeneous, flat witnessed by $\rho$, and its tangle is a truncation of the tangle of $W_0$.
\end{proof}

\paragraph{Computing our bounds.}
We give an explicit estimate for the order of the function in \zcref{thm:IntroHomoMuralis}.
To justify this, we now briefly go backwards through the steps that the above proof takes to give a justification for this bound.
First, we call $\mathsf{d}_{\ref{cor:almostmainthm}}(2k, q)$, which then unpacks to $\mathsf{d}_{\ref{lem:rainbowrowfinder}}(4k, 2k(2k-1), q)$.
If we further follow these definitions, this lands us in $\mathsf{d}_{\ref{lemma:Tiles}}(q, 2(2k(2k-1)), 2 , q(2k(2k-1)-2)+1 )$.
This leads us to the definition of the function $\mathsf{d}_{\ref{lemma:Tiles}}$ which lets us translate this expression to
\[ (q(\mathsf{r}(q,q(2k(2k-1)-2)+1)+1) + 2q +1 ) (8k(q+1)(2k-1) + 2) . \]
Let us also resolve the definition of $\mathsf{r}$ to translate $\mathsf{r}(q,q(6k(6k-1)-2)+1)+1)$ to get 
\[ \mathsf{r}(q,q(2k(2k-1)-2)+1) = (q(2k(2k-1)-2)+1)q(2k(2k-1)-2) + (q+1)(q(2k(2k-1)-2)+1) . \]
Since we now have a concrete expression, it is a simple matter of performing the necessary calculations (or letting a computer perform the necessary calculations) to see that for our chosen $f$ in the proof of \zcref{thm:IntroHomoMuralis} we have
\begin{align*}
   f(q,k)=  & 256 k^6 q^4 + 256 k^6 q^3 - 384 k^5 q^4 - 384 k^5 q^3 + 160 k^4 q^3 \ + \\
            & 128 k^4 q^2 + 160 k^3 q^4 - 128 k^3 q^2 - 16 k^2 q^4 - 48 k^2 q^3 + 64 k^2 q^2 \ + \\
            & 80 k^2 q + 16 k^2 - 16 k q^4 + 20 k q^3 - 16 k q^2 - 40 k q - 8 k + 4 q^3 - 6 q^2 + 8 q + 2 .
\end{align*}
Clearly, the term $256 k^6 q^4$ in this expression dominates the others, justifying the estimate we state in \zcref{thm:IntroHomoMuralis}.

\section{Conclusion}\label{sec:Conclusion}
As mentioned in the introduction, in recent years a program for establishing close-to-optimal bounds for the results originating from Robertson and Seymour's Graph Minor Series has emerged.
Several important milestones in this program have already been reached including polynomial bounds for the Grid Theorem \cite{ChuzhoyT2021Tighter}, the Flat Wall Theorem \cite{Chuzhoy2015Improved,KawarabayashiTW2018New,GorskySW2025Polynomial}, and the Graph Minor Structure Theorem (GMST) \cite{GorskySW2025Polynomial}.
However, several challenges remain.
One such challenge is the establishment of tighter bounds for the Unique Linkage Function, as mentioned in the introduction.
Towards this goal, our main theorem represents an important step, but to further close in on this target, additional and new arguments are required.

We want to mention that we believe the bounds we give for the function $f$ in \zcref{thm:IntroHomoMuralis} are likely not optimal.
In particular, the specific constructions we are choosing in \zcref{sec:Rainbow,sec:ProofThm1} to find the homogeneous wall in the output of \zcref{lemma:Tiles} can almost certainly be optimised.
It should be possible to find the homogeneous wall directly starting from \zcref{lemma:Tiles} without having to find a rainbow middle row.
This should reduce the exponent for $k$ in the estimate of $f$ in \zcref{thm:IntroHomoMuralis} somewhat.
If \zcref{sec:Rainbow} is any indication, we expect this to be quite arduous, technical work.
Accordingly we chose a construction which we considered to be comparatively easier to write down.
We expect that optimising the estimate for the exponent for $q$ is both harder and of more interest to the broader community, since in the known applications $q$ is often the much larger quantity.

Finally, we wish to direct attention at another problem deeply rooted in the Graph Minors Series, but maybe less visible outside a small group of experts.
Roughly speaking, the GMST explains that every $H$-minor-free graph has a tree-decomposition of small adhesion into pieces which are ``almost embeddable'' into some surface where $H$ does not embed.
The notion of almost embeddability has three ingredients apart from the Euler-genus of the surfaces involved.
Those ingredients are
\begin{itemize}
    \item apex vertices,
    \item vortices, and
    \item the depth of vortices.
\end{itemize}
In short, apex vertices are an extension of the set $A$ of vertices which needs to be deleted in order to find the flat wall in the Flat Wall Theorem.
Vortices are a slightly different matter, they represent imperfections of an embedding of a graph in some surface.
Roughly speaking, a vortex of an embedding is a disk $\Delta$ such that the graph on the interior of $\Delta$ is not required to be properly embedded.
To control the structure of vortices, the concept of \textsl{depth} provides a measure for how these non-embedded parts attach to the embedded outside of the vortex.
Both apices and vortices are known to be unavoidable parts of the GMST, which states that the number of apex vertices, the number of vortices, as well as their depth may all be bounded by a function only depending on $|H|$.

The \emph{apex number} of a graph $H$, denoted by $\mathsf{apex}(H)$, is the smallest integer $k$ such that there exists set $S\subseteq V(H)$ of size at most $k$ for which $H-k$ is planar.
Notice that $\mathsf{apex}(H) \leq |H|-4$.
Given an almost embedding, we say that an apex vertex $a$ is a \emph{major apex} if it has neighbours outside of the vortices of the almost embedding.
Otherwise $a$ is called a \emph{minor apex}.
It was well-known in the graph minors community for a long time that it is possible to prove a variant of the GMST where the number of major apices is bounded by $|H|-5$ which is optimal whenever $H$ is a complete graph on at least five vertices.
A proof of this theorem was first written down by Dvorak and Thomas \cite{DvorakT2016Listcoloring} and this strengthening of the GMST has since found a wide range of applications.
For instance, bounding the number of major apices played a key role in the recent proof of the \textsl{Clustered Hadwiger Conjecture} by Dujmovi{\'c}, Esperet, Morin, and Wood \cite{DujmovicEMW2023Proof}.
Moreover, the strengthening unlocked a structure theorem for excluding so called \emph{apex graphs}, i.\@e.\@ graphs $H$ with $\mathsf{apex}(H)\leq 1$ \cite{Grohe2003Local,KorhonenNPS2024Fully,HendreyW2025Polynomial}.
It is therefore of interest to prove a polynomial variant of this stronger GMST.

In order to establish such a polynomial Apex Graph Minor Structure Theorem (AGMST) one would need to deal with the apex vertices potentially arising in every step of the inductive proof of the GMST by separating them into those that are avoidable, i.\@e.\@ those which attach only in few areas, and those which are not avoidable, i.\@e.\@ those on whose neighbours one may root a large grid minor.
This separation step requires homogenisation similar to the one explained in detail in the introduction.
Indeed, \zcref{thm:IntroHomoMuralis} is already strong enough to deal with this task for the very base of the inductive proof of the GMST.
The induction step, however, requires a more sophisticated and specialised homogeneity lemma.
The way the GMST is proven is by first finding a flat wall and then starting to attach a sequences of ``flat strips'' to this initial flat area.
These strips are known as \emph{transactions} (see \cite{RobertsonS1990Graph,KawarabayashiTW2021Quickly,GorskySW2025Polynomial}).
In order to homogenise a flat transaction, the methods presented in this paper provide a good starting point, however, they do not suffice as both ends of the strip of a transaction are fixed and may not be pruned in the same way our proof deals with ``end tiles''.
Moreover, as one requires a full and consecutive strip for the proof, leaving ``gaps'' as those found in-between the horizontal and vertical strips of out technique is not feasible.
If such gaps are unavoidable, then one would require at the very least a strategy to deal with them in some way which is not given by our proof of \zcref{thm:IntroHomoMuralis}.
A strategy for homogenising transactions would not only allow to prove a polynomial variant of AGMST, but it would also allow for polynomial bounds on a recently established colorful variant of the GMST by Protopapas, Thilikos and Wiederrecht \cite{ProtopapasTW2025Colorful}.
Indeed, such a theorem might even be strong enough to prove polynomial bounds on the gap function in a recent variant of Liu's \cite{Liu2022Packing} proof for the half-integral Erdos-Posa property for minors by Paul, Protopapas, Thilikos, and Wiederrecht \cite{PaulPTW2024Obstructionsa}.

We therefore consider our main result, that is \zcref{thm:IntroHomoMuralis}, to be an important starting point towards proving a ``homogeneous flat transaction lemma'' which would be necessary to further the program towards an efficient graph minors theory.

\bibliographystyle{alphaurl}
\bibliography{literature}

\end{document}